\documentclass[a4paper, 11pt]{amsart}
\usepackage[left=3cm,right=3cm]{geometry}
\usepackage[english]{babel}
\usepackage{amsthm}
\usepackage{amssymb}
\usepackage{amsfonts}
\usepackage{amsmath}
\usepackage{amsthm}
\usepackage[all]{xy}
\usepackage{geometry}
\usepackage{ae,aecompl}
\usepackage{eurosym}
\usepackage{verbatim}
\usepackage{mathrsfs}
\usepackage{url}
\usepackage{tikz}
\usepackage{hyperref}

\newcounter{stepcounter}
\newtheoremstyle{smallcaps}
    {3pt}                    
    {3pt}                    
    {\itshape}                   
    {}                           
    {\sc}                   
    {.}                          
    {.5em}                       
    {}  
    
\newtheoremstyle{smallcapsdef}
    {3pt}                    
    {3pt}                    
    {}                   
    {}                           
    {\sc}                   
    {.}                          
    {.5em}                       
    {}  

\theoremstyle{plain}

\newtheorem{thm}{Theorem}[section]

\newtheorem{lem}[thm]{Lemma}
\newtheorem{prop}[thm]{Proposition}
\newtheorem{cor}[thm]{Corollary}
\newtheorem{conj}[thm]{Conjecture}

\theoremstyle{definition}

\newtheorem{eg}[thm]{Example}
\newtheorem{defn}[thm]{Definition}
\newtheorem{rem}[thm]{Remark}
\newtheorem{remark}[thm]{Remark}

\newtheorem{notation}[thm]{Notation}

\date{}

\def\bal#1\eal{\begin{align}#1\end{align}}
\newcommand\bas{\begin{align*}}
\newcommand\bit{\begin{itemize}}
\newcommand\eit{\end{itemize}}
\newcommand\bet{\begin{enumerate}}
\newcommand\eet{\end{enumerate}}
\newcommand\ed{\end{document}}

\renewcommand{\a}{\alpha}
\renewcommand{\d}{\delta}
\newcommand{\e}{\varepsilon}
\newcommand{\f}{\varphi}
\renewcommand{\k}{\kappa}

\renewcommand\r{\rho}
\newcommand\s{\sigma}

\newcommand\w{\omega}

\newcommand\Om{\Omega}
\newcommand\z{\zeta}
\newcommand\del{\partial}
\newcommand\adel{\ol{\partial}}
\newcommand\DEL{\Delta}
\newcommand\G{\Gamma}

\newcommand\bC{{\mathbb C}}
\newcommand\bN{{\mathbb N}}

\newcommand\bR{{\mathbb R}}
\newcommand\bZ{{\mathbb Z}}

\newcommand\bP{\mathbb{P}}

\newcommand\A{{\mathcal{A}}}
\newcommand\F{{\mathcal F}}
\renewcommand\H{\mathcal{H}}
\renewcommand{\O}{\mathcal{O}}

\newcommand\co{\mathrm{co}}

\newcommand\exd{\mathrm{d}}
\newcommand\dom{\mathrm{dom}}

\newcommand\haar{\mathrm{\bf h}}

\newcommand\unit{\mathrm{U}}
\newcommand\counit{\mathrm{C}}

\newcommand\hw{\mathrm{hw}}
\newcommand\lw{\mathrm{lw}}
\newcommand\id{\mathrm{id}}
\newcommand\proj{\mathrm{proj}}

\newcommand\spn{\mathrm{span}}

\newcommand\vol{\mathrm{vol}}

\newcommand\hol{^{(1,0)}}
\newcommand\ahol{^{(0,1)}}

\newcommand\inv{^{-1}}

\newcommand\by{\times}
\newcommand\coby{\, \square_{H}}
\newcommand\oby{\otimes}

\newcommand\wed{\wedge}
\newcommand\sseq{\subseteq}

\newcommand\tr{\triangleright}

\newcommand\wh{\widehat}
\newcommand\ol{\overline}

\newcommand\la{\left\langle}
\newcommand\ra{\right\rangle}
\newcommand\bs{\backslash}
\newcommand\mto{\mapsto}

\newcommand\ccpn{\bC \bP^{n-1}}

\newcommand\alg{algebra~}
\newcommand\algn{algebra}
\newcommand\algs{algebras~}

\newcommand\hk{Heckenberger--Kolb~}

\newcommand\nc{noncommutative~}

\newcommand\st{such that~}
\newcommand\stn{such that}

\newcommand\wrt{with respect to~}

\DeclareMathOperator{\dt}{det}

\usepackage{tikz}
\usetikzlibrary{decorations.pathreplacing}

\usepackage{ytableau}

\usepackage[english]{babel}

\address{Laboratory of Advanced Combinatorics and Network Applications,
Department of Applied Mathematics, Moscow Institute of Physics and Technology, Moscow, Russia}
\author{Biswarup Das}
\address{Instytut Matematyczny, Uniwersytet Wroc\l{}awski, pl.Grunwaldzki 2/4, 50-384 Wroc\l{}aw, Poland}
\email{biswarup.das@math.uni.wroc.pl}

\author{R\'eamonn \'O Buachalla}

\address{D\'epartement de Math\'ematiques, Facult\'e des sciences, Universit\'e Libre de Bruxelles, Boulevard du Triomphe, B-1050 Bruxelles, Belgium}
\email{reamonnobuachalla@gmail.com}

\author{Petr Somberg}
\address{Mathematical Institute of Charles University, Sokolovsk\'a 83, Prague, Czech Republic} \email{somberg@karlin.mff.cuni.cz}

\title[{\bf A Dolbeault--Dirac Spectral Triple for $\O_q(\mathbb{CP}^{n-1})$}]{{\bf A Dolbeault--Dirac Spectral Triple for Quantum Projective Space}}

\thanks{The second author acknowledges FNRS support through  a postdoctoral fellowship within the framework of the MIS Grant ``Antipode'' grant number F.4502.18. The second and third authors acknowledge support from the Eduard \v{C}ech Institute within the framework of the grants GACR $P201/12/G028$ and GACR $19-28628X$.
  }

\begin{document}

\maketitle

\begin{abstract}
The notion of a  K\"ahler structure for a differential calculus 
was recently introduced  by the second author as a framework in which to study  the noncommutative  geometry of the quantum flag manifolds. It was subsequently  shown that any covariant positive definite K\"ahler structure has a canonically  associated  triple satisfying, up to the compact resolvent condition, Connes' axioms for a spectral triple. In this paper we begin the development of a robust framework in which to investigate the compact resolvent condition, and  moreover, the general spectral behaviour of covariant K\"ahler structures.
This framework is then applied to quantum projective space endowed with its \hk differential calculus. An even spectral triple with non-trivial associated $K$-homology class is produced, directly $q$-deforming the Dolbeault--Dirac  operator of complex projective space.  Finally, the extension of this approach to a certain canonical class of irreducible quantum flag manifolds is discussed in detail.
\end{abstract}

\tableofcontents

\section{Introduction}

In Connes' $K$-theoretic approach to noncommutative geometry, spectral triples generalise Riemannian spin manifolds and their associated Dirac operators to the noncommutative setting.
The question of how to reconcile the theory of spectral triples with Drinfeld--Jimbo quantum groups is one of the major open problems in noncommutative geometry. Since their appearance in the 1980s, quantum groups have attracted serious and significant  attention at the highest mathematical levels. In the compact case, the foundations of their  noncommutative topology and their noncommutative measure theory have now been firmly established. By contrast, the noncommutative spectral geometry of quantum groups is still very poorly understood.  Indeed, despite a large number of important contributions over the last thirty years, there is still no consensus on how to construct a spectral triple for $\O_q(SU_2)$, probably the most basic example of a quantum group. These difficulties aside, the prospect of reconciling  these two areas still holds great promise for their mutual enrichment. On one hand, it would provide quantum groups with powerful tools from operator algebraic $K$-theory and $K$-homology. On the other hand, it would  provide the theory of spectral triples with a large class of  examples of fundamental importance with which to test and guide the future development of the subject.


One of the most important lessons to emerge from the collected efforts to understand the noncommutative geometry of quantum groups is that quantum homogeneous spaces tend to be more amenable to geometric investigation than quantum groups themselves. Philosophically, one can  think of the  process of forming a quantum homogeneous spaces as quotienting out the most exotic  noncommutativity of the quantum group. This produces quantum spaces which are closer to their classical counterparts, and which possess more recognisable differential structures. 
The prototypical example here is the {\em standard Podle\'s sphere}, the Drinfeld--Jimbo $q$-deformation of the Hopf fibration presentation of the $2$-sphere $S^2$. In contrast to the case of $\O_q(SU_2)$, the Podle\'s sphere admits a canonical spectral triple which directly $q$-deforms the classical Dolbeault--Dirac operator of the $2$-sphere \cite{DSPodles}. Moreover, it is the most widely and consistently accepted example of a spectral triple in the Drinfeld--Jimbo setting.  The Podle\'s sphere also forms a well behaved and motivating example for Majid's Hopf algebraic theory of braided noncommutative geometry \cite{SMLeabhMor}. In particular, 
it admits an essentially unique differential calculus, the Podle\'s calculus,  to which Majid was able to apply his quantum frame bundle theory to directly $q$-deform the classical K\"ahler geometry of the $2$-sphere  (cf. \cite{Maj}).


The Podle\'s sphere is itself a special example of a large and very beautiful family of quantum homogeneous spaces: the quantum flag manifolds \cite{LR91,DijkStok}.  
In one of the outstanding results of the algebraic approach to the noncommutative geometry of quantum groups, Heckenberger and Kolb showed that the quantum flag manifolds of irreducible  type admit an essentially unique $q$-deformation of their classical Dolbeault double complex.  This result forms a far reaching generalisation  of  the Podle\'s calculus endowed with Majid's complex structure, and strongly suggests that the irreducible quantum flag manifolds, or more generally the quantum flag manifolds, have a central role to play in reconciling quantum groups and spectral triples.


The \hk classification, however, contains no generalisation of the K\"ahler geometry of the Podle\'s sphere (cf. \cite{Maj}).  In a recent paper by the second author, the notion of a noncommutative K\"ahler structure was introduced to provide a framework in which to do just this.  In the quantum homogeneous space case many of the fundamental results of  K\"ahler geometry have been shown to follow from the existence of such a structure. For example, it implies noncommutative generalisations of  Lefschetz decomposition, the Lefschetz and K\"ahler
identities, Hodge decomposition,  the hard Lefschetz theorem, and the refinement of de Rham cohomology by Dolbeault cohomology. The existence of a K\"ahler structure was verified for the quantum projective spaces in \cite{MMF3}, and it was conjectured that a K\"ahler structure exists for all the compact quantum Hermitian spaces. Subsequently, for all but a finite number of values of $q$, the conjecture was verified for every irreducible quantum flag manifold by Matassa in \cite{MarcoConj}. 


The Dolbeault--Dirac operator \text{\small $D_{\adel} := \adel + \adel^{\dagger}$} associated to a K\"ahler structure is an obvious candidate for a noncommutative Dirac operator. In \cite{DOS1} the authors 
associated to any covariant positive definite Hermitian structure $(\Om^{(\bullet,\bullet)}, \kappa)$, over a quantum homogeneous space $B=G^{\co(H)}$, a Hilbert space $L^2(\Om^{(0,\bullet)})$, 
 carrying a bounded $*$-representation $\r$ of $B$. Moreover,  $D_{\adel}$  was shown to act on $L^2(\Om^{(0,\bullet)})$ as an essentially self-adjoint operator, with bounded commutators $[D_{\adel},b]$, for all $b \in B$. Hence, to show that the triple $\left(B,L^2(\Om^{(0,\bullet)}), D_{\adel}\right)$ is a spectral triple, it  remains to verify the compact resolvent condition. 
 Even in the Drinfeld--Jimbo case, however,  it is not clear at present how to conclude the compact resolvent condition from the properties of a general covariant K\"ahler structure. (See \textsection \ref{Subsubsection:Conjs} for a brief discussion on how this might be achieved.) In the study of classical homogeneous spaces, a difficult problem can often be approached by assuming restrictions on the multiplicities of the $U(\frak{g})$-modules appearing in an equivariant geometric structure \cite{Kobay}. Taking inspiration from this approach, we choose to focus on covariant complex structures of {\em  weak Gelfand type}, that is to say, those for which $\adel \Om^{(0,k)}$ is multiplicity-free as a $U_q(\frak{g})$-module. This  implies diagonalisability of the Dolbeault--Dirac operator $D_{\adel}$ over irreducible modules, which when combined with the considerable geometric structure of the calculus, allows us to make a number of strong statements about the spectral behaviour of $D_{\adel}$. 
In particular, for those covariant complex structures of  {\em Gelfand type}, that is to say, those for which 
 $\Om^{(0,k)}$ is multiplicity-free, we produce a sufficient set of routinely verifiable conditions for $D_{\adel}$ to have compact resolvent.


To place our efforts in context, we briefly recall previous spectral calculations for quantum groups, and in particular for quantum flag manifolds. The construction of spectral triples over quantum groups can be very roughly divided into two approaches. The first is isospectral deformation, as exemplified by the work of Neshveyev and Tuset, who constructed isospectral  Dirac operators for all the Drinfeld--Jimbo quantum groups \cite{NeshTus}. In this approach one takes as an ansatz that the spectrum of the Dirac survives $q$-deformation unchanged. A representation of the quantum group is then constructed around this ansatz so as to retain bounded commutators. This approach has the advantage of avoiding the need for spectral calculations, but the disadvantage that the spectral triples produced are too close to the classical case to be completely natural. By contrast, the second approach constructs canonical $q$-deformations of the classical spin geometry of a space, and then calculates the spectrum of the resulting $q$-deformed Dirac.  The prototypical examples here are the D\c{a}browski--Sitarz construction of a spectral triple on the Podle\'s sphere, as discussed above, and Majid's spectral calculations for his Dolbeault--Dirac operator over the Podle\'s sphere, as also mentioned above. This is  the approach followed in this paper, and just as for the Podle\'s examples, an unavoidable consequence is a $q$-deformation of the classical spectrum.

At around the same time as these works,  Kr\"ahmer introduced an influential algebraic Dirac operator for the irreducible quantum flag manifolds, which gave a commutator realisation of their \hk calculus \cite{UKFlagDirac}.  A series of papers by D\c{a}browski,  D'Andrea, and Landi, followed, where spectral triples were constructed for the all quantum projective spaces \cite{SISSACP2,SISSACPn}. This approach used a noncommutative generalisation of the Parthasarathy formula \cite{Partha} to calculate the Dirac spectrum and hence verify Connes' axioms. 
Matassa would subsequently reconstruct this spectral triple \cite{MatassaCPn} in a more formal manner by connecting with the work of Kr\"ahmer and Tucker--Simmons \cite{MTSUK}. This approach was  subsequently extended to  the  quantum Lagrangian Grassmannian $\O_q(\mathbf{L}_2)$, a $C$-series irreducible quantum flag manifold \cite{MatassaLG2,MatassaLG22}. We note that in the quantum setting the Parthasarathy relationship with the Casimir is much more involved than in the classical case. This reflects our poor understanding of Casimirs in the Drinfeld--Jimbo setting, and the resulting challenges associated with a quantum Casimir approach to spectral calculations.

One of the primary purposes of a spectral triple is to serve as unbounded representatives for the $K$-homology classes of a $C^*$-algebra. Having an unbounded representative constructed in such a geometric manner has a number of advantages. In particular, it allows us to convert index theory calculations into questions about the Dolbeault cohomology of the complex structure. In \cite[Theorem 5.4]{DOS1} the index of the associated $K$-homology class has been shown to be equal to the anti-holomorphic Euler characteristic of the calculus. This is calculable using Hodge decomposition in general.  In particular, in the noncommutative Fano seting, it follows from the Kodaira vanishing theorem for noncommutative K\"ahler structures that all cohomologies are concentrated in degree $0$ \cite{OSV}. Hence, the index will be non-zero and the associated $K$-homology class non-trivial.  This is particularly important given the well-known difficulty of applying  Connes' local index formula  in the quantum group setting.

This paper forms part of a series of works investigating the noncommutative geometry of the quantum flag manifolds and their connections with Nichols algebras, Schubert calculus, and non-commutative projective algebraic geometry \cite{BS,KMOS,StaffordICM}. It is intended that this paper will  serve as a point of contact between these areas and operator algebraic $K$-theory. Moreover, in its  discussion of order I and order II presentations, the  paper can be regarded as a first step towards a systematic extension of the classical results on Harish-Chandra
modules (cf. \cite{KV}, \cite{V}) to the quantum group setting. The typical structure
here is a pair $(\frak{g},K)$ consisting of a real reductive Lie group $G$ with complexified
Lie algebra $\mathfrak{g}$, and  a compact subgroup $K \subset G$,
for which the differential of $\mathrm{Ad}(K)$ and the restriction
$ad(\mathfrak{g})|_{\mathfrak{k}}$ are compatible. The
representation category ${\mathcal C}(\mathfrak{g},K)$ of
$(\mathfrak{g},K)$-modules, if an infinitesimal character is
specified, is characterized by the existence of finitely irreducible
representations of $K$ termed the collection of {\em minimal $K$-types}
(every irreducible $(\mathfrak{g},K)$-module with an infinitesimal
character contains one of these $K$-types.) As we shall discuss in
the present article, this property conveniently carries over to the
quantum group setting. Moreover, there is an analogous transfer of
the other important related structures such as the Hecke
algebra of the pair $(\frak{g},K)$, dualities, and so on, which will be treated elsewhere.
Finally, an important issue to be addressed is how the refined spectral analysis of \cite{RennieSenior} extends to the general quantum projective space setting, and if a local index formula in twisted cyclic cohomology can be produced.

The paper is organised as follows: In \textsection 2 we recall the necessary basics of differential calculi, complex structures, and Kahler structures, as well as their interaction with compact quantum matrix group algebras, as originally considered in  \cite{DOS1}. We also recall the necessary basics of spectral triples and $K$-homology.

In \textsection 3, we show that  the Laplacian $\DEL_{\adel}$ decomposes with respect to Hodge decomposition, allowing us to deduce the spectrum of $\DEL_{\adel}$  from the spectrum of the operator {\small $\adel^{\dagger} \adel$}.  Restricting to the covariant case, we then  consider multiplicity-free comodules  as a framework in which to present complex structures of Gelfand type. In this case, it is observed that the operator {\small $\adel^{\dagger} \adel$} always diagonalises over  irreducible comodules.

In \textsection 4, we restrict for the first time to the setting of Drinfeld--Jimbo quantum groups,
exploiting the associated highest weight structure of their representation theory. In particular we consider products of a highest weight form $\w$, with powers of a zero form $z^l$, for $l \in \bN_0$. The form $z^l\omega$ is shown to always be an eigenvector of the Laplacian, and the corresponding eigenvalues $\mu_l$ are described explicitly.  Such sequences $\{\mu_l\}_{l \in \bN_0}$ of eigenvalues form the basis of our approach to  the spectrum of the Dolbeault--Dirac operator.

In \textsection 5, we abstract the representation theoretic properties of $\ccpn$ and introduce the notion of an {\em order I} compact quantum homogeneous complex space. We then establish a necessary and sufficient set of conditions (given in terms of the eigenvalue sequences $\{\mu_l\}_{l\in \bN_0}$ discussed above) for such a space to have a Dolbeault--Dirac operator with compact resolvent. 

In \textsection 6  we examine our motivating example $\O_q(\ccpn)$. We begin by recalling the necessary details about its definition as a quantum homogeneous space, its \hk calculus, and its covariant complex and K\"ahler structures. We then  construct an order I presentation for the calculus in \textsection \ref{subsection:OrderICPN}. This allows us to apply the general framework of the paper and to prove one of its main results:

\textbf{Theorem \ref{thm:theTHMCPN}}
For quantum projective space   $\O_q(\mathbb{CP}^{n-1})$, endowed with its Heckenberger--Kolb calculus and its unique covariant K\"ahler structure,  a pair of spectral triples is given by
\begin{align*}
\Big(\O_q(\ccpn), L^2\big(\Om^{(\bullet,0)}\big), D_{\del}\Big), & & \Big(\O_q(\ccpn), L^2\big(\Om^{(0,\bullet)}\big), D_{\adel}\Big).
\end{align*}

In \textsection 7  we generalise  the notion of Gelfand type to {\em weak Gelfand type}, and determine  which non-exceptional irreducible quantum flag manifolds satisfy the condition. We show that in addition to $\O_q(\ccpn)$, we have the quantum $2$-plane Grassmannians $\O_q(\mathrm{Gr}_{n,2})$, the odd- and even-dimensional quantum quadrics $\O_q(\mathbf{Q}_n)$. The extension of the framework of the paper to this larger class of examples is then discussed in detail. 

We finish with three appendices. In the first  we recall the basic definitions of compact quantum group algebras, and quantum homogeneous spaces. In the second we recall basic results about Drinfeld--Jimbo quantised enveloping algebras and their representation theory. In the third we use Frobenius reciprocity for quantum homogeneous spaces, along with some classical branching laws, to derive the decomposition of the anti-holomorphic forms into irreducible $U_q(\frak{sl}_n)$-modules. 

\subsubsection*{Acknowledgements:} The authors would like to thank Marco Matassa, Elmar Wagner, Fredy D\'iaz Garc\'ia, Andrey Krutov, Karen Strung, Shahn Majid, Simon Brain, Bram Mesland, Branimir \'{C}a\'{c}i\'c, Adam Rennie, Paolo Saracco, Kenny De Commer, Matthias Fischmann, Peter Littelmann, Willem A. de Graaf, Jan \v{S}\v{t}\!ov\'i\v{c}ek, and Adam-Christiaan van Roosmalen,  for useful discussions during the preparation of this paper.  The second author would like to thank IMPAN Wroclaw for hosting him in November 2018, and would also  like to thank Klaas Landsman and the Institute for Mathematics, Astrophysics and Particle Physics for hosting him in the winter of 2017 and 2018.

\section{Preliminaries}

We recall necessary details about complex, Hermitian, and K\"ahler structures for differential calculi. We highlight in particular the Dirac and Laplace operators associated to an Hermitian structure, as well as the associated Hodge theory, which plays a central in our spectral calculations.  We also introduce the novel notions of CQH-complex, CQH-Hermitian, and CQH-K\"ahler spaces. These serve as a convenient setting in which to present Dolbeault--Dirac spectral triples in the coming sections.

\subsection{Complex, Hermitian, and K\"ahler Structures on Differential Calculi}

In this subsection we present the basic definitions and results for complex structures, as  introduced in \cite{KLvSPodles} and  \cite{BS}. (For a presentation using the conventions of this paper see \cite{MMF2}.) We also recall the basic definitions and results of Hermitian and K\"ahler structures, as introduced in \cite{MMF3}. For an excellent presentation of classical complex and K\"ahler geometry see \cite{HUY}.

Recall that a {\em differential calculus} is a differential graded algebra $\big(\Om^\bullet \simeq \bigoplus_{k \in \bN_0} \Om^k, \exd\big)$ which is generated in degree $0$ as a differential graded algebra, which is to say, it is generated as an algebra by the elements $a, \exd b$, for $a,b \in \Om^0$. We denote the degree of a homogeneous element $\w \in \Om^{\bullet}$ by $|\w|$. For a given algebra $B$, a differential calculus {\em over} $B$ is a differential calculus such that $B = \Om^0$.  A differential calculus is said to be of {\em total degree} $m \in \bN_0$ if $\Om^m \neq 0$, and $\Om^{k} = 0$, for every $k > m$. A {\em differential $*$-calculus} over a $*$-algebra $B$ is a differential calculus over $B$ such that the $*$-map of $B$ extends to a (necessarily unique) conjugate linear involutive map $*:\Om^\bullet \to \Om^\bullet$ satisfying $\exd(\w^*) = (\exd \w)^*$, and 
\begin{align*}
\big(\w \wed \nu\big)^*  =  (-1)^{kl} \nu^* \wed \w^*, &  & \text{ for all } \w \in \Om^k, \, \nu \in \Om^l. 
\end{align*}
For $A$ a Hopf algebra, and $P$ a left $A$-comodule algebra, a differential calculus  $\Om^\bullet$ over $P$ is said to be {\em covariant} if the coaction $\DEL_L:P \to A \otimes P$ extends to a (necessarily unique) comodule algebra structure  $\DEL_L:\Om^\bullet \to A \otimes \Om^\bullet$, \wrt which the differential $\exd$  is a left $A$-comodule map. 

\begin{defn}\label{defnnccs}
A {\em complex structure} $\Om^{(\bullet,\bullet)}$ for a  differential $*$-calculus  $(\Om^{\bullet},\exd)$, over a \mbox{$*$-\alg } $A$, is an $\bN^2_0$-\alg grading $\bigoplus_{(a,b)\in \bN^2_0} \Om^{(a,b)}$ for $\Om^{\bullet}$ such that, for all $(a,b) \in \bN^2_0$: 
\begin{enumerate}
\item \label{compt-grading}  $\Om^k = \bigoplus_{a+b = k} \Om^{(a,b)}$,
\item  \label{star-cond} $*\big(\Om^{(a,b)}\big) = \Om^{(b,a)}$,
\item  $\exd \Om^{(a,b)} \sseq \Om^{(a+1,b)} \oplus \Om^{(a,b+1)}$.
\end{enumerate}
\end{defn}

We call an element of $\Om^{(a,b)}$ an $(a,b)$-form. Denoting by  $ \proj_{\Om^{(a+1,b)}}$, and $ \proj_{\Om^{(a,b+1)}}$,  the projections from $\Om^{a+b+1}$ onto $\Om^{(a+1,b)}$, and $\Om^{(a,b+1)}$ respectively, we can define the operators 
\begin{align*}
\del|_{\Om^{(a,b)}} : = \proj_{\Om^{(a+1,b)}} \circ \exd, & & \ol{\del}|_{\Om^{(a,b)}} : = \proj_{\Om^{(a,b+1)}} \circ \exd.
\end{align*}
Part 3 of the definition of a complex structure then implies the following identities:
\begin{align*}
\exd = \del + \adel, & &  \adel \circ \del = - \, \del \circ \adel, & & \del^2=0, & &  \adel^2 = 0. 
\end{align*}
Thus  $\big(\bigoplus_{(a,b)\in \bN^2}\Om^{(a,b)}, \del,\ol{\del}\big)$ is a double complex, which we call  the {\em Dolbeault double complex} of the $\Om^{(\bullet,\bullet)}$. Moreover, it is easily seen  that 
\bal \label{stardel}
\del(\w^*) = \big(\adel \w\big)^*, & &  \adel(\w^*) = \big(\del \w\big)^*, & & \text{ for all  } \w \in \Om^\bullet.
\eal
For any complex structure $\Om^{(\bullet,\bullet)} = \bigoplus_{(a,b) \in \bN_0} \Om^{(a,b)}$, a second complex structure, called its {\em opposite complex structure}, is given by
$
\ol{\Om}^{(\bullet,\bullet)} = \bigoplus_{(a,b) \in \bN^2_0} \ol{\Om}^{(a,b)}, \text{ where } \ol{\Omega}^{(a,b)} = \Omega^{^{(b,a)}}.
$

For a left $A$-comodule algebra $P$, and a covariant differential $*$-calculus $\Om^\bullet$ over $P$,   we say that a complex structure for $\Om^\bullet$ is {\em covariant} if  $\Om^{(a,b)}$ is a left $A$-sub-comodule of $\Om^\bullet$, for all  $(a,b) \in \bN^2_0$. 
A direct consequence of covariance is that the maps $\del$ and $\adel$ are left $A$-comodule maps.

\begin{defn} An {\em Hermitian structure} $(\Om^{(\bullet,\bullet)}, \s)$ for a differential $*$-calculus $\Om^{\bullet}$, of even total degree $2n$,  is a pair  consisting of  a complex structure  $\Om^{(\bullet,\bullet)}$, and  a central real $(1,1)$-form $\s$, called the {\em Hermitian form}, such that, \wrt the {\em Lefschetz operator}
\begin{align*}
L:\Om^\bullet \to \Om^\bullet,  & &   \w \mto \s \wed \w,
\end{align*}
isomorphisms are given by
\bal \label{eqn:Liso}
L^{n-k}: \Om^{k} \to  \Om^{2n-k}, & & \text{ for all } 1 \leq k < n.
\eal
\end{defn}


For $L$ the Lefschetz operator of an Hermitian structure, we denote
\begin{align*}
P^{(a,b)} : = \begin{cases} 
      \{\a \in \Om^{(a,b)} \,|\, L^{n-a-b+1}(\a) = 0\}, &  \text{ ~ if } a+b \leq n,\\
      0 & \text{ ~ if } a+b > n.
   \end{cases}
\end{align*}
Moreover, we denote $P^k := \bigoplus_{a+b = k} P^{(a,b)}$, and $P^{\bullet} := \bigoplus_{k \in \bN_0}  P^k$. An element of $P^\bullet$ is called a {\em primitive form}.

\begin{prop}[Lefschetz decomposition] \label{LDecomp}
For $L$ the Lefschetz operator of an Hermitian form, an  $A$-bicomodule decomposition, called the  {\em Lefschetz decomposition}, is given by
\begin{align*}
\Om^{(a,b)} \simeq \bigoplus_{j \geq 0} L^j\big(P^{(a-j,b-j)}\big).
\end{align*}
\end{prop}

Finally, we come to the definition of a K\"ahler structure, a simple strengthening of the Hermitian structure requirements, but one with profound consequences.

\begin{defn}
A {\em K\"ahler structure} $(\Om^{(\bullet,\bullet)},\k)$ for a differential $*$-calculus is an Hermitian structure  \st  $\exd \k = 0$. We call  $\k$ a  {\em K\"ahler form}.
\end{defn}

 One of the most important consequences of the K\"ahler condition (which is not necessarily true for a general Hermitian structure) is the equality, up to a scalar multiple of the three Laplacian operators
\bal \label{eqn:identityofLAPS}
\DEL_{\del} = \DEL_{\adel} = \frac{1}{2} \DEL_{\exd}.
\eal

Let $A$ be a Hopf \algn,  $P$ a left $A$-comodule algebra, and  $\Om^\bullet$ a covariant $*$-calculus over $P$. A {\em covariant Hermitian} structure for $\Om^\bullet$  is an Hermitian structure $(\Om^{(\bullet,\bullet)},\s)$ such that $\Om^{(\bullet,\bullet)}$ is a covariant complex structure, and such that the Hermitian form $\s$ is left $A$-coinvariant, which is to say $\DEL_L(\s) = 1 \oby \s$. A {\em covariant K\"ahler} structure is a covariant Hermitian structure which is also a K\"ahler structure.  Note that in the covariant case, in addition to being a $P$-bimodule map and a \mbox{$*$-homomorphism}, $L$ is also a left $A$-comodule map.

\subsection{Metrics, Adjoints, and the Hodge Map}

In classical Hermitian geometry, the Hodge map of an Hermitian metric is related to the associated Lefschetz decomposition through the well-known Weil formula (see \cite[Theorem 1.2]{Weil} or \cite[Proposition 1.2.31]{HUY}). In the noncommutative setting  we take the direct generalisation of the Weil formula for our definition of the Hodge map.

\begin{defn} \label{defn:HDefn}
The {\em Hodge map} associated to an Hermitian structure $\big(\Om^{(\bullet,\bullet)},\s\big)$ is the morphism uniquely defined by
\begin{align*}
\ast_{\s}\big(L^j(\w)\big) = (-1)^{\frac{k(k+1)}{2}}i^{a-b}\frac{j!}{(n-j-k)!}L^{n-j-k}(\w), & & \w \in P^{(a,b)} \sseq P^k,
\end{align*}
\end{defn}

Many of the basic properties of the classical Hodge map can now be understood as consequences of the Weil formula. (See \cite[\textsection 4.3]{MMF3} for a proof.)

\begin{lem}\label{lem:Hodgeproperties}
Let $\Om^\bullet$ be  a differential $*$-calculus, of total degree $2n$. For $(\Om^{(\bullet,\bullet)},\s)$ a choice of Hermitian structure for $\Om^{\bullet}$ and $\ast_\s$ the associated Hodge map, it holds that:
\bet
\item $\ast_\s$ is a $*$-map,
\item  $ \ast_{\s}(\Om^{(a,b)}) = \Om^{(n-b,n-a)}$, 
\item $ \ast_{\s}^2(\w) = (-1)^k \w, \text{ ~~~~~ for all } \w \in \Om^k$,
\item whenever $\Om^\bullet$ is a covariant calculus over a left $A$-comodule algebra, and $(\Om^{(\bullet,\bullet)},\s)$ is a covariant Hermitian structure, then $\ast_\s$ is a left $A$-comodule map.
\eet
\end{lem}

Reversing the classical order of construction we now define a metric in terms of the Hodge map.

\begin{defn}
The {\em metric} associated to the Hermitian structure $\big(\Om^{(\bullet,\bullet)},\s\big)$ is  the unique map $g_\s:\Om^\bullet \times \Om^\bullet \to A$  for which $g\big(\Om^k \times \Om^l\big) = 0$, for all $k \neq l$, and 
\begin{align*}
g_\s(\w, \nu) =  \ast_\s\big(\ast_\s(\w^*) \wed \nu \big), & &  \w,\nu \in \Om^k.
\end{align*}
\end{defn}

\subsection{CQH-Complex and CQH-Hermitian Spaces} \label{section:CQHHS}

We now introduce the definitions of the various compact quantum homogeneous spaces which form the framework for our investigation  of Dolbeault--Dirac operators. These definitions detail a natural list of compatibility conditions between differential calculi, complex structures, and Hermitian structures on one hand and compact quantum group algebras on the other. 

Throughout this subsection, and indeed the rest of the paper,  $A$ and $H$ will denote Hopf algebras, $\pi:A \to H$ a Hopf algebra map, and 
$$
B := A^{\co(H)} = \{ b \in A \,|\, b_{(1)} \otimes \pi(b_{(2)})\}
$$
the associated quantum homogeneous space. See Appendix \ref{app:CQGA} for further details on quantum homogeneous spaces.

\begin{defn}
A {\em compact quantum calculus homogeneous  space}, or simply {a CQH-calculus space} is  a triple 
\begin{align*}
{\bf B} = \Big(B = A^{\co (H)}, \Om^\bullet, \vol\Big),
\end{align*}
comprised of the following elements:
\bet 
\item $B = A^{\co (H)}$  a CMQGA-homogeneous space, for which $A$ is a domain,

\item $\Om^\bullet$ a covariant differential $*$-calculus over $B$, finite-dimensional as an object in $^A_B\mathrm{Mod}_0$,  and of  total degree $m \in \bN$,

\item $\vol: \Om^m \simeq B$ an isomorphism in $^A_B \mathrm{Mod}_0$, which is also a $*$-map, and  with respect to which the {\em integral} 
\begin{align*}
\int := \haar \circ \vol: \Om^m \to \bC
\end{align*}
 is {\em closed}, which is to say, satisfies $\int  \exd  \w = 0$, for all $\w \in \Om^{m-1}$.
\eet
\end{defn}


\begin{defn}
A {\em compact quantum  homogeneous complex space}, or a {\em CQH-complex} {\em space}, is a pair $\text{\bf C} = \big(\text{\bf B}, \Om^{(\bullet,\bullet)} \big)$  where 
\bet

\item $\text{\bf B} = \big(B = A^{\co(H)}, \Om^\bullet, \vol)$ is a CQH-calculus space,

\item $\Om^{(\bullet,\bullet)}$ is covariant complex structure for $\Om^\bullet$.

\eet
The associated {\em opposite} CQH-complex space is the pair  $\text{\bf C}^{\mathrm{op}} := \big(\text{\bf B}, \ol{\Om}^{(\bullet,\bullet)} \big)$
\end{defn}

In the same spirit, a CQH-Hermitian space is  a CQH-calculus space endowed with an Hermitian structure in a natural way. The interaction here, however, is a little more subtle.

\begin{defn}
A {\em compact quantum  homogeneous Hermitian space}, or alternatively a  {\em CQH-Hermitian space}, is a triple $\big(\text{\bf C}, \Om^{(\bullet,\bullet)},\s \big)$  consisting of  
\bet
\item $\text{\bf C} = \big(\text{\bf B}, \Om^{(\bullet,\bullet)}\big)$ a CQH-complex space,

\item $\big(\Om^{(\bullet,\bullet)},\s\big)$ a covariant Hermitian structure for  the  differential $*$-calculus \mbox{ $\Om^\bullet \in  \text{\bf B}$,}

\item  $\vol = \ast_{\s}|_{\Om^{2n}}: \Om^{2n} \to B$, for $2n$ the total degree of the constituent calculus of $\mathbf{B}$,

\item the associated metric $g$ is {\em positive definite}, which is to say
\begin{align*}
g_{\sigma}(\w,\w) \in B_{>0} := \Big\{\sum_{i=1}^l \lambda_i b_i^* b_i  \neq 0 ~ | ~ b_i \in B, \, \lambda_i \in \bR_{>0}, \, l \in \bN \Big\}, & & \textrm{ for all $\omega \in \Omega^{\bullet}$}.
\end{align*}

\eet
\end{defn}

As an immediate consequence of the definition, we get the following useful lemma.

\begin{lem} \label{lem:nonzerozaction}
For any CQH-Hermitian space  $\big(\text{\bf C}, \Om^{(\bullet,\bullet)},\s \big)$, it holds that 
$z\omega$ and $ \omega z$ are non-zero,  for all   non-zero $z \in B$, and $\omega \in \Omega^{\bullet}$.
\end{lem}
\begin{proof}
By the definition of $g_{\sigma}$, it holds that $g_{\sigma}(\omega,\omega z) = g_{\sigma}(\omega,\omega) z$. Since $B$ is a domain, and $g_{\sigma}$ is positive definite, $g_{\sigma}(\omega,\omega) z$ is necessarily non-zero, implying that $\omega z$ is non-zero. The proof that $z \omega$ is non-zero is entirely analogous. 
\end{proof}

For any CQH-Hermitian  space, composing $\haar$ with $g_{\sigma}$ gives a sesqui-linear map
\begin{align*}
\la \cdot, \cdot \ra: \Om^{(\bullet,\bullet)}  \times \Om^{(\bullet,\bullet)}  \to \bC, & & (\w ,\nu) \mto  \haar \circ g_{\sigma}(\w,\nu).
\end{align*}
We note that positivity of $g$, together with positivity of the Haar state $\haar$, imply that $\la \cdot, \cdot \ra$ is an inner product. We finish with the obvious extension of these definitions to the K\"ahler case.

\begin{defn}
A {\em compact quantum K\"ahler homogeneous space}, or a CQH{\em -K\"ahler space}, is a CQH-Hermitian  space $\text{\bf K} = \big(\text{\bf B},  \Om^{(\bullet,\bullet)},\k\big)$ such that the covariant Hermitian structure $(\Om^{(\bullet,\bullet)}, \, \k)$ is a K\"ahler structure. 
\end{defn}

\subsubsection{Hodge Theory}

We now recall the noncommutative generalisation of Hodge theory associated to any CQH-Hermitian space. Hodge decomposition will play a central role in our calculation of Laplace operator spectra, and the implied equivalence between harmonic forms and cohomology groups allows us to calculate the Dirac operator index in terms of the anti-holomorphic holomorphic Euler characteristic of the calculus, as defined in Definition \ref{defn:Euler}.

The exterior derivatives $\exd, \,\del, \, \adel$ are  adjointable with respect to the inner product. Moreover,  as established in \cite[\textsection 5.3.3]{MMF3}, their adjoints also are expressible  in terms of the Hodge operator:
\begin{align*}
\exd^\dagger = - \ast_\s \circ \, \exd \circ \ast_\s, & & \del^\dagger = - \ast_\s \circ \, \adel  \circ \ast_\s, & & \adel^\dagger = - \ast_\s \circ \, \del \circ \ast_\s.
\end{align*}
For any Hermitian structure on a differential calculus,  generalising the classical situation, we define the $\exd$-, $\del$-, and $\adel$-Dirac operators to be respectively, 
\begin{align*}
& D_{\exd} := \exd + \exd^\dagger, &  \, & D_{\del} := \del + \del^\dagger,&  & D_{\adel} := \adel + \adel^\dagger.
\end{align*}

Moreover, we define the $\exd$-, $\del$-, and $\adel$-Laplace operators to be
\begin{align*}
& \DEL_{\exd} := (\exd + \exd^\dagger)^2, &  \, &\DEL_{\del} := (\del + \del^\dagger)^2,&  &\DEL_{\adel} := (\adel + \adel^\dagger)^2.&
\end{align*}
We introduce  {\em $\exd$-harmonic}, {\em $\del$-harmonic}, and {\em $\adel$-harmonic} forms, respectively, according to
\begin{align*}
& \H_{\exd} :=\ker(\DEL_{\exd}),& &\H_{\del} :=\ker(\DEL_{\del}),&  &\H_{\adel} :=\ker(\DEL_{\adel}).&
\end{align*}

For the case of CQH-Hermitian space, it was shown in \cite[Corollary 4.17]{DOS1} that the Dirac and Laplace operators  are diagonalisable. Just as in the classical case, this allows us to show (\cite[Lemma 6.1]{MMF3}) that the harmonic forms admit the following alternative presentation
\bal \label{eqn:harmcap}
\H_{\exd} = \ker(\exd) \cap \ker(\exd^\dagger), & & \H_{\del} = \ker(\del) \cap \ker(\del^\dagger), & &  \H_{\adel} = \ker(\adel) \cap \ker(\ol{\del}^\dagger).
\eal
Moreover, as shown in  \cite[Theorem 6.2]{MMF3}, building on earlier work in \cite{KMT},  diagonalisability also allows us to conclude the following noncommutative generalisation of Hodge decomposition for Hermitian manifolds.

\begin{thm}[Hodge decomposition] For a CQH-Hermitian space, direct sum decompositions of $\Om^\bullet$, orthogonal \wrt  $\la\cdot,\cdot\ra$, are given by 
\begin{align*}
\Om^{\bullet}  = {\mathcal H}_{\exd} \oplus \exd \Om^{\bullet} \oplus \exd^\dagger\Om^{\bullet}, ~& & ~ \Om^{\bullet} = {\mathcal H}_{\del} \oplus \del\Om^{\bullet} \oplus \del^\dagger\Om^{\bullet}, ~ & & ~ \Om^{\bullet}  = {\mathcal H}_{\ol{\del}} \oplus \ol{\del}\Om^{\bullet} \oplus \ol{\del}^\dagger\Om^{\bullet}.
\end{align*}
\end{thm} 

From Hodge decomposition it is easy to conclude the following equivalence of harmonic forms and cohomologies, see \cite[\textsection 6.2]{MMF3} for details. 

\begin{cor} \label{cor:harmonictoclass}  
It holds that 
\begin{align*}
\ker(\exd) = \H_{\exd} \oplus \exd \Om^\bullet, & & \ker(\del) = \H_{\del} \oplus \del \Om^\bullet, & & \ker(\adel) = \H_{\adel} \oplus \adel \Om^\bullet,
\end{align*}
and so, we have isomorphisms
\begin{align*}
\H^k_{\exd}  \to  \,H^k_{\exd}, & &  \H^{(a,b)}_{\del} \to H^{(a,b)}_{\del}, & & \H^{(a,b)}_{\adel} \to H^{(a,b)}_{\adel},
\end{align*}
where $H^k_{\exd}$, $H^{(a,b)}_{\del}$, and $H^{(a,b)}_{\adel}$, denote the cohomology groups of the complexes $(\Omega^{\bullet},\exd)$, $(\Omega^{\bullet},\del)$, and $(\Omega^{\bullet},\adel)$ respectively.
\end{cor}


We finish with an easy but novel observation that directly generalises the classical situation. 

\begin{lem} \label{lem:HodgeOpDecomp}
For any Hermitian structure $(\Om^{(\bullet,\bullet)},\s)$, the Hodge operator restricts to linear isomorphisms
\begin{align*}
\ast_{\s}: \adel \Om^\bullet \to \del^\dagger \Om^\bullet, & & \ast_{\s}: \del \Om^\bullet \to \adel^\dagger\! \Omega^\bullet.
\end{align*}
\end{lem}
\begin{proof}
The fact that $\ast_{\s}\big(\adel \Om^\bullet)$ is contained in $\del^\dagger \Om^\bullet$ follows from 
\begin{align*}
\ast_{\s}(\adel \w) = (-1)^k \ast_{\s} \circ \, \adel \circ \ast_{\s}\big(\ast_{\s}(\w)\big) = (-1)^{k+1} \del^\dagger\big(\ast_{\s}(\w)\big) \in \del^\dagger\Om^\bullet, & & \text{ for } \w \in \Om^k.
\end{align*}
Similarly, the fact that $\ast_{\s}\big(\del^\dagger \Om^\bullet \big)$ is contained in $\adel \Om^\bullet$ follows from 
\begin{align*}
\ast_{\s}\big(\del^\dagger \w \big) = - \ast_{\s}^2 \circ \, \adel \circ \ast_{\s}(\w) = (-1)^{2n-k +1} \, \adel\big(\ast_{\s}(\w)\big) \in \adel \Om^\bullet, & & \text{ for } \w \in \Om^k.
\end{align*}
Thus, since $\ast_{\s}:\Om^\bullet \to \Om^\bullet$ is a linear isomorphism, it must restrict to an isomorphism between $\adel \Om^\bullet$ and $\del^\dagger \Om^\bullet$. The proof of the second isomorphism is completely analogous, and hence omitted.
\end{proof}

\subsection{Dolbeault--Dirac Spectral Triples}

In this subsection we recall the definition of a spectral triple,  Connes' notion of a noncommutative Riemannian spin manifold \cite{Connes}. For a presentation of the classical Dolbeault--Dirac operator of an Hermitian manifold as a commutative spectral triple, see  \cite{HigsonRoe} or \cite{FriedrichDirac}. For a standard reference on the general theory of spectral triples, see  \cite{Varilly} or \cite{RennieSpecTrip}. Motivated by our construction of spectral triples from CQH-Hermitian spaces, we find it convenient to break the definition into two parts.

\begin{defn}
A {\em bounded-commutator triple} $(A,\H,D)$, or simply a {\em BC-triple},  consists of a unital  \mbox{$*$-algebra} $A$, a separable Hilbert space $\H$ endowed with a faithful \linebreak $*$-representation $\r:A \to \mathbb{B}(\H)$, and  $D: \dom(D)  \to \H$ a densely-defined self-adjoint operator on $\H$, such that
\begin{enumerate}
\item $\r(a)\mathrm{dom}(D) \sseq \mathrm{dom}(D)$, for all $a \in A$, 

\item $[D,\r(a)]$ is a bounded operator,  for all $a \in A$.

\end{enumerate}

An {\em even BC-triple} is a BC-triple $(A,\H,D)$ together with a  $\bZ_2$-grading $\H = \H_0 \oplus \H_1$ of Hilbert spaces, \wrt which $D$ is a degree $1$ operator, and $\r(a)$ is a degree $0$ operator, for each $a \in A$.
\end{defn}

It follows the discussions of  \cite[\textsection 4]{DOS1} that every CQH-Hermitian space automatically gives a BC-triple. We now briefly recall the relevant details. Let  $\text{\bf H} = \left(\mathbf{B},\Om^{(\bullet,\bullet)},\s\right)$ be a CQH-Hermitian space, with constituent quantum homogeneous space $B = A^{\co(H)}$. We denote by $L^2\big(\Om^{\bullet}\big)$ the Hilbert space completion of $\Om^{\bullet}$ \wrt its inner product $\la \cdot,\cdot\ra$. 
By \cite[Lemma 5.2.1]{MMF3} the $\bN^2_0$-grading of the complex structure is orthogonal \wrt $\la \cdot, \cdot \ra$, hence we have the following decomposition of Hilbert spaces
\begin{align} \label{eqn:ComplexHilbertDecomp}
L^2\left(\Om^\bullet\right) = \bigoplus_{(a,b) \in \bN^2_0} L^2\big(\Om^{(a,b)}\big).
\end{align}
The constituent Hopf algebras of $\text{\bf H}$ are CMQGAs, in particular, they are finitely generated. This implies that the Hilbert space $L^2\left(\Omega^\bullet\right)$ is separable \cite[Lemma 4.3]{DOS1}. Since $B$ is a unital algebra, a faithful $*$-representation $\r:B \to \mathbb{B}\big(L^2(\Om^\bullet)\big)$ is given by 
\bal \label{eqn:faithfulRep}
\r(b) \w := b \w, & & \text{ for } \w \in \Om^\bullet, ~ b \in B.
\eal

It follows from the basic theory of unbounded operators  \cite[\textsection 13]{Rudin} that  the Dirac operators $D_{\del}$ and $D_{\adel}$, as well as the Laplace operators $\DEL_{\del}$ and $\DEL_{\adel}$,  are essentially self-adjoint, see \cite[Corollary 4.17]{DOS1} for details. By abuse of notation,  we will not distinguish notationally between an operator and its closure. As  observed in  \linebreak \cite[\textsection 7.2]{DOS1}, it now follows from \cite[Proposition 2.1]{CoreFMR} that, for $\mathrm{dom}(D_{\del})$ and $\mathrm{dom}(D_{\adel})$ the domain of the closures of the respective Dirac operators, 
\begin{align*}
\rho(b) \mathrm{dom}(D_{\del}) \sseq \mathrm{dom}(D_{\del}), & & \rho(b) \mathrm{dom}(D_{\adel}) \sseq \mathrm{dom}(D_{\adel}), & & \text{ for all } b \in B. 
\end{align*} 
Moreover,  boundedness of the commutators $[D_{\del},m]$, and $[D_{\adel},m]$ follows easily from  the Leibniz rule \cite[Corollary 4.11]{DOS1}. Collecting these observations together gives us the following proposition.

\begin{prop} \label{prop:BCCQHHS}
For a CQH-Hermitian space $\mathbf{H} = (\mathbf{B},\Om^{(\bullet,\bullet)},\s)$, with constituent quantum homogeneous space $B$, a pair of BC-triples, which we call a {\em Dolbeault--Dirac pair}, is given by
\begin{align} \label{eqn:DDpair}
\left(\!B, L^2\big(\Om^{(\bullet,0)}\big), D_{\del}\right)\!,  & & \left(\!B, L^2\big(\Om^{(0,\bullet)}\big), D_{\adel}\right)\!.
\end{align}
\end{prop}
\begin{cor}
The Hilbert space decompositions 
\begin{align*}
L^2\big(\Om^{(\bullet,0)}\big) = \bigoplus_{k \in 2 \mathbb{N}_0} L^2\big(\Om^{(k,0)}\big) \oplus \bigoplus_{k \in 2 \mathbb{N}_0 + 1} L^2\big(\Om^{(k,0)}\big), \\ 
L^2\big(\Om^{(0, \bullet)}\big) = \bigoplus_{k \in 2 \mathbb{N}_0} L^2\big(\Om^{(0,k)}\big) \oplus \bigoplus_{k \in 2 \mathbb{N}_0 + 1} L^2\big(\Om^{(0,k)}\big),
\end{align*}
define an even structure for the Dolbeault--Dirac pair of BC-triples.
\end{cor}

Finally, we come to the definition of a spectral triple, which we present as a BC-triple whose unbounded operator  $D$ has compact resolvent. 

\begin{defn}
A {\em spectral triple} is a BC-triple $(A,\H,D)$ such that
\begin{align*}
(1+D^2)^{-1}  \in \mathbb{K}\left(\H\right),
\end{align*}
where $\mathbb{K}(\H)$ denotes the compact operators on $\H$.
\end{defn}

As discussed in the introduction, it is not clear at present how to conclude the compact resolvent condition from the properties of a general CQH-Hermitian space. (See \textsection \ref{Subsubsection:Conjs} for a brief discussion on how this might be achieved.) Hence, in our examples we resort to calculating the spectrum explicitly, and directly confirming the appropriate eigenvalue growth.  We do, however, know two important general results. Firstly, since the $*$-map satisfies $* \circ \Delta_{\adel} = \Delta_{\del} \circ *$ we have the following lemma.
\begin{lem}
For a CQH-Hermitian space, the operator $D_{\adel}:\Omega^{(0,\bullet)} \to \Omega^{(0,\bullet)}$ has compact resolvent if and only if the operator $D_{\del}: \Omega^{(\bullet,0)} \to \Omega^{(\bullet,0)}$ has compact resolvent.
\end{lem}
Secondly, we know from  \cite[Corollary 4.17]{DOS1}  that the operators $D_{\del}$ and $ D_{\adel}$  are diagonalisable on $L^2\big(\Om^\bullet\big)$. Hence, a CQH-Hermitian space gives a spectral triple if and only if the eigenvalues $\{\mu_n\}_{n \in \bN_0}$  of $\Delta_{\adel}$ (with repetitions representing multiplicities)  tend to infinity. When dealing with questions of eigenvalue growth, we find the following notation useful. 

\begin{notation}
For $T:\dom(T) \to \H$ a diagonalisable operator on a separable Hilbert space $\H$, we write
$
\s_P(T) \to \infty
$
to denote that the eigenvalues of $T$ (with repetitions representing multiplicities) tend to infinity.
\end{notation}

\subsection{Analytic $K$-Homology and the Anti-Holomorphic Euler Characteristic}

Let us now recall the basics of analytic $K$-homology, bearing in mind that one of the main aims of this paper is to establish a connection between $K$-theoretic index theory and the Dolbeault cohomology of noncommutative Hermitian structures.

\begin{defn}
Let $\A$ be a separable unital $C^*$-algebra.   A {\em Fredholm module} over $\A$ is a triple $(\H,\r,F)$, where  $\H$ is a Hilbert space, $\r: \A \to  \mathbb{B}(\H)$ is a $\ast$-representation,  and  $F:\H \to \H$ a bounded  linear operator, such that 
\begin{align*}
F^2 - 1,   & & F - F^*, & &  [F,\r(a)], 
\end{align*}
are all compact operators, for any $a \in \A$. An {\em even Fredholm module} is a Fredholm module $(\rho,F,\H)$ together with a  $\bZ_2$-grading $\H = \H_0 \oplus \H_1$ of Hilbert spaces, \wrt which $F$ is a degree $1$ operator, and $\r(a)$ is a degree $0$ operator, for each $a \in \A$.
\end{defn}

The {\em direct sum} of two even Fredholm modules is formed by taking the direct sum of Hilbert spaces, representations, and operators. For $(\H, \rho, F)$ an even Fredholm module, and  $U:\H \to \H'$ a degree $0$ unitary transformation, the triple $(\H',U^*\rho \, U, U^*FU)$ is again a Fredholm module. This defines an equivalence relation on Fredholm modules over $\A$, which we call {\em unitary equivalence}.  We say that a norm continuous family of Fredholm modules $(\rho,\H,F_t)$, for  $t \in [0,1]$, defines an {\em operator homotopy} between the two Fredholm modules $(\rho,\H,F_0)$ and $(\rho,\H,F_1)$. 

\begin{defn}
The {\em $K$-homology} group $K^0(\A)$ of a $C^*$-algebra $\A$ is the abelian group with one generator for each unitary equivalence class of even Fredholm modules, subject to the following relations: for any two even Fredholm modules $x_0$, $x_1$, 
\begin{enumerate}
\item $[x_0] = [x_1]$ if there exists an operator homotopy between  $x_0$ and $x_1$,

\item $[x_0 \oplus x_1] = [x_0] + [x_1]$, where $+$ denotes addition in $K^0(\A)$. 
\end{enumerate}
\end{defn}
Denoting the Fredholm operator
$
F_+ := F|_{\H_+}: \H_+ \to \H_-,
$
 a well-defined group homomorphism is given by
\begin{align*}
\mathrm{Index}: K^0(\A) \to \bZ, & & [F]_K \mto \mathrm{Index}(F_+).
\end{align*}

Spectral triples are important primarily because they  provide unbounded representatives for $K$-homology classes. For a spectral triple $(A,\H,D)$, its {\em bounded transform} is the  operator
\begin{align*}
\frak{b}(D) := \frac{D}{\sqrt{1+D^2}} \in \mathbb{B}(\H),
\end{align*}
defined via the functional calculus. Denoting by $\ol{A}$ the closure of $\rho(A)$ \wrt the operator topology of $\mathbb{B}(\H)$, an even Fredholm module over $\overline{A}$ is given by $(\H, \rho, \frak{b}(D))$. (See  \cite{CP1} for details.)

For a classical Hermitian manifold, the index of the Dolbeault--Dirac operator, and the associated index of its $K$-homology class, are equal to the holomorphic Euler characteristic of the manifold. This picture extends to the noncommutative setting.

\begin{defn} \label{defn:Euler} Let ${\bf H} := (\text{\bf B},\Om^{(\bullet,\bullet)},\s)$ be a CQH-Hermitian space  with constituent differential calculus $\Om^{\bullet} \in \text{\bf B}$ of  total degree $2n$. The {\em anti-holomorphic Euler characteristic}  of $\text{\bf H}$ is given by
\begin{align*}
\chi_{\adel} := \sum_{k=1}^{n} (-1)^{k} \dim_{\mathbb{C}}\big(H^{(0,k)}\big).
\end{align*}
\end{defn}

The following proposition now allows us to conclude non-triviality of the $K$-homology class of a Dolbeault--Dirac spectral triple from non-vanishing of the  Euler characteristic.

\begin{prop}\cite[Theorem  5.4]{DOS1}  For a {\em CQH}-Hermitian space $(\text{\bf B},\Om^{(\bullet,\bullet)},\s)$, with an associated  pair of Dolbeault--Dirac spectral triples, it holds that 
\begin{align*}
\mathrm{Index}\big(\frak{b}(D_{\del})\big) =  \mathrm{Index}\big(\frak{b}(D_{\adel})\big) = \chi_{\adel}.
\end{align*}
\end{prop}

\begin{remark}
The calculation of cohomology groups can in general be quite difficult. However, in the K\"ahler setting there exists a powerful noncommutative generalisation of the Kodaira vanishing theorem \cite{OSV}. In the more specialised  Fano setting (whose definition directly generalises the classical definition) this implies vanishing of all higher cohomologies, implying that the anti-holomorphic Euler characteristic is equal to $\dim_{\mathbb{C}}\big(H^{(0,0)}\big)$. Since $H^{(0,0)}$ always contains the identity element of $B$, this implies that the $K$-homology class of $\frak{b}(D_{\adel})$ is always non-zero in the Fano setting.
\end{remark}

\section{CQH-Complex Spaces of  \mbox{Gelfand Type}}

In this section we  begin our examination of the spectral behaviour of the Laplacian operator associated to a CQH-Hermitian space.  
Our  strategy is to exploit the subtle interactions between the Laplace spectrum, Hodge decomposition, and comodule multiplicities. This leads us to focus on a special type of CQH-Hermitian space,  Gelfand type spaces, 
for which the problem is significantly more tractable. The work of this section underlies our investigation of the  Drinfeld--Jimbo  case in  \textsection \ref{section:DJCQHHS}.

\subsection{Hodge Decomposition of the Laplacian}

We decompose the Laplacian \wrt Hodge decomposition and then, in Corollary \ref{cor:decomposingCTEG}, show that $\s_P\big(\DEL_{\adel}\big) \to \infty$ if and only if $\s_P${\small $(\adel \adel^\dagger)$}$ \to \infty$.  We emphasise the fact that nowhere in this subsection do we make any assumption on multiplicities, all results hold for a general CQH-Hermitian space.

\begin{lem}\label{lem:HDOFDEL}
For any CQH-Hermitian space, the Laplacian operator $\DEL_{\adel}$ admits a direct sum decomposition with  respect to  Hodge decomposition $\Om^\bullet = \H^\bullet \oplus \adel \Om^\bullet \oplus \adel^\dagger\!\Om^\bullet$, namely
\begin{align*}
\DEL_{\adel} = \,0 \, \oplus \, \adel\adel^\dagger \, \oplus \, \adel^\dagger\adel.
\end{align*}
\end{lem}
\begin{proof}
By definition,  $\DEL_{\adel}$ restricts to the zero map on  the harmonic forms $\H^\bullet$. For a non-harmonic $\w \in \Om^\bullet$, we have 
\begin{align*}
\DEL_{\adel}(\adel \w) = & \, (\adel \adel^\dagger + \adel^\dagger \adel)(\adel \w) = \adel \adel^\dagger(\adel \w) \in \adel \Om^\bullet.
\end{align*}
Thus we see that $\adel \Om^\bullet$ is closed under the action of the Laplacian, and moreover, that 
\begin{align*}
\DEL_{\adel}\, |_{\,\adel \Om^\bullet}  = \adel \adel^\dagger.
\end{align*}
Similarly,   $\adel^\dagger \Om^\bullet$ is closed under the action of  $\DEL_{\adel}$, and $\DEL_{\adel}$ restricts to the operator $\adel^\dagger \adel $ on $\adel^\dagger \Om^\bullet$.
\end{proof}

\begin{lem} \label{lem:VANISHINGCOMMS}
For any CQH-Hermitian space, it holds that 
$
[\Delta_{\adel},\adel] = [\Delta_{\adel},\adel^\dagger]  = 0.  
$
\end{lem}
\begin{proof}
Starting with the first commutator, for $\w \in \Om^\bullet$, we see  that
\begin{align*}
[\Delta_{\adel},\adel](\w) =  \, \Delta_{\adel}  \circ \adel(\w) - \adel \circ \DEL_{\adel}(\w)
                                             =  \, \adel  \adel^\dagger  \adel(\w) - \adel   \adel^\dagger  \adel(\w)
                                             =  \, 0.
\end{align*}
Vanishing of the commutator $[\Delta_{\adel},\adel^\dagger]$ is established similarly.
\end{proof}

\begin{prop} \label{prop:ADELISO}
For a CQH-Hermitian space, with constituent quantum homogeneous space $B = A^{\co(H)}$,  left  $A$-comodule isomorphisms are given by
\bet

\item $\adel: \adel^\dagger \Omega^\bullet \to  \adel \Omega^\bullet$,

\item $\adel^\dagger: \adel \Omega^\bullet \to  \adel^\dagger \Omega^\bullet$.

\eet
\end{prop}
\begin{proof}
Since by assumption the calculus $\Om^\bullet$ and the Hermitian structure $(\Om^{(\bullet,\bullet)},\sigma)$ are covariant, the maps $\adel$ and $\adel^\dagger$ are left $A$-comodule maps. Hence, it suffices to show that $\adel$ and $\adel^\dagger$ are linear isomorphisms. By Hodge decomposition $\adel^\dagger \!\Om^{(0,\bullet)}$ is orthogonal to  the space of harmonic forms, and so, the kernel of the restriction of the Laplacian to  $\adel^\dagger \! \Om^{(0,\bullet)}$ is trivial. Now Lemma \ref{lem:HDOFDEL} tells us that the Laplacian restricts to $\adel^\dagger \adel$ on the subspace $\adel^\dagger \!\Om^{(0,\bullet)}$, giving us that
\begin{align*}
\ker\left(\adel^\dagger \adel \, | \, _{\adel^\dagger \Om^{(0,\bullet)}}\right) = 0.
\end{align*} 
Similarly, the  kernel of the restriction of $\adel \adel^\dagger$  to the subspace $\adel \Om^{(0,\bullet)}$ is trivial.

Since the Laplacian is a self-adjoint operator,  its restrictions to the subspaces $\adel \Om^{(0,\bullet)}$ and $\adel^\dagger \!\Om^{(0,\bullet)}$ are diagonalisable. Thus, both operators 
\begin{align*}
\adel^\dagger \adel: \adel^\dagger \Om^\bullet \to \adel^\dagger \Om^\bullet , & & \adel \adel^\dagger: \adel \Om^\bullet \to \adel \Om^\bullet, 
\end{align*}
are diagonalisable with trivial $0$-eigenspaces. From this we see that we can construct explicit inverses for $\adel$ and $\adel^{\dagger}$, implying that both maps are isomorphisms as required.
\end{proof}

We now discuss the relationship between the eigenvalues of $\DEL_{\adel}$ and its decomposition \wrt Hodge decomposition. In doing so, we find the following notation useful: 

\begin{notation}
For $f:W \to W$ a linear operator on a vector space $W$, we denote by $E(\mu,f)$ the eigenspace of an eigenvalue $\mu$ of $f$.
\end{notation} 

\begin{cor}
The  non-zero eigenvalues of the following three operators coincide: 
\bet

\item \, $\DEL_{\adel}: \Om^{(0,\bullet)} \to \Om^{(0,\bullet)}$,

\item $\adel^{\dagger} \adel: \adel^{\dagger}\! \Om^{(0,\bullet)} \to \adel^{\dagger} \!\Om^{(0,\bullet)}$,

\item $\adel \adel^{\dagger}: \adel \Om^{(0,\bullet)} \to \adel \Om^{(0,\bullet)}$.

\eet
Moreover, for any eigenvalue $\mu$ of $\DEL_{\adel}$ with finite multiplicity, it holds that
\bal \label{eqn:delmult}
\dim_{\mathbb{C}}\big(\text{E}(\mu,\DEL_{\adel})\big) = 2 \dim_{\mathbb{C}}\Big(\text{E}\big(\mu,\adel^\dagger \adel\,|\,_{\adel^\dagger \Om^\bullet}\big)\Big )= 2 \dim_{\mathbb{C}}\Big(\text{E}\big(\mu,\adel \adel^\dagger\,|\,_{\adel \Om^\bullet}\big)\Big).
\eal
\end{cor}
\begin{proof}
The decomposition of $\DEL_{\adel}$ \wrt Hodge decomposition, as given in Lemma \ref{lem:HDOFDEL} above, implies that its non-zero eigenvalues are equal to the union of the eigenvalues of $\adel^\dagger \adel\,|\,_{\adel^\dagger \Om^\bullet}$ and the eigenvalues of $\adel \adel^\dagger\,|\,_{\adel \Om^\bullet}$. Hence, for any eigenvalue $\mu$, 
\begin{align*}
\text{E}\left(\mu,\DEL_{\adel}\right) = \text{E}\left(\mu,\adel^\dagger \adel\,|\,_{\adel^\dagger \Om^\bullet}\right) \oplus \text{E}\left(\mu,\adel \adel^\dagger\,|\,_{\adel \Om^\bullet}\right).
\end{align*}
Moreover,  Lemma \ref{lem:VANISHINGCOMMS} and Proposition \ref{prop:ADELISO} together imply that
the set of eigenvalues of the operators $\adel^\dagger \adel\,|\,_{\adel^\dagger \Om^\bullet}$ and $\adel \adel^\dagger\,|\,_{\adel \Om^\bullet}$ coincide and have the same multiplicity. The fact that these three operators have a common set of non-zero eigenvalues now follows, as does that identity for multiplicities given in (\ref{eqn:delmult}).
\end{proof}

An immediate consequence of the above corollary is now given. In short, it says that  verifying the compact resolvent  for the Laplacian can be reduced to verifying the condition for either $\adel^\dagger \adel\,|\,_{\adel^\dagger \Om^\bullet}$ or $\adel \adel^\dagger\,|\,_{\adel \Om^\bullet}$

\begin{cor} \label{cor:decomposingCTEG}
Assuming finite-dimensional cohomologies, the following three conditions are equivalent:
\bet

\item \,\,\, $\s_P\big(\DEL_{\adel}: \Om^{(0,\bullet)} \to \Om^{(0,\bullet)}\big) \to \infty$,

\item $\s_P\big(\adel \adel^{\dagger}: \adel \Om^{(0,\bullet)} \to \adel \Om^{(0,\bullet)}\big) \to \infty$,

\item $\s_P\big(\adel^{\dagger} \adel: \adel^{\dagger} \Om^{(0,\bullet)} \to \adel^{\dagger} \Om^{(0,\bullet)}\big) \to \infty$.

\eet
\end{cor}

We finish with an easy observation which follows from the equality of Laplacians in the K\"ahler case. While not needed elsewhere, the identity is interesting as an alternative description of the action of the Laplacian $\DEL_{\adel}$ on anti-holomorphic forms.

\begin{lem}
For a CQH-K\"ahler space $\text{\bf K}$,  the operators $\DEL_{\adel}$ and $\del^\dagger \del$ coincide on $\Om^{(0,\bullet)}$.
\end{lem}
\begin{proof}
From (\ref{eqn:identityofLAPS}) we know that in the K\"ahler case $\DEL_{\adel} = \DEL_{\del}$. Thus the identity follows from the fact that  $\DEL_{\del} = \del  \del^\dagger + \del^\dagger  \del$ restricts to $\del^{\dagger}  \del$ on $\Om^{(0,\bullet)}$.
\end{proof}

\subsection{CQH-Complex Spaces of Gelfand Type}

In this subsection we introduce CQH-complex spaces of Gelfand type. As shown in Lemma \ref{lem:GelfandDiag} below, the Laplacians associated to CQH-Hermitian spaces of Gelfand type  admit a particularly  nice diagonalisation, which makes the calculation of their spectrum significantly more tractable. We begin by introducing graded multiplicty-free comodules  as a convenient abstract framework in which to discuss comodule multiplicities for covariant calculi. 

\begin{defn} ~~~
\bet

\item We say that a left $A$-comodule $C$  is {\em multiplicity-free} if for any irreducible left $A$-comodule $V$, it holds that 
\begin{align*}
\dim_{\mathbb{C}}\big(\mathrm{Hom}^A(V,C)\big) = 1.
\end{align*}

\item A {\em graded left $A$-comodule} $P$ is a left $A$-comodule, together with an $\bN_0$-algebra grading $P = \bigoplus_{m \in \bN_0} P_m$, such that each $P_m$ is a left $A$-sub-comodule of $P$, or equivalently, such that the decomposition $P = \bigoplus_{m \in \bN_0} P_m$ is a decomposition in the category $\,^A\mathrm{Mod}$.  

\item We say that a graded comodule $P = \bigoplus_{m \in \bN_0} P_m$ is {\em graded multiplicity-free} if the left $A$-comodule  $P_m$ is multiplicity-free, for all  $m \in \bN_0$.
\eet
\end{defn}

An elementary, but very useful, observation about multiplicity-free comodules is presented in the following lemma. The proof is a direct  application of Schur's lemma, and so,  is omitted.

\begin{lem} \label{lem:DiagComoduleMaps}
Let $P$ be a graded multiplicity-free left $A$-comodule,   $\f:P \to P$  a degree $0$ left $A$-comodule map, and $V$ an irreducible $A$-sub-comodule of $P$. Then $\f$ acts on $V$ as a scalar multiple of the identity. Consequently, $\f$ is diagonalisable on $P$.
\end{lem}

We now present the notion of a CQH-complex space of Gelfand type. The definition is given in terms of graded multiplicity-free comodules, and is followed by a direct application of Lemma \ref{lem:DiagComoduleMaps}.

\begin{defn} We say that  a CQH-complex space ${\bf C} = \big(\text{\bf B}, \Om^{(\bullet,\bullet)}\big)$ is of {\em  Gelfand type} if $\Om^{(0,\bullet)}$ is a graded multiplicity-free  left $A$-comodule. We say that a CQH-Hermitian space is of {\em Gelfand type} if its constituent CQH-complex space is of Gelfand type.
\end{defn}


Since the  Laplacian is a self-adjoint operator, we already know that it is diagonalisable. However, since it is a degree $0$ operator, Corollary \ref{lem:DiagComoduleMaps}  implies the following stronger result. 

\begin{lem} \label{lem:GelfandDiag}
For a CQH-Hermitian space of Gelfand type,  the Laplacian $\DEL_{\adel}$ acts on every irreducible left $A$-sub-comodule of $\Om^{(0,\bullet)}$ as a scalar multiple of the identity.
\end{lem}

We now use the various symmetries of the Dolbeault double complex to find four equivalent formulations of Gelfand type. We begin with the symmetry induced by the $*$-map of the calculus.

\begin{lem} \label{lem:FOURGTS}
For ${\bf C} = \big(\text{\bf B},  \Om^{(\bullet,\bullet)}\big)$ a CQH-complex space, the following conditions are equivalent:
\bet

\item ${\bf C}$ is  of Gelfand type,

\item $\text{\bf C}^{\mathrm{op}}$, the opposite CQH-complex space,  is of Gelfand type,

\item the graded left $A$-comodule  $\Om^{(\bullet,0)}$ is graded multiplicity-free.



\eet
\end{lem}
\begin{proof}
~~ 
\bet

\item[$1 \Leftrightarrow 3$] The image  of an irreducible $A$-sub-comodule of $\Om^{\bullet}$ under the $*$-map is again an irreducible $A$-sub-comodule. Thus if $\Om^{(0,k)} \simeq \bigoplus_{\a} \Om^{(0,k)}_{\a}$ denotes a decomposition of $\Om^{(0,k)}$ into irreducible $A$-sub-comodules, then a decomposition of  $\Om^{(k,0)}$ into irreducible $A$-sub-comodules is given by 
\begin{align*}
\Om^{(k,0)} \simeq \bigoplus_{\a} \big(\Om^{(0,k)}_{\a}\big)^*.
\end{align*}
Thus we see that $\Om^{(k,0)}$ is multiplicity-free if and only if $\Om^{(0,k)}$ is multiplicity-free, which is to say $\text{\bf C}$ is  of Gelfand type  if and only if $\text{{\bf C}}^{\mathrm{op}}$ is of Gelfand type.

\item[$2 \Leftrightarrow 3$] The equivalence of $2$ and $3$ follows directly from the definition of  opposite complex structure and the definition of Gelfand type.  \qedhere

\eet
\end{proof}

If we additionally assume the existence of a covariant Hermitian structure, then the symmetries induced by the Hodge map imply two additional equivalent formulations of Gelfand type. Note that in this case the Gelfand condition can be verified on any of the outer edges of the Hodge diamond.

\begin{lem}\label{lem:HermitianDolbeaultSymm} For ${\bf H} = \big(\text{\bf B},  \Om^{(\bullet,\bullet)}, \s\big)$  a CQH-Hermitian space, 
the following conditions are equivalent:
\bet

\item $\mathbf{H}$ is of Gelfand type,

\item the graded left $A$-comodule  $\Om^{(\bullet,n)}$ is graded multiplicity-free,

\item the graded left $A$-comodule $\Om^{(n,\bullet)}$ is graded multiplicity-free.

\eet
\end{lem}
\begin{proof} ~
\bet
\item[$1 \Leftrightarrow 2$:] 
By Lemma  \ref{lem:Hodgeproperties}, the Hodge map $\ast_\s$ associated to $\s$ restricts to a left $A$-comodule isomorphism
$
\ast_\s: \Om^{(0, \bullet)} \simeq  \Om^{(\bullet,n)}.
$
Thus $ \Om^{(\bullet,n)}$ is graded multiplicity-free if and only if $\Om^{(0,\bullet)}$ is graded multiplicity-free, which is to say, if and only if ${\bf H}$  is of Gelfand type.

\item[$1 \Leftrightarrow 3$:] By Lemma \ref{lem:FOURGTS} above,  $\mathbf{H}$ is of Gelfand type,
if and only if $\Om^{(\bullet,0)}$ is graded multiplicity-free. Moreover, using the Hodge map, just as above, we see that  $\Om^{(n, \bullet)}$ is graded multiplicity-free if and only if $\Om^{(0,\bullet)}$ is graded multiplicity-free, giving the required equivalence. \qedhere
\eet
\end{proof}

\section{Drinfeld--Jimbo Quantum Groups and CQH-Hermitian Spaces} \label{section:DJCQHHS}

From now on we restrict to the case of Drinfeld--Jimbo quantised enveloping algebras $U_q(\frak{g})$ and their quantised coordinate algebras $\O_q(G)$, as presented in Appendix  \ref{APP:secnoumi}. This allows us to exploit the associated highest weight structure on the category $U_q(\frak{g})$-modules of type $1$, leading to the construction of canonical sequences of eigenvalues in $\s_P(\DEL_{\adel})$. In the next section, under suitable assumptions, we decompose $\s_P\big(\DEL_{\adel}\big)$ into unions  of such sequences allowing us to give sufficient and necessary conditions for $\s_P(\DEL_{\adel}) \to \infty$. Note that throughout this section, we assume the notation and conventions presented in Appendix  \ref{APP:secnoumi}.

\subsection{Highest and Lowest Weight Vectors} \label{section:setofHandLWs}

In this subsection we examine the behaviour of highest weight vectors in an $\O_q(G)$-comodule algebra $P$. We show that the space of highest weight elements is a multiplicative submonoid of $P$. Moreover, we show  that  in the graded multiplicity-free case, any pair of  highest weight vectors commute up to a scalar. The algebra of anti-holomorphic forms $\Om^{(0,\bullet)}$ of a covariant complex structure is then presented as a motivating example.

\begin{defn}
For any ${\mathcal Z} \in \, ^{\O_q(G)}\textrm{Mod}$, \wrt the $U_q(\frak{g})$-action induced by the dual pairing $U_q(\frak{g}) \by \O_q(G) \to \bC$, we denote 
\begin{align*}
{\mathcal Z}_{\mathrm{hw}} := &  \{ f \in {\mathcal Z} \,|\, f \text{ is a highest weight vector of } {\mathcal Z}\}, \\
 {\mathcal Z}_{\mathrm{lw}} := &  \{ f \in {\mathcal Z} \,|\, f \text{ is a lowest weight vector of } {\mathcal Z}\}.
\end{align*}
\end{defn}


\begin{lem} \label{lem:monoidtomonoid} 
For a left $\O_q(G)$-comodule \alg $P$, which is to say a monoid object in the category  $\,^{\O_q(G)}\textrm{\em Mod}$,   it holds that 
\bet 
\item $\text{\em wt}(ab) = \text{\em wt}(ba) = \text{\em wt}(a)+\text{\em wt}(b)$, ~~ for all $a,b \in P$,
\item the multiplication of $P$  restricts to the structure of a monoid  on $P_{\mathrm{hw}}$ and $P_{\lw}$.
\eet
\end{lem}
\begin{proof}
Let $a,b \in P_{\mathrm{hw}}$. For $k =1, \dots, r$, 
\begin{align*}
E_k \tr (ab)  =&  (K_k \tr a)(E_k \tr  b)  + (E_k \tr a)(1 \tr b) = 0.
\end{align*}
Moreover,  $K_k \tr (ab)   =  (K_k \tr a)(K_k \tr b) =  q^{(\mathrm{wt}(a),\alpha_k) + (\mathrm{wt}_k(b),\alpha_k)} ab$. Thus  $ab \in P_\hw$. To show that we have a monoid it remains to show that the unit of $P$ is contained in $P_{\mathrm{hw}}$. This follows directly from the properties of a dual pairing since $X \tr 1 =  \e(X) 1$.
\end{proof}

In this subsection we discuss  graded multiplicity-free comodules for  the Drinfeld--Jimbo quantum groups. This allows us to produce a collection of identities describing proportionality relations between certain highest weight vectors. For the special case of CQH-Hermitian spaces, these identities will be our main tool for calculating the spectrum of a Dolbeault--Dirac operator. The proof of the following lemma is elementary and hence omitted.

\begin{lem}\label{lem:CoPMathair} 
Let $\A \simeq \bigoplus_{i \in \bN_0} \A_i$ be a graded multiplicity-free $\O_q(G)$-comodule, and  $a,b \in \A_{\hw}$ such that
\begin{align*}
\text{\em 1. } ~ |a| = |b|, & & \text{\em 2. } ~  \text{\em wt}(a) = \text{\em wt}(b), & & \text{\em 3. } ~  b \neq 0.
\end{align*}
Then there  exists a uniquely defined scalar $C \in \bC$ such that 
$
a = C b.
$
\end{lem}

\begin{defn}
A {\em graded  $A$-comodule algebra} is a graded  $A$-comodule $P = \bigoplus_{i \in \bN_0} P_i$,  which is also an $A$-comodule algebra, such that the grading and multiplication of $P$ combine to give it the structure of a graded algebra.
\end{defn}

The following result now follows immediately from Lemma \ref{lem:CoPMathair}.

\begin{cor} \label{cor:addcommweightsplusCab}
Let $P$ be a graded multiplicity-free $A$-comodule algebra. If $c,d \in P_{\hw}$ such that $d c \neq 0$, then there exists a unique scalar $C \in \bC$ \st 
\begin{align*} 
c  d = C  \, d  c.
\end{align*}
\end{cor}
%
%

\begin{eg} \label{eg:CalcMonoid}
The simplest example of a left $\O_q(G)$-comodule \alg is a quantum homogeneous space $B = \O_q(G)^{\co(H)}$. 
As discussed in \textsection  \ref{section:QHSPs}, the highest weight monoid is always finitely generated as a monoid. However, as we will see in \textsection \ref{section:CPN} and \textsection \ref{section:QHSPs} the number of generators  has strong consequences for the complexity of the spectrum of the Dolbeault--Dirac operators constructed to a CQH-Hermitian spaces over $B$.
\end{eg}

\begin{eg} \label{eg:calcmonoid}
Each of the following objects are left $\O_q(G)$-comodule algebras 
\begin{align*}
 \Om^\bullet, & &  \Om^{(0,\bullet)}, & &  \Om^{(\bullet,0)}.
\end{align*}
Lemma \ref{lem:monoidtomonoid} implies that the respective multiplications  restrict to monoid structures on the sets 
\begin{align*}
\Om^\bullet_{\hw}, & & \Om^{(0,\bullet)}_{\hw}, & &  \Om^{(\bullet,0)}_{\hw}. 
\end{align*}
Note that by restriction of the monoid structure of $\Om^\bullet_{\hw}$, we have the following monoid actions
\begin{align*}
\Om^{(\bullet,0)}_{\hw} \times  \Om^{(\bullet,n)}_{\hw} \to \Om^{(\bullet,n)}_{\hw} & & \Om^{(0,\bullet)}_{\hw} \times  \Om^{(n, \bullet)}_{\hw} \to \Om^{(n,\bullet)}_\hw. 
\end{align*}
Moreover, restricting to highest weight forms of degree $0$, we have the following  monoid action 
\begin{align*}
B_{\hw} \times \Om^{(a,b)}_{\hw}  \to \Om^{(a,b)}_{\hw}, & & \text{ for all } (a,b) \in \bN^2_0.
\end{align*}
\end{eg}

\begin{remark}
The previous examples can be extended to a more formal general setting using the language of relative Hopf modules over comodule algebras.  Let $A$ be a  Hopf algebra, and $(P,\DEL_P)$ a left $A$-comodule algebra. A {\em relative Hopf $P$-module algebra} $N$ is a left $A$-comodule  $(N,\DEL_N)$, which is also a module  over the algebra $P$,  satisfying  the compatibility condition
\begin{align*}
\DEL_N(pn) = \DEL_P(p) \DEL_N(n), & & \text{ for all } p \in P,  n \in N.
\end{align*}
Alternatively, considering $P$ as a monoid object in the category $^A\mathrm{Mod}$, a relative Hopf $P$-module algebra is just  a module object over $P$ in the category $^A\mathrm{Mod}$. It is instructive to observe that any object in $^{\O_q(G)}_{~~~~B} \mathrm{Mod}$ is a relative Hopf $B$-module algebra. Following the same argument as in Lemma \ref{lem:monoidtomonoid}, one can now establish the following  result, formalising the actions appearing in Example \ref{eg:calcmonoid}.

\begin{lem}
For $P$ a left $\O_q(G)$-comodule algebra, and $N$ a relative Hopf $P$-module algebra,  the action of $P$ on $N$ restricts to the structure of a $P_\hw$-space 
\begin{align*}
P_\hw \times N_\hw \to N_{\hw}, & & (p,n) \mapsto pn. 
\end{align*}
\end{lem}
\end{remark}

\subsection{CQH-Complex Spaces and Leibniz Constants}

In this subsection we apply the general results of the previous subsection to the CQH-complex spaces of Gelfand type. As a result we identify a collection of constants, which we call Leibniz constants, intrinsic to the structure of the calculus. First however, we prove a useful result relating highest weight vectors and the $*$-map of a covariant $*$-calculus. 

\begin{lem} \label{lem:AntiModoidBijection}
For a CQH-complex space $\text{\bf C} = \big(\text{\bf B}, \Om^{(\bullet,\bullet)}\big)$,  the $*$-map of the constituent calculus $\Om^\bullet \in \text{\bf B}$ restricts to a bijection between $\Om^\bullet_{\hw}$ and $\Om^\bullet_{\lw}$. Moreover, the bijection is an anti-monoid map.
\end{lem}
\begin{proof}
For any $\w \in \Om^{\bullet}_{\mathrm{hw}}$, we have, for $k =1, \dots, \text{rank}(\frak{g})$, 
\begin{align*}
F_k \tr \w^* =  \big< S(F_k), \w_{(-1)}^* \big> \w_{(0)}^*
=  & \,  \ol{\la S^2(F_k)^*,\w_{(-1)}\ra} \, \w_{(0)}^* \\ 
=  & \, \left(\la S\big(S^{-3}(F_k^*)\big),\w_{(-1)} \ra \, \w_{(0)}\right)^* \\
=  & \,  \left(S^{-3}(F_k^*) \tr \w\right)^*\!.
\end{align*}
A direct calculation confirms that  $S^{-3}(F_k^*) = -q^{-2} K_k^{-2} E_k$. 
Thus, since $\w$ is by assumption a highest weight vector, we must have that  $F_k \tr \w^* = 0$. 
Analogously, it can be shown that $K_k \tr \w^* = q^{-\mathrm{wt}_k(\w)} \w^*$.
Hence, $\w$ is a lowest weight element of $\Om^{\bullet}$. The proof that $*$ sends lowest weight forms to highest weight forms is analogous. Thus since the $*$-map is an involution, it must induce a bijection between highest and lowest weight forms. Moreover, since  the $\ast$-map is an anti-algebra map, it restricts to an anti-monoid map between $\Om^\bullet_{\hw}$ and $\Om^\bullet_{\lw}$.
\end{proof}

We now restrict to the Gelfand case, beginning  with the following lemma which is a direct consequence of Corollary \ref{cor:addcommweightsplusCab}.

\begin{lem} \label{lem:FullGelfandSwap} Let $\text{\bf H} = \big(\text{\bf B}, \Om^{(\bullet,\bullet)} \big)$ be a CQH-complex space of  Gelfand type, and  $\w, \nu \in \Om^{(0,\bullet)}_{\hw}$ such that $\nu \wed \w \neq 0$. Then there exists a uniquely defined scalar $C$ satisfying 
\begin{align*}
 \w \wed \nu  = C \, \nu \wed \w .
\end{align*}
\end{lem}

A very special, but very important, case of this result are the Leibniz commutation constants presented in the corollary below. These constants play a central part in our later description of the spectrum of the Dolbeault--Dirac operator associated to a CQH-space Hermitian space.

\begin{cor} \label{cor:LeibnizConstants}
For every non-harmonic element $z \in B_{\hw}$, there exist non-zero constants $\lambda_{z}, \, \zeta_{z} \in \bC$, uniquely defined by 
\begin{align*}
\big(\del z\big)z = \lambda_{z} \, z \del z, & & \big(\adel z\big)z = \z_z  \, z \adel z.
\end{align*}
We  call $\lambda_z$, and $\z_z$,  the holomorphic Leibniz constant, and anti-holomorphic Leibniz constant, of $z$ respectively.
\end{cor}
\begin{proof}
By Lemma \ref{lem:nonzerozaction}, the product $z \adel z$ is non-zero.  Since the complex structure is of Gelfand type by assumption,  Lemma \ref{lem:FullGelfandSwap} implies the existence and uniqueness of the constants $\lambda_z$ and $\z_z$.
\end{proof}

\begin{cor}\label{cor:QLEIBNIZ}
For any $z \in B_{\hw}$, with Leibniz constant $\lambda_z$, and $l \in \bN_0$, it holds that
\bet

\item   $\del z^l = (l)_{\lambda_z} \, z^{l-1} \del z$,


\item   $\adel z^l = (l)_{\z_z} \, z^{l-1} \adel z$,


\eet
where $(l)_{\lambda_z}$ and $(l)_{\z_z}$ are quantum integers, as presented in Appendix  \ref{app:quantumintegers}.
\end{cor}
\begin{proof}
Let us assume that the required identity  holds for some $l > 1$. Then 
\begin{align*}
\del z^{l+1} = & \, \big(\del z\big) z^l + z \del z^l 
                        =  \, \lambda^{l}_{z} \, z^l \del z + (l)_{\lambda_z} z^l \del z
                        =  \, (\lambda^l_{z} + (l)_{\lambda_z}) z^l \del z
                        =  \, (l+1)_{\lambda_z} z^l \del z. 
\end{align*}
The required formula now follows by an inductive argument. The formula for $\adel z^l$ is established analogously. 
\end{proof}


In what follows, it proves very useful to have  a simple relationship between the holomorphic and anti-holomorphic Leibniz constants. While it is not clear that such a relation exists in general,  the assumption of a certain type of self-conjugacy on zero forms is enough to imply an inverse relation between $\lambda_z$ and $\z_z$.

\begin{defn}
We say that a quantum homogeneous space $B = \O_q(G)^{\co(H)}$ is {\em self-conjugate} if every irreducible sub-comodule $V \sseq B$  is a {\em $*$-closed subspace},  which is to say $V =  \{ v^* \, |\, v \in V\}$.
\end{defn}

The following technical lemma serves as a useful means of checking $*$-invariance of an irreducible submodule $V$ in terms of the  highest weight vectors of  $V$.

\begin{lem}
For a quantum homogeneous space $B = \O_q(G)^{\co(H)}$, and an element $z \in B_{\hw}$, the irreducible sub-comodule $U_q(\frak{g})z$ is a $*$-closed subspace if and only if $z^* \in U_q(\frak{g})z$.
\end{lem}
\begin{proof}
Since $z$ is a highest weight vector of the irreducible comodule $U_q(\frak{g})z$, for every $v \in U_q(\frak{g})z$, there exists an $X \in U_q(\frak{g})$, such that $X \tr z = v$. Note next that 
\begin{align*}
v^* = (X \tr z)^* = \big(\la S(X), z_{(1)} \ra z_{(2)} \big)^* = \ol{\big< S(X), z_{(1)} \big>} z_{(2)}^*.
\end{align*}
Recalling now that we have a dual pairing of Hopf $*$-algebras, we see that 
\begin{align*}
\ol{\la S(X), z_{(1)} \ra} z_{(2)}^* = &  \big< S^2(X)^*, z_{(1)}^* \big> z_{(2)}^* \\
                                           = & \big< S^{-2}(X^*), z_{(1)}^* \big> z_{(2)}^* \\
                                           = & \big< S\big(S^{-3}(X^*)\big), z_{(1)}^* \big> z_{(2)}^*\\
                                           = & \, S^{-3}(X^*) \tr z^*. 
\end{align*}
Thus if we assume that $z^* \in U_q(\frak{g})z$, then we necessarily have $v^*  \in U_q(\frak{g})z$, for all elements $v~\in~U_q(\frak{g}) z$, implying that $U_q(\frak{g})z$ is $*$-closed. The opposite implication is obvious.
\end{proof}

We finish by showing that the assumption of self-conjugacy does indeed imply an inverse relation between the Leibniz constants $\lambda_z$ and $\z_z$.

\begin{prop} \label{prop:selfconjLeibdual}
Let $\text{\bf H}$  be a self-conjugate CQH-Hermitian space of Gelfand type. For any $z \in B_{\hw}$ with  real Leibniz constants,  it holds that 
\begin{align*}
\big(\adel z\big) z = \lambda_z \inv z \adel z,   & & \text{ or equivalently,  } & & \z_z = \lambda_z \inv.
\end{align*}
\end{prop}
\begin{proof}
Applying the $*$-map to the identity $\big(\del z\big) z = \lambda_z z \del z$ gives us the new identity 
\bal \label{eqn:starLeib}
 \lambda_z  \inv z^* \adel z^*  =  \big(\adel z^* \big) z^*.
\eal
Lemma \ref{lem:AntiModoidBijection}, combined with our assumption that  $\text{\bf H}$ is self-conjugate, implies that $z^*$ is a lowest weight vector of the irreducible module $U_q(\frak{g}) z$. Thus there exists an $X \in U_q(\frak{g})$ such that $X \tr z^* = z$. Note next that 
\begin{align*}
X^2 \tr \big(z^* \adel z^* \big) = \big(X \tr z^* \big)\adel \big(X \tr z^*\big) = z \adel z,
\end{align*} 
and that analogously, $X^2 \tr \big((\adel z^*)z^*\big) = (\adel z) z$. Thus applying $X^2$ to both sides of (\ref{eqn:starLeib}) gives  the required identity
$
 (\adel z) z = \lambda_z \inv z \adel z.
$
\end{proof}

\subsection{Laplacian Eigenvalues for CQH-Hermitian Spaces of Gelfand Type}

In the final subsection of this section we show that for a CQH-Hermitian spaces of Gelfand type, Hodge decomposition is a decomposition of $B_\hw$-spaces. Combining this result with the Hodge decomposition of the Laplacian,  we compute the eigenvalues of a general sequence of eigenvectors of the form  $z^l \w$, where $z \in B_\hw$, and $\w \in \adel \Om^{(0,\bullet)}_{\hw}$. In the next section, our strategy is to decompose the point spectrum of the Laplacian $\sigma_P(\Delta_{\adel})$ into a finite union of such sequences and to use this to conclude that,  under sufficient assumptions, $\s_P\big(\DEL_{\adel}\big) \to \infty$. As discussed in \textsection \ref{section:QHSPs}, the general ideas of this section can be extended to the more general weak Gelfand setting with sufficient care. 

Note that in this subsection we make heavy use of the quantum integer notation as presented in Appendix  \ref{app:quantumintegers}.

\begin{lem}\label{lem:ABscalars}  Let $\text{\bf H}  = \big(\bf{M}, \Om^{(\bullet,\bullet)}, \s \big)$ be a CQH-Hermitian space of  Gelfand type, and $z \in B_{\hw}$. Then, for every  $\w \in \adel \Om^{(0,\bullet)}_{\hw}$,   there exists unique  scalars $A_{z,\w}, \, B_{z,\w} \in \bC$, such that 
\bet

\item   $\del z \wed \ast_\s(\w) = A_{z,\w} \, z \big(\del \circ \ast_\s (\w)\big)$,

\item $ \adel z \wed \adel^\dagger\!  \w = B_{z,\w} z \big(\adel^{\,} \adel^\dagger\! (\w)\big)$. 

\eet
\end{lem}
\begin{proof}		 ~~~
\bet
\item 
Note first that both sides of the identity are highest weight vectors. Next we see that 
\begin{align*}
K_i \tr(\del z \wed \ast_{\s}(\w)) = & \, \del \big(K_i \tr z \big) \wed \ast_{\s}\big(K_i \tr \w \big) \\
= & \, q^{(\mathrm{wt}(z),\alpha_i) + (\mathrm{wt}(\w),\alpha_i)} \del z \wed \ast_{\s}(\w).
\end{align*}
Analogously, we have that $K_i \tr\big(z (\del \circ \ast_\s (\w)\big) = q^{(\mathrm{wt}(z),\alpha_i) + (\mathrm{wt}(\w),\alpha_i)} z \big(\del \circ \ast_\s (\w)\big)$, meaning that both forms are highest weight vectors of the same weight.  Moreover,  both forms are homogeneous elements, of the same degree,  of the graded comodule $\Om^{(\bullet,n)}$. Now $\Om^{(\bullet,n)}$ is multiplicity-free by  Lemma \ref{lem:HermitianDolbeaultSymm}. Thus if $z  \big(\del \circ \ast_\s(\w)\big) \neq 0$, then the  existence of the required constant would follow from Lemma \ref{lem:CoPMathair}. By Lemma \ref{lem:HodgeOpDecomp} the form  $\ast_{\s}(\w)$ is contained in $\del^\dagger \Om^{\bullet}$, and so, $\adel \circ  \ast_\s (\w) \neq 0$. Thus it follows from Lemma \ref{lem:nonzerozaction} that the product $z \big(\del \circ \ast_\s(\w)\big) \neq 0$ as required.

\item  The proof is analogous to the proof of 1, and so, is omitted. 
\qedhere
\eet
\end{proof}

We now use the existence of the constant $A_{z,\w}$ presented in the above lemma to establish an identity needed for the proofs of  Proposition \ref{prop:HodgeClosure} and Theorem \ref{thm:TheFormula}.

\begin{cor} \label{cor:actionofadeldagger}
It holds that 
\begin{align*}
 \adel^{\dagger}\!(z^l\w) = \big(A_{z,\w} (l)_{\lambda_z}  + 1\big) z^l \adel^\dagger\!(\w).
\end{align*}
\end{cor}
\begin{proof}
From the identity  $\adel^{\dagger} = - \ast_\s \circ \del \circ \ast_\s$, we have that 
\begin{align*}
 \adel^{\dagger}(z^l\w) = & - \ast_\s \circ \del \circ \ast_\s(z^l \w) \\
                                          = &  - \ast_\s \circ \del\big(z^l \ast_\s(\w)\big)\\
                                          = &  - \ast_\s \Big((l)_{\lambda_z} z^{l-1}\del z \wed \ast_\s(\w) + z^l \big(\del \circ \ast_\s(\w)\big)\Big)\\
                                          =  &   - (l)_{\lambda_z} z^{l-1} \ast_\s \big(\del z \wed \ast_\s(\w)\big) - z^l \big(\ast_\s \circ \, \del \circ \ast_\s(\w)\big).
\end{align*}
Recalling Lemma \ref{lem:ABscalars}, we see that there exists a scalar $A_{z,\w}$ such that
\begin{align*}
 \adel^{\dagger}\!(z^l\w)   = &   - (l)_{\lambda_z} z^{l} A_{z,\w} \big(\ast_\s \circ \, \del  \circ \ast_\s(\w)\big) + z^l \adel^\dagger\!(\w)\\
                                          = &  \,  \big(A_{z,\w}(l)_{\lambda_z}  + 1\big) z^l \adel^\dagger\!(\w),
\end{align*}
which gives us the claimed identity. \qedhere
\end{proof}

With these results in hand we are now ready to show that Hodge decomposition implies  a decomposition of highest weight forms into $B_\hw$-subspaces. 

\begin{prop} \label{prop:HodgeClosure}
For any CQH-complex space of Gelfand type, with constituent quantum homogeneous space $B$, the spaces $\adel \Om^{(0,\bullet)}_{\hw}$ and $\adel^\dagger \! \Om^{(0,\bullet)}_{\hw}$ are closed under the action of the monoid $B_{\hw}$.
\end{prop}
\begin{proof}
Consider elements $\w \in \adel \Om^{(0,k)}$, and $z \in B_{\hw}$. Since Hodge decomposition is a decomposition of left $\O_q(G)$-comodules,  either $U_q(\frak{g})z^l \w \sseq \adel \Om^{(0,k)}$, or $U_q(\frak{g})z^l\w \sseq \adel^\dagger \Om^{(0,k)}$, for all $l \in \bN_0$.  In particular either $z^l \w \in \adel \Om^{(0,k)}$, or $z^l\w \in \adel^\dagger \Om^{(0,k)}$. We observe that
\begin{align*}
\adel(z^l\w) = \adel(z^l) \wed \w + z^l \adel \w = (l)_{\lambda_z \inv} z^{l-1} \adel z \wed \w = (l)_{\lambda_z \inv} z^{l-1} \adel(z\w).
\end{align*}
Thus $\adel(z^l\w) = 0$ if and only if $\adel(z\w) = 0$. This means that $z\w \in \adel \Om^{(0,\bullet)}$ if and only if $z^l \w \in \adel \Om^{(0,\bullet)}$, for all $l \in \bN_0$. 


Now, for $l=1$, this is zero if and only if $B_{z,\w} = -1$. In that case, for $l>1$, recalling that $q \in \bR$ (and hence not a complex root of unity), we have 
\begin{align*}
 \adel^{\dagger}\!(z^l\w)  = \big(-(l)_{\lambda_z} + 1\big) z^l \, \adel^\dagger\!(\w) \neq 0.
\end{align*}
However,  this contradicts our earlier observation that $z^l\w \in \adel \Om^{(0,\bullet)}$ if and only if $z\w$ is contained in $\adel \Om^{(0,\bullet)}$. Hence, we can conclude that  $\adel \Om^{(0,\bullet)}_{\hw}$ is closed under the action of $B_{\hw}$. The proof that $\adel^\dagger \Om^{(0,\bullet)}_{\hw}$ is closed under the action of $B_{\hw}$ is analogous.
\end{proof}

We now use this lemma to construct an explicit sequence of eigenvalues starting from an element  $z \in B_\hw$ and a form $\w \in  \text{\small{$\adel\Om^{(0,k)}_\text{\em hw}$}}$. In the next section, we introduce an approach to verifying the compact resolvent condition based around  such sequences of eigenvalues. The eigenvalues are presented in terms of quantum $\lambda_z$-integers, and quantum $\lambda_z\inv$-integers, where as usual $\lambda_z$ is the Leibniz constant of $z$. In the case of quantum projective space, as presented in \textsection 6, we see that eigenvalues of its Dolbeault--Dirac operator are exactly of this form, with the quantum $\lambda_z$-integers $q$-deforming the integer eigenvalues of the classical operator.

\begin{thm}\label{thm:TheFormula}
Let $\mathbf{H} := \big(\mathbf{B},\Om^{(\bullet,\bullet)},\s \big)$ be a CQH-Hermitian space of Gelfand type. For any form $\w \in  \adel\Om^{(0,k)}_\text{\em hw}$, and $z \in B_\hw$, it holds that 
\bet

\item  $z^l \w$ is an eigenvector of $\DEL_{\adel}$, \, for all $l \in \bN_0$,

\item denoting by $\mu_\w$  the eigenvalue of $\w$, it holds that  
\begin{align*}
 \DEL_{\adel}(z^l\w) = \Big(A_{z,\w}\,(l)_{\lambda_{z}}  + 1\Big)\Big(B_{z,\w} \, (l)_{\lambda_{z}\inv}+ 1\Big)\, \mu_{\w} \, z^l \w.
\end{align*}
\eet
\end{thm}
\begin{proof}
By Corollary \ref{cor:QLEIBNIZ} we have that
\begin{align*}
\adel\big(z^l \, \adel^\dagger\! \w\big) = & \, \adel(z^l) \wed \adel^\dagger\!  \w + z^l \big(\adel  \adel^\dagger\! \w \big) = (l)_{\lambda_{z}\inv} z^{l-1} \adel z \wed \adel^\dagger\!  \w + z^l \big(\adel  \adel^\dagger\! \w\big).
\end{align*}
Lemma \ref{lem:ABscalars} 
implies that there exists a  uniquely defined scalar $B_{z,\w}$ such that
\begin{align*}
 (l)_{\lambda_z \inv} z^{l-1} \adel z \wed \big(\adel^\dagger\! \w\big) + z^l \, \adel  \adel^{\dagger}\w 
                      =   & \,  B_{z,\w} (l)_{\lambda_z \inv} z^l \,\adel  \adel^{\dagger}\w + z^l \, \adel  \adel^{\dagger}\!\w \\
                      = & \, \big(B_{z,\w} (l)_{\lambda_z \inv}  +  1 \big) z^l \adel \adel^{\dagger}\w.
\end{align*}
From  Proposition \ref{prop:HodgeClosure} above, we know that $z^l \w \in \adel \Om^{(0,\bullet)}$. Moreover by Hodge decomposition of the Laplacian, we know that $\DEL_{\adel}$ restricts to $\adel   \adel^\dagger$ on  $\adel \Om^{(0,\bullet)}$.  Combining these facts with  Corollary  \ref{cor:actionofadeldagger}, we now see that
\begin{align*}
\DEL_{\adel}(z^l \w) & = \adel   \adel^\dagger\big(z^l\w\big) \\
                                  & = \big(1+A_{z,\w} (l)_{\lambda_z}\big) \adel \big(z^l \, \adel^\dagger \w \big)\\
                                  & = \big(1+A_{z,\w} (l)_{\lambda_z}\big)\big(1+ B_{z,\w} (l)_{\lambda_z \inv}\big)z^l \adel\adel^\dagger(\w)\\
                                  & = (1+A_{z,\w} (l)_{\lambda_z }) (1+ B_{z,\w} (l)_{\lambda_z \inv})z^l\DEL_{\adel}( \w) \\
                                  & = (1+A_{z,\w} (l)_{\lambda_z}) (1+ B_{z,\w} (l)_{\lambda_z \inv})\mu_{\w}  z^l \w,
\end{align*}
establishing the required identity. \end{proof}

\subsection{Connectedness  for Gelfand Type CQH-Complex Spaces}

We finish with a discussion of connectedness in the Gelfand type setting, proving that it is equivalent to finite-dimensionality of the zeroth cohomology group $H^{0}$. This is an interesting, and useful,  application of the notion of Gelfand type, especially given the difficulty of demonstrating  connectedness in general. We begin by recalling the standard definition of connectedness for a differential calculus. 

\begin{defn}
We say that a differential calculus $(\Om^\bullet,\exd)$ is {\em connected} if 
\begin{align*}
H^0 = \ker(\exd: \Om^0 \to \Om^1) = \bC 1.
\end{align*}
\end{defn}
It is important to note that if $\Om^\bullet$ is endowed with a complex structure $\Om^{(\bullet,\bullet)}$, then an elementary application of the $*$-map demonstrates that the calculus is connected if and only if
\begin{align*}
\ker(\del: \Om^{(0,0)} \to \Om^{(1,0)}) = \ker(\adel: \Om^{(0,0)} \to \Om^{(0,1)}) = \bC 1.
\end{align*}

Note that the following  lemma does not rely on our discussions above. What is used is no more than a multiplicity-free assumption for forms of degree $0$ (as implied by Gelfand type) and the assumption that we are working with Drinfeld--Jimbo quantum groups.  

\begin{lem} \label{lem:GelfandandConnectivity}
Let $\text{\bf C} = (\text{\bf B}, \Om^{(\bullet,\bullet)})$ be a CQH-complex space for which $\Om^0$ is multiplicity-free as a left $A$-comodule. Then the following are equivalent:
\bet

\item The constituent different calculus  $\Om^\bullet \in \text{\bf B}$ is connected,

\item $\dim_{\mathbb{C}}\big(H^0\big) < \infty$.

\eet  
\end{lem}
\begin{proof}
Assume that the constituent calculus $\Omega^{\bullet} \in \mathbf{B}$ is not connected, which is to say,  assume that $B$ contains a $\adel$-closed element $y$ which is not a scalar multiple of the identity.  Denote by $y = \sum_k y_k$ the decomposition of $y$ into summands which are homogeneous \wrt the decomposition of $B$ into irreducible $U_q(\frak{g})$-modules. Since $\Om^0$ is multiplicity-free by assumption, and $\adel$ is a left $U_q(\frak{g})$-module map, $\adel y = 0$ if and only if  $\adel y_k = 0$, for all $k$.  Hence, we can assume, without loss of generality, that $y$ is homogeneous \wrt  the decomposition of $B$ into irreducibles. By Schur's lemma every element in the irreducible module containing $y$ must be $\adel$-closed.  Thus we can assume, without loss of generality, that $y$ is a weight vector with non-zero weight. Since weights are additive  $K_i \tr y^l = q^{l(\mathrm{wt}(y),\alpha_i)} y^l$, for any $l \in \bN_0$, and $i = 1, \dots, \text{rank}(\frak{g})$. Thus the weights of the elements $y^l$ are distinct, for each $l$. In particular, the set $\{y^l \, |\, l \in \bN_0\}$ is linearly independent, and so, infinite-dimensional. Since $\adel y = 0$, the Leibniz rule implies that $\adel(y^l) = 0$, for all $l \in \bN_0$. This means that the space of harmonic elements is infinite-dimensional, and so, by Hodge decomposition $H^0$ is infinite-dimensional. The  proof in the other direction is trivial, meaning that we have established the required equivalence.
\end{proof}

\section{CQH-Complex Spaces of Order I and Spectral Triples} \label{section:OrderI}

In this section we introduce the notion of CQH-complex space of order I, which can be viewed as an abstraction of the essential representation theoretic properties of the space of holomorphic forms of complex projective space. Necessary and sufficient conditions are then produced for a CQH-Hermitian space to give a Dolbeault--Dirac pair of spectral triples, under the assumption that its underlying CQH-complex space is of order I.  

It is proposed in Conjecture \ref{conj:weakGelfClass} that the only irreducible quantum flag manifolds of order I  are the quantum projective spaces $\O_q(\mathbb{CP}^{n-1})$. We  formalise its properties for three principal reasons. Firstly, the abstract picture helps to clarify and elucidate the processes at work for quantum projective space. Secondly, it sets the stage for our subsequent investigation of the compact quantum Hermitian spaces of weak Gelfand type, highlighting the subtle but significant  changes that occur when passing to this more general setting. Finally, it is hoped that new examples will arise from non-Drinfeld--Jimbo quantisations of $U_q(\frak{g})$. In fact, it is important to note  that the only  essential feature of $U_q(\frak{g})$ used in this paper is the preservation under $q$-deformation of the highest weight structure of the category of $U(\frak{g})$-modules.

Note that in this subsection we make heavy use of the quantum integer notation as presented in Appendix  \ref{app:quantumintegers}.

\subsection{Positivity for Leibniz Constants}

In this subsection we give sufficient conditions for real Leibniz constants to be positive. As well as being an interesting observation in its own right, positivity must hold for any CQH-Hermitian space satisfying $\s_P(\DEL_{\adel}) \to \infty$, as we will see in  \textsection \ref{subsection:solidityDDSTs}.

\begin{lem} \label{lem:LapActionononeforms}
Let $\text{\bf H} = (\text{\bf B}, \Om^{(\bullet,\bullet)},\s)$ be a  self-conjugate CQH-Hermitian space of Gelfand type, with constituent quantum homogeneous space  $B$. For any $z \in B_\hw$, and $l \in \bN_0$,
\bal \label{eqn:zleigenvalue}
\DEL_{\adel}(z^l \adel z) =  \big(A_{z,\adel z}(l)_{\lambda_z} + 1\big)(l+1)_{\lambda_z\inv}  \mu_z  z^l,
\eal
where $\mu_z$ is the $\DEL_{\adel}$-eigenvalue of $z$, and as usual $\lambda_z$ is the Leibniz constant of $z$. 
\end{lem}
\begin{proof}
For $\mu_z$ the $\DEL_{\adel}$-eigenvalue of $z$, 
\begin{align*}
\adel z \wed \adel^{\dagger}\! \adel z = \adel z \wed  (\mu_z  z) =  \mu_z \lambda_z\inv z \adel z = \lambda_z \inv z \adel(\mu_z z) = \lambda_z \inv z \adel\adel^{\dagger}\!\adel z.
\end{align*}
Thus we see that $B_{z,\adel z} = \lambda_z \inv$. Equation (\ref{eqn:zleigenvalue}) now follows from Theorem \ref{thm:TheFormula}. \qedhere
\end{proof}

\begin{prop}  \label{prop:Leibnizpositivity}
Let $\text{\bf H} = (\text{\bf B}, \Om^{(\bullet,\bullet)},\s)$ be a connected self-conjugate CQH-Hermitian space of Gelfand type. For any non-harmonic $z \in B_{hw}$, with real Leibniz constant $\lambda_z$, it holds that
\begin{enumerate}
\item $\lambda_z \notin [-1,0)$,
\item if $A_{z,\adel z} \neq 0$, then $\lambda_z \in \bR_{>0}$.
\end{enumerate}
\end{prop}
\begin{proof} ~~
\begin{enumerate}

\item Assume that $-1 < \lambda_z < 0$. As $l \to \infty$, the sign of the scalar
\begin{align*}
(l+1)_{\lambda_z\inv} = \frac{1 - \lambda_z^{-(l+1)}}{1-\lambda_z\inv} 
\end{align*}
alternates, and its absolute value goes to infinity. Moreover, $(l)_{\lambda_z}$ is positive for all $l \in \bN_0$, implying that $A_{z,\adel z}(l)_{\lambda_z} + 1$ will eventually have a constant sign. Thus
 there exist values of $l$ for which the eigenvalue in (\ref{eqn:zleigenvalue}) is negative. However, this contradicts the fact that $\DEL_{\adel}$ is a positive operator, forcing us to conclude that $\lambda_z \notin (-1,0)$. Moreover, if  $\lambda_z = -1$, then $[2]_{\lambda_z} = 0$, implying that $z^2$  is harmonic,  contradicting our assumption that $\text{\bf H}$ is connected. Thus we must have that  $\lambda_z \notin [-1,0)$.

\item For $ \lambda_z < -1$, assuming that $A_{z,\adel z} \neq0$ allows one to produce a negative eigenvalue for $\DEL_{\adel}$, just as above.  Since this again  contradicts the positivity of $\DEL_{\adel}$, we are forced to conclude that $-1 < \lambda_z$. Taken together with the fact that $\lambda_z \notin [-1,0)$, this means that $\lambda_z \in \bR_{>0}$.   \qedhere

\end{enumerate}
\end{proof}

\subsection{CQH-Complex Spaces  of Order I}

In this subsection we introduce the notion of a CQH-complex space of order I. This collects the properties of Gelfand type and self-conjugacy together with the existence of a ladder presentation, a particularly convenient form for the decomposition of the exact anti-holomorphic forms  into irreducibles. When reading the definition below, it is important to bear Proposition \ref{prop:HodgeClosure} in mind, in particular the fact that $\adel \Om^{(0,\bullet)}$ is closed under the action of the monoid  $B_\hw$.

The definition is an abstraction of the $U_q(\frak{sl}_n)$-module structure
of the anti-holomorphic forms of complex projective form, and in particular of
the $K$-types appearing in the decomposition into irreducibles. (See \textsection
1 and \cite{V} for a more detailed discussion of Vogan's minimal $K$-types.)

\begin{defn} \label{defn:ladder}
Let  $\text{\bf C} = \big(\text{\bf B},\Om^{(\bullet,\bullet)}\big)$ be a CQH-complex space. 
A  {\em ladder presentation} for $\text{\bf C}$  is a pair $(z,\Theta)$,  where  $z \in B_{\hw}$, and $\Theta \sseq \adel \Om^{(0,\bullet)}_{\hw}$ is a finite subset of homogeneous forms (that is $\Theta = \cup_{k} \Theta_k$, where  $\Theta_k := \Theta \cap \adel \Om^{(0,k)}$) satisfying
\begin{align} \label{eqn:ladderdecomp}
\adel \Om^{(0,k)} \, \simeq \, \bigoplus_{\w \in \Theta_k} \bigoplus_{l \in \bN_0} U_q(\frak{g})\,z^l \w.
\end{align}
A  ladder presentation $(z,\Theta)$ is said to be {\em real} if the Lefschetz constant of $z$ is a real number, which is to say, if $\lambda_z \in \bR$. Moreover,  $(z,\Theta)$ is said to be {\em positive} if the Lefschetz constant of $z$ is a positive  real number, which is to say, if $\lambda_z \in \bR_{>0}$.
\end{defn}

It is instructive to note that in the Hermitian case we have the following implication of the existence of  a ladder presentation. 

\begin{lem}
Let $\text{\bf H} = (\text{\bf B}, \Om^{(\bullet,\bullet)}, \s)$ be a CQH-Hermitian space  of Gelfand type. If $(z,\Theta)$ is a ladder presentation of $(\text{\bf B},\Om^{(\bullet,\bullet)})$, then 
\begin{align*}
\adel^{\dagger}\! \Om^{(0,k)}    \simeq   \bigoplus_{\w \in \Theta_k} \bigoplus_{l \in \bN_0} U_q(\frak{g}) \,z^l  \adel^{\dagger}\!\w,  & & \text{ for all $k$. }
\end{align*}
\end{lem}
\begin{proof}
Operating by $\adel^{\dagger}$ on the decomposition (\ref{eqn:ladderdecomp}) associated to $(z,\Theta)$ gives 
\begin{align*}
\adel^{\dagger}  \adel \Om^{(0,k)} \simeq    \bigoplus_{\w \in \Theta_k} \bigoplus_{l \in \bN_0} U_q(\frak{g}) \, \adel^{\dagger} \big(z^l \w\big), & & \text{ for all $k$. }
\end{align*}
By Hodge decomposition  
\begin{align*}
\adel^\dagger  \adel \, \Om^{(0,k)} = &\, \adel^\dagger  \adel\Big( \adel \Om^{(0,k-1)} \oplus \adel^\dagger \!\Om^{(0,k+1)} \oplus \H^{(0,k)}\Big)
         =  \,  \adel^\dagger  \adel\big(\adel^\dagger \Om^{(0,k+1)}\big)
        =  \, \adel^\dagger \!\Om^{(0,k)},
\end{align*}
where in the last identity we have used the fact that $\adel^\dagger \adel: \adel^\dagger\!\Om^{(0,\bullet)} \to \adel^\dagger\!\Om^{(0,\bullet)}$ is an isomorphism, as established in Proposition \ref{prop:ADELISO}. Moreover, since $\text{\bf H}$ is assumed to be of Gelfand type, and  $\adel^{\dagger}\!(z^l \w)$ and $z^l \, \adel^{\dagger}\!\w$ are clearly highest weight elements of the same degree, it follows from Lemma \ref{lem:CoPMathair} that they are linearly proportional. Thus
\begin{align*}
\adel^{\dagger}\! \Om^{(0,k)}    \simeq   \bigoplus_{\w \in \Theta_k} \bigoplus_{l \in \bN_0} U_q(\frak{g}) \, z^l \, \adel^{\dagger}\!\w,
\end{align*}
as claimed. \qedhere
\end{proof}

Finally, we come to the definition of CQH-complex space of order I. As stated above, this can be viewed as a noncommutative generalisation of projective space. More explicitly, it can be viewed as a noncommutative generalisation of the properties of $\ccpn$ considered as the homogeneous space $SU_n/U_{n-1} \simeq U_n/(U_{n-1} \by U_1)$.

\begin{defn} \label{defn:orderI}
A CQH-complex space $\text{\bf C}$ is said to be of  {\em order I} if it is 
\bet
\item self-conjugate, 
\item of Gelfand type, 
\item admits a real ladder presentation.
\eet
\end{defn}

Note that if  $\text{\bf C}$ is of order I, properties 1 and 2 of the definition together imply that every highest weight space contains an element of the form $z^l \w$, for some $l \in \bN_0$, and some $\w \in \Theta$.  The notion of an order II presentation, which deals with CQH-spaces of weak Gelfand type, is discussed in \textsection \ref{section:QHSPs} along with the  motivating quantum flag manifolds.

\subsection{CQH-Hermitian Spaces of Order I and Solidity}

In this subsection we introduce a crucial property, which we call solidity, for those CQH-Hermitian spaces whose underlying CQH-complex spaces are of order I. In the next subsection, it will be shown that  solidity is equivalent to the Laplace operator having eigenvalues tending to infinity.

\begin{defn} Let $\text{\bf H} = (\text{\bf B},\Om^{(\bullet,\bullet)},\s)$ be a CQH-Hermitian space.
\bet

\item We say that $\text{\bf H}$ is {\em solid$^{\,0}$} if $(\text{\bf B},\Om^{(\bullet,\bullet)})$  admits a real ladder presentation $\big(z,\Theta \big)$ and satisfies  
\begin{align} \label{eqn:solidzeroorderI}
\text{(a)} ~~    \lambda_z = 1, & & \text{(b)} ~~ A_{z,\w} \neq 0,  ~~ \text{ or }  B_{z,\w} \neq 0. ~~~~~~~~~~~~~~~~~~~~~~~\,
\end{align}

\item We say that $\text{\bf H}$ is {\em solid$^{+}$} if $(\text{\bf B},\Om^{(\bullet,\bullet)})$  admits a real ladder presentation $\big(z,\Theta \big)$ and satisfies 
\begin{align} \label{eqn:solidplusorderI}
 \text{(a)} ~~  1 <  \lambda_z, & & \text{(b)} ~~ A_{z,\w} \neq 0, & & \text{(c)} ~~ B_{z,\w} \neq \lambda_z - 1.
\end{align}

\item We say that $\text{\bf H}$ is {\em solid$^{-}$} if $(\text{\bf B},\Om^{(\bullet,\bullet)})$ admits a real ladder presentation $\big(z,\Theta \big)$ and satisfies 
\bal \label{eqn:solidminusorderI}
~~~ \text{(a) } ~~ 0 <  \lambda_z < 1, & & \text{(b)} ~~ A_{z,\w} \neq \lambda_z - 1, & & \text{(c)} ~~ B_{z,\w} \neq 0.
\eal

\item We say that  $\text{\bf H}$ is {\em solid} if it is either solid$^{\,0}$, solid$^{+}$, or solid$^{-}$.
\eet
\end{defn}

In practice  it can be difficult to directly verify condition (b) in  (\ref{eqn:solidzeroorderI}) and conditions (b) and (c) in (\ref{eqn:solidplusorderI}) and (\ref{eqn:solidminusorderI}). The following two lemmas give more convenient  reformulations of condition (b) in (\ref{eqn:solidplusorderI}) and condition (b) in (\ref{eqn:solidminusorderI}).  As is easily confirmed, the proof extends to  analogous reformulations of condition (c) in (\ref{eqn:solidplusorderI}) and condition (c) in (\ref{eqn:solidminusorderI}).

\begin{lem} \label{lem:ANonPrimitive}
Let $\text{\bf H}$ be  a  CQH-Hermitian space of order I, and $(z,\Theta)$ a ladder presentation of $\text{\bf H}$. For any $\w \in \Theta$, the following are equivalent:
\bet 
\item $A_{z,\w} \neq 0$,

\item  $\del z \wed \w$ is a non-primitive form.

\eet
\end{lem}
\begin{proof}
From the defining formula of the Hodge map, and centrality of the Hermitian form, for any $\w \in \Om^{(0,k)}$, we have
\begin{align*}
\del z \wed \ast_{\s}(\w) = &  \,  (-1)^{\frac{k(k+1)}{2}} i^{-k} \frac{1}{(n-k)!} \del z \wed L^{n-k}(\w)\\
                                         = & \, (-1)^{\frac{k(k+1)}{2}} i^{-k} \frac{1}{(n-k)!} \del z \wed \s^{n-k} \wed \w\\
                                          = & \, (-1)^{\frac{k(k+1)}{2}} i^{-k} \frac{1}{(n-k)!} \s^{n-k} \wed \del z \wed  \w\\
                                         = & \, (-1)^{\frac{k(k+1)}{2}} i^{-k} \frac{1}{(n-k)!}L^{n-k}\big(\del z \wed \w).
\end{align*}
Thus we see that $\del z \wed \ast_{\s}(\w)  \neq 0$ if and only if $L^{n-(k+1)+1}\big(\del z \wed \w) \neq 0$, which is to say, if and only if $\del z \wed \w$ is a non-primitive form.
\end{proof}

\begin{lem} \label{lem:BETALminus1}
Let $\text{\bf H}$ be a  CQH-Hermitian space of order I, and $(z,\Theta)$ a real ladder presentation of $\text{\bf H}$. If $A_{z,\w} \neq 0$, then for any $\w \in \Theta$, the following  are equivalent: 
\bet 

\item $A_{z,\w} \neq \lambda_z - 1$,

\item  $\w \wed \del z$ is a non-primitive form.

\eet
\end{lem}
\begin{proof}
By Lemma \ref{lem:nonzerozaction},  both $\ast_{\s}(\w)z$ and $z \left(\del \circ \ast_{\s}(\w)\right)$ are non-zero forms. Since both forms have the same weight and degree, Lemma \ref{lem:HermitianDolbeaultSymm} and  Lemma \ref{lem:CoPMathair} imply the existence of scalars 
$C, C' \in \bC$   such that 
\bal \label{eqn:zwCwzswap}
z \ast_{\s}\!(\w) = C  \ast_{\s}(\w) z, & & \left(\del \circ \ast_{\s}(\w)\right) z = C' z \left(\del \circ \ast_{\s}(\w)\right). 
\eal
Combining these two identities with Lemma \ref{lem:ABscalars} we see that  
\begin{align*}
\del z \wed z \ast_{\s}\!(\w) = C (\del  z \wed \ast_{\s}\!(\w))z = C A_{z,\w} z (\del \circ \ast_{\s} (\w)) z =  C C'  A_{z, \w}  z^2 (\del \circ \ast_{\s}(\w)). 
\end{align*}
Moreover,  
\begin{align*}
\del z \wed z \ast_{\s}(\w) = \lambda_z z (\del z \wed \ast_{\s}(\w)) = \lambda_z A_{z, \w} z^2 (\del \circ \ast_{\s}(\w)).  
\end{align*}
Since  $A_{z, \w} \neq 0$ by assumption,  comparing the two expressions for $\del z \wed z \ast_{\s}(\w)$ gives us that 
\bal \label{eqn:CzwCdwzlz}
C   C'  = \lambda_z.
\eal
Returning now to (\ref{eqn:zwCwzswap}), we apply $\del$ to the  first  identity to obtain
\begin{align*}
\del z \wed \ast_{\s}(\w) + z(\del \circ \ast_{\s}(\w)) = C  \big(\del \circ  \ast_{\s} \w\big) z  + (-1)^{|\ast_{\s}(\w)|} \,C \ast_{\s}(\w) \wed \del z.
\end{align*}
Now by (\ref{eqn:zwCwzswap}) and (\ref{eqn:CzwCdwzlz}) it holds that $C \big(\del  \circ \ast_{\s}(\w)\big)  z = CC' z  (\del \circ  \ast_{\s} (\w))  = \lambda_z z (\del \circ  \ast_{\s}(\w))$. Hence  
\begin{align*}
 \del z \wed \ast_{\s}(\w) = (\lambda_z - 1) \, z (\del \circ \ast_{\s}(\w))  + (-1)^{|\ast_{\s}(\w)|}\,C\!\ast_{\s}\!(\w) \wed \del z.
\end{align*} 
Thus we see that $\ast_{\s}(\w) \wed \del z = 0$ if and only if $\del z \wed \ast_{\s}(\w) = (\lambda_z - 1)z (\del \circ  \ast_{\s}(\w))$.

Following the same argument as given in Lemma \ref{lem:ANonPrimitive} above, it can be shown that the form $\ast_{\s}(\w) \wed \del z$ is non-zero if and only if the form $\w \wed \del z$  is non-primitive. The claimed equivalence now follows.
\end{proof}

\subsection{Solidity and Dolbeault--Dirac Spectral Triples} \label{subsection:solidityDDSTs}

In this subsection, we show that for a CQH-Hermitian space of order I, the eigenvalues of its Laplacian tend to infinity if and only if it is solid. Combining this equivalence with Proposition \ref{prop:BCCQHHS} we can give necessary and sufficient 
conditions for a CQH-Hermitian space to give a Dolbeault--Dirac pair of spectral triples. For the convenience of the reader we break the proof into two parts.

\begin{lem}\label{lem:GTimpliesSolid}
The Laplacian of a CQH-Hermitian space $\text{\bf H}$ of order I has eigenvalues tending to infinity only if $\text{\bf H}$ has finite-dimensional anti-holomorphic cohomologies and all ladder presentations are solid.
\end{lem}
\begin{proof}
If the eigenvalues of $\DEL_{\adel}$ tend to infinity then by definition no eigenvalue can have infinite multiplicity, and in particular $0$ cannot have infinite multiplicity.  By Hodge decomposition this  is equivalent to the complex structure having finite-dimensional anti-holomorphic cohomologies, giving us one part of the implication. 

Next we show that $\lambda_z \geq 0$. Since we have just shown that $\dim_{\mathbb{C}}\!\left(H^{(0,0)}\right) < \infty$, Proposition \ref{lem:GelfandandConnectivity} implies that the calculus is connected. Thus Proposition \ref{prop:Leibnizpositivity} tells us that either $\lambda_z > 0$, or $\lambda_z < -1$ and  $A_{z,\adel z} = 0$. In the latter case, Lemma \ref{lem:LapActionononeforms} implies that
\begin{align*}
\DEL_{\adel}\left(z^l \adel z\right) = (l+1)_{\lambda_z\inv} z^l\adel z.
\end{align*}
Since $(l-1)_{\lambda_z\inv}$ converges as $l \to \infty$, in this case we cannot have that  $\s_P(\DEL_{\adel}) \to \infty$. Thus we are forced to conclude that $\lambda_z$ is positive.

For a general ladder presentation $(z,\Theta)$, we now have three possibilities: $0 < \lambda_z < 1$, $\lambda_z = 1$, or $1 < \lambda_z$. 

1. Let us assume that $\lambda_z = 1$. If $A_{z,\w}$ and $B_{z,\w}$ are both zero, for some $\w \in \Theta$, then it follows directly from Theorem \ref{thm:TheFormula} that $\mu_\w$ is an eigenvalue of $\DEL_{\adel}$ with infinite-dimensional multiplicity. Thus if $\lambda_z = 1$, then $(z,\Theta)$ must be solid$^0$.

2. Let us assume that  $1 < \lambda_z$, and show that if condition (b) or (c) of the definition of solid$^+$ does not hold for $\text{\bf H}$, then $\s_P(\DEL_{\adel})$ has a limit point. We start with condition (b), which is to say,  let us assume that $A_{z,\w} = 0$, for some  $\w \in \Theta.$ It follows directly from Theorem \ref{thm:TheFormula} that, in this case,
\begin{align*}
\DEL_{\adel}(z^l \w) = \big(1 + B_{z,\w}(l)_{\lambda_z\inv}\big)z^l\w, & & \text{ for all } l \in \bN_0.
\end{align*} 
Looking now at the limit of this sequence of eigenvalues, we see that
\begin{align*}
 \lim_{l \to \infty}\big(1 + B_{z,\w} (l)_{\lambda^{-1}_z}\big)
                                                =  & \, 1 + B_{z,\w}  \lim_{l \to \infty} \big((l)_{\lambda^{-1}_z}\big)\\
                                                =  & \, 1+ B_{z,\w} \lim_{l\to \infty}\Big(\frac{1- \lambda^{-l}_z}{1-\lambda_z\inv}\Big)\\
                                                = & \, 1 + \frac{B_{z,\w}}{1-\lambda_z\inv}. 
\end{align*}
Thus  $\s_P(\DEL_{\adel})$ has a limit point in $\bR$.

\noindent Let us next assume that condition (c) of the definition of solid$^+$ does not hold,   which is to say, let us assume that $B_{z,\w} = \lambda_z\inv - 1$, for some $\w \in \Theta$.  The limit of the eigenvalues of $z^l \w$, as $l \to \infty$, is given by
\begin{align*}
& \lim_{l \to \infty}\Big(\big(1+A_{z,\w} (l)_{\lambda_z}\big)\big(1 + B_{z,\w}(l)_{\lambda_z\inv}\big)\Big) \\ = &  \lim_{l\to \infty}\left[\big(1+A_{z,\w} (l)_{\lambda_z}\big)\left(1-(1-\lambda_z\inv)\frac{\,1-\lambda_z^{-l}}{1-\lambda_z\inv}\right)\right]\\
= & \, \lim_{l\to \infty}\bigg[\left(1 + A_{z,\w}\frac{\,1-\lambda_z^l}{~\,1-\lambda_z\inv}\right)\lambda_z^{-l}\bigg] \\
\end{align*}

\begin{align*}
= & \, \lim_{l\to \infty}\bigg(\lambda^{-l}_z + A_{z,\w}\frac{\,\lambda^{-l}_z - 1}{~\,1-\lambda_z\inv}\bigg)\\
= & \, \,  \frac{A_{z,\w}}{\lambda_z\inv - 1}.
\end{align*}
Thus the point spectrum of $\DEL_{\adel}$ again has a limit point in $\bR$. 

\noindent Taking these two results together we see that if either of the requirements for a solid$^+$ fail to hold for $(z,\Theta)$, then the eigenvalues of $\DEL_{\adel}$ do not tend to infinity.  Thus we are forced to conclude that $(z,\Theta)$ is solid$^+$.

3.  Finally, for the case of $0 < \lambda_z < 1$,  an argument analogous to that in 2 verifies that if $\s_P\big(\DEL_{\adel}\big) \to \infty$, then $(z,\Theta)$  must be solid$^-$. \qedhere
\end{proof}

\begin{thm} \label{thm:THETHEOREM}
Let $\text{\bf H}$ be a CQH-Hermitian structure of order I.  Then the following are equivalent:
\bet
\item  $\s_P(\DEL_{\adel}) \to \infty$,

\item  $\text{\bf H}$ is solid and has finite-dimensional anti-holomorphic cohomologies.

\eet
\end{thm}
\begin{proof}
By Lemma \ref{lem:GTimpliesSolid} above, we need only show that 2 implies 1.  Since the complex structure is of order I by assumption, it is in particular of Gelfand type. Thus by Lemma \ref{lem:DiagComoduleMaps}, the Laplacian  $\DEL_{\adel}:\adel \Om^{(0,k)} \to \adel \Om^{(0,k)}$ acts on any  irreducible $U_q(\frak{g})$-submodule  $V \sseq \adel \Om^{(0,k)}$ as a scalar multiple of the identity. This scalar can be determined by letting $\DEL_{\adel}$ act on any element of $V$. In particular, it can be determined by letting $\DEL_{\adel}$ act on a highest weight vector of $V$.  Since we are assuming that  $\text{\bf H}$ admits a ladder presentation $(z,\Theta)$, every highest weight space must contain an element of the form $z^l \w$, for some $l \in \bN_0$, and some $\w \in \Theta$. Thus, if  the  eigenvalues of $z^l w$ tend to infinity, for each $\w \in \Theta$, then we know that the eigenvalues of $\DEL_{\adel}: \adel \Om^{(0,\bullet)} \to \adel \Om^{(0,\bullet)}$ tend to infinity, and have finite multiplicity.  Since by  assumption we have  finite-dimensional anti-holomorphic cohomologies, Corollary \ref{cor:decomposingCTEG} says that this is sufficient to imply that  $\s_P(\DEL_{\adel}) \to \infty$. 

Let us now break the rest of the proof into the three cases where $(z,\Theta)$ is either solid$^0$, solid$^+$, or solid$^-$.
\begin{enumerate}
\item  We first assume that $(z,\Theta)$ is solid$^0$, and recall for convenience the identity from Theorem \ref{thm:TheFormula} above:
\begin{align*} 
\DEL_{\adel}\big(z^l \w\big) = \big(1+A_{z,\w} (l)_{\lambda_z}\big)\big(1+B_{z,\w}(l)_{\lambda_z\inv}\big)z^l\w.
\end{align*}
Looking at the first factor, we see that since $A_{z,\w} \neq 0$,
\begin{align*}
\lim_{l \to \infty} \left(1+A_{z,\w}(l)_{\lambda_z}\right) = & \, 1 + A_{z,\w} \lim_{l \to \infty}\left((l)_{\lambda_z}\right) = \infty.
\end{align*}
Similarly, the second factor tends to infinity, as $l \to \infty$. Thus the eigenvalues of $z^l\w$ tend to infinity, with finite multiplicity, for all $\w \in \Theta$.

\item Let us now assume that $(z,\Theta)$ is  solid$^+$. 
Just as in the solid$^0$ case, the first factor tends to infinity. Looking next at the limit of the second factor, we see
\begin{align*}
\lim_{l \to \infty}\left(1+B_{z,\w}(l)_{\lambda_z\inv}\right) = 1 + B_{z,\w}  \lim_{l \to \infty}\left(\frac{1-\lambda_z^{-(l+1)}}{1-\lambda_z\inv}\right) = 1 + \frac{B_{z,\w}}{1-\lambda_z\inv}.
\end{align*}  
Since we are assuming that $B_{z,\w} \neq \lambda\inv_z - 1$, we see that the second factor does not approach zero. Combining these two observations, we see that the eigenvalues of $z^l \w$ go to infinity, with finite multiplicity,  as $l$ goes to infinity, for all $\w \in \Theta$.

\item An analogous argument proves that the existence of a  solid$^-$  ladder presentation implies that the eigenvalues of $z^l \w$ go to infinity, with finite multiplicity,  as $l$ goes to infinity.   \qedhere
\end{enumerate}
\end{proof}

Combining the results of this section we arrive at the following theorem,  which will be used in the next section to construct spectral triples for $\O_q(\ccpn)$.

\begin{thm}\label{thm:spectraltriplesfromSOLIDandCoH}  Let $\mathbf{H} = \big(\text{\bf B}, \Om^{(\bullet,\bullet)}, \s\big)$ be  a  CQH-Hermitian  space of order I, with constituent quantum homogeneous space $B$. Then a Dirac--Dolbeault pair of spectral triples is given by 
\begin{align*}
\Big(B, L^2\big(\Om^{(0,\bullet)}\big), D_{\adel}\Big), & & \Big(B, L^2\big(\Om^{(\bullet,0)}\big), D_{\del}\Big),
\end{align*}
if and only if $\text{\bf H}$ is solid and has finite-dimensional anti-holomorphic cohomologies.
%
%
\end{thm}

\section{A Dolbeault--Dirac Spectral Triple for Quantum Projective Space} \label{section:CPN}

We are now ready to apply the general framework developed in the previous sections to our motivating example $\O_q(\ccpn)$ \cite{DijkStok,MEYER}. We begin by recalling the necessary basics about $\O_q(\ccpn)$ and its \hk calculus. We then construct an  order I presentation of the associated CQH-space. This allows us to verify the compact resolvent condition by demonstrating solidity, and hence produce a Dolbeault--Dirac pair of spectral triples for $\O_q(\ccpn)$ with non-trivial $K$-homology class. 

\subsection{Quantum Projective Space as a Quantum Homogeneous Space} \label{subsection:CPNasaQHS}

Consider the Hopf sub\alg  $U_q(\frak{l}_{n-1}) \sseq U_q(\frak{sl}_{n})$ generated by the elements
\begin{align*}
 K_i, \, E_j, \, F_j ~~~~~~~~ \text{ for } i = 1, \ldots, n-1, \text{ and }  j = 1, \dots, n-2.
\end{align*} 
Note that  $U_q(\frak{sl}_{n-1})$ canonically embeds into $U_q(\frak{l}_{n-1})$  as a sub\algn. 
For any irreducible  $U_q(\frak{sl}_{n-1})$-module $V_{\mu}$, and any $m \in \bZ$, we see that there exists an irreducible representation of $U_q(\frak{l}_{n-1})$ on $V_{\mu}$, extending the action of $U_q(\frak{sl}_{n-1})$,  uniquely determined by
\begin{align*}
 K_{n-1} \tr v  = q^m v, & & \text{ for } v \text{ a highest weight vector in } V_{\mu}.
\end{align*}
We denote this irreducible representation by $V_{\mu}(m)$, and call $m$ the {\em weight} of $K_{n-1}$. (See \cite[\textsection 4.4]{KOS} for a more detailed discussion of the representation theory of $U_q(\frak{l}_{n-1})$.)

As standard, we use the superscript $\circ$ to denote the Hopf dual of a Hopf algebra. Dual to   the Hopf  algebra embedding $\iota:U_q(\frak{l}_{n-1}) \hookrightarrow U_q(\frak{sl}_{n})$, we have the Hopf algebra map $\iota^{\circ}: U_q(\frak{sl}_{n})^{\circ} \to U_q(\frak{l}_{n-1})^{\circ}$.  By construction $\O_q(SU_{n}) \sseq U_q(\frak{sl}_{n})^{\circ}$, and so we can consider the restriction map
\begin{align*}
\pi:= \iota^\circ|_{\O_q(SU_n)}: \O_q(SU_{n}) \to U_q(\frak{l}_{n-1})^{\circ},
\end{align*}
as well as the Hopf subalgebra 
$
\O_q(U_{n-1}) := \pi\big(\O_q(SU_{n})\big).$ 
{\em Quantum projective space} $\O_q(\ccpn)$ is  the quantum homogeneous space associated to the surjective Hopf $*$-algebra map $\pi:\O_q(SU_{n}) \to \O_q(U_{n-1})$,  which is to say, it is the space of coinvariants 
\begin{align*}
\O_q(\ccpn) = \O_q(SU_{n})^{\text{co}(\O_q(U_{n-1}))}.
\end{align*} 
A standard set of generators for $\O_q(\ccpn)$ is given by
\begin{align*}
\Big\{ z_{ij} := u^i_nS(u^n_j) \, | \, i,j = 1, \dots, n  \Big\}. 
\end{align*}

\subsubsection{The Heckenberger--Kolb Calculus}

We present the calculus in two steps, beginning with Heckenberger and Kolb's classification of first-order differential  calculi over $\O_q(\ccpn)$, and then discussing the maximal prolongation of the direct sum of the two calculi identified. 

First, we recall that a {\em differential map} between two differential calculi $(\Om^\bullet,\exd)$ and $(\Gamma^\bullet, \d)$, is a degree $0$ algebra map $\f:\Om^\bullet \to \Gamma^\bullet$  such that  $\f \circ \d =  \d \circ \f$. Moreover, a {\em first-order differential calculus} is a differential calculus of total degree $1$. We say that a first-order differential calculus is {\em irreducible} if $\Om^1$ is irreducible as a bimodule over $\Om^0$. 

\begin{thm} \cite[Theorem 7.2]{HK} \label{thm:HKClassCPN}
There exist exactly two non-isomorphic irreducible, left-covariant, finite-dimensional, first-order differential calculi  over $\O_q(\ccpn)$. 
\end{thm}

We denote these two calculi by $\Om\hol$ and $\Om\ahol$. Moreover, we denote their direct sum by $\Om^1 := \Om\hol \oplus \Om\ahol$ and call it the {\em \hk first-order differential calculus} of $\O_q(\ccpn)$. For a proof of the following lemma see \cite[Lemma 5.2]{MMF1}. 

\begin{lem}\label{basislem}
Bases of $\Phi(\Om\hol)$ and $\Phi(\Om\ahol)$ are given respectively by 
\begin{align*}
\Big\{ e^+_a := [\del z_{an}] \,  \,|\, a=1,\ldots, n-1 \Big\}, & &  \Big\{e^-_a := [\adel z_{na}]\ \,|\, i = 1,\ldots, n-1 \Big\}.
\end{align*}
Moreover, it holds that 
\begin{align*}
[\del z_{nn}] = [\adel z_{nn}] = [\del z_{ab}] = [\adel z_{ab}] = 0, & & \text{ for all \, } a,b = 1, \dots, n-1.
\end{align*}
\end{lem}

We now present the action of the generators of $U_q(\frak{l}_{n-1})$ on the basis elements of $V^{(0,k)}$. The proof is a direct calculation in terms of the dual pairing (\ref{eqn:dualpairing}) and the definition of the action (\ref{eqn:ComoduleToModule}).

\begin{lem} \label{lem:EKFBasis}
The only non-zero actions of the generators of $U_q(\frak{l}_{n-1})$ on the basis elements of $V^{(0,1)}$ are given by
\begin{align*}
E_k \tr e^-_k = q^2 e^-_{k+1}, & & K_k \tr e^-_a = q^{-\d_{k,n-1} - \d_{ka} +\d_{k,a-1}} e^-_a, & & F_k \tr e^-_{k+1} = q^{-2} e^-_k.
\end{align*}
\end{lem} 
\begin{proof}
The first identity is established by the calculation
\begin{align*}
E_{k} \tr e^-_{a} = E_{k} \tr [\adel z_{na}] = [\adel(E_{k} \tr  z_{na})]  = q^2 \d_{ka} [\adel z_{n,k+1}] = q^2 \d_{ka} e^-_{k+1}. 
\end{align*}
The other two identities are established similarly.
\end{proof}

We say that a differential calculus over an algebra $A$, of total degree greater than or equal to $2$,  {\em extends} a first-order calculus if there exists a {\em differential injection}, which is to say, an injective differential map,  from the first-order calculus to the differential calculus. Any first-order calculus admits an extension $(\Om^\bullet,\exd)$ which is {\em maximal}  in the sense that, for any other extension $(\G^\bullet,\d)$, there exists a unique surjective  differential map $\f:\Om^\bullet \to \Gamma^\bullet$. We call this extension, which is necessarily unique up to differential isomorphism, the {\em maximal prolongation} of the first-order calculus. The maximal prolongation of a covariant calculus is again covariant. (See \cite[\textsection 2.5]{MMF2} for a more detailed discussion in the notation of this paper.)

We denote the maximal prolongation of the \hk first-order calculus by $\Om^\bullet$, and call it the {\em Heckenberger--Kolb differential calculus}. Note that since $\Om^\bullet$ is covariant, it is a monoid object in {\small$\, ^{\O_q(SU_{n})}_{\O_q(\ccpn)}  \mathrm{Mod}_0$}. Consequently, since Takeuchi's equivalence is a monoidal equivalence, $\Phi\big(\Om^\bullet\big)$ is a monoid object in {\small $\,^{\O_q(U_{n-1})} \mathrm{Mod}$}. We denote the multiplication in $\Phi(\Om^\bullet)$ by $\wed$. We call a subset $I \sseq \{1,\ldots, n-1\}$ 
 {\em ordered} if $i_1 < \cdots < i_k$. For any two ordered subsets $I, J \sseq  \{1,\ldots, n-1\}$, we denote 
\begin{align*}
e^+_I \wed e^-_J := e^+_{i_1} \wed \cdots \wed e^+_{i_k} \wed  e^-_{j_1} \wed \cdots \wed e^-_{j_{l}},
\end{align*}
where $J = \{j_1,\dots,j_l\}$. We now collect some basic facts about $\Phi(\Om^\bullet)$ as an algebra. Theorem 6.4.1 was established in \cite[\textsection 3.3]{HK}. For a proof of Theorem 6.4.2 and Theorem 6.4.3  see  \cite[Proposition 5.8]{MMF2}. 

\begin{thm}\label{thm:HKBasisRels} For $\Om^\bullet$ the \hk calculus of $\O_q(\ccpn)$:
\bet
\item $\Phi(\Om^k)$ has dimension  $\binom{2n-2}{k}$, 
\item a basis of $\Phi(\Om^k)$ is given by 
\begin{align*}
\Big\{ e^+_I \wed e^-_J \, | \, I,J \sseq \{1,\ldots, n-1\} \text{ \em ordered subsets  \st } |I|+|J| = k \Big\},
\end{align*}
\item a full set of relations for $\Phi\big(\Om^{(0,\bullet)}\big)$ is given by
\begin{align*}
e^-_j \wed e^-_i + q \inv e^-_i \wed e^-_j, & & e^-_i \wed e^-_i, & & \text{  for } i,j = 1, \dots, n-1, \text{ \st } i < j.
\end{align*}

\eet
\end{thm}

We now use the given basis of $\Phi\big(\Om^\bullet\big)$ to define a complex structure for the calculus. Consider first the subspaces of $\Phi\big(\Om^\bullet\big)$  defined by 
\begin{align*}
V^{(a,b)} := \spn_\bC\Big\{ e^+_I \wed e^-_J \, | \, I,J \sseq \{1, \ldots, n-1\} \text{ ordered subsets with } |I|=a,\, |J| = b \Big\}.
\end{align*}
Two immediate consequences of the definition of $V^{(a,b)}$ are that
\begin{align*}
 \dim_{\mathbb{C}}\big(V^{(a,b)}\big)  = \binom{n-1}{a}\!\binom{n-1}{b}, & & \text{ and }& & \Phi\big(\Om^k\big) \simeq \bigoplus_{a+b = k} V^{(a,b)}, \text{ ~~ for all } k.
\end{align*}
The following proposition is implied by the presentation of \cite[\textsection 3.3.4]{HKdR}. (Alternatively,  see \cite[\textsection 6, \textsection7]{MMF2} for a direct proof in the notation of this paper.)

\begin{thm} 
For $\Om^\bullet$ the Heckenberger--Kolb calculus over $\O_q(\ccpn)$, there is a unique covariant complex structure  $\Om^{(\bullet,\bullet)}$ on $\Om^\bullet$ \st 
\begin{align*}
 \Phi(\Om^{(a,b)}) = V^{(a,b)}.
\end{align*}
\end{thm}

As shown in \cite[Lemma 4.17]{MMF3} there exists a left $\O_q(SU_n)$-coinvariant form $\k \in \Om^{(1,1)}$ uniquely  defined by 
\bal \label{eqn:TheKahlerForm}
[\k] = 
  i \sum_{a=1}^{n-1}   e^-_a \wed   e^+_a =  i \sum_{a=1}^{n-1}   (K^{a}_a \tr e^+_a) \wed  (K^{-a}_a \tr e^-_a).
\eal
This $(1,1)$-form is a direct $q$-deformation of the fundamental form of the classical Fubini--Study K\"ahler metric on complex projective space. As the following theorem shows, much of the associated K\"ahler geometry survives intact under $q$-deformation, see \cite[\textsection 4.5]{MMF3} for details.

\begin{thm} \label{thm:CPNKahler}  Let $\Omega^{\bullet}$ be the Heckenberger--Kolb calculus of  quantum projective space $\O_q(\mathbb{CP}^{n-1})$.

\bet

\item The pair $\big(\Om^{(\bullet,\bullet)}, \k\big)$ is a covariant Hermitian structure for $\Om^\bullet$.

\item Up to real scalar multiple of $\kappa$, it is the unique covariant Hermitian structure for the calculus $\Om^\bullet$.

\item The pair $\big(\Om^{(\bullet,\bullet)}, \k\big)$ is a K\"ahler structure for $\Om^\bullet$.

\item There exists an open interval $I \sseq \bR$, containing  $1$, such that  the associated metric $g_{\k}$ is positive definite, for all $q \in I$.

\eet
\end{thm}

Thus we see that, for every $q\in I$, quantum projective space $\O_q(\mathbb{CP}^{n-1})$, endowed with the Heckenberger--Kolb calculus, is a CQH-K\"ahler space with an associated pair of Dolbeault--Dirac BC-triples. 

We finish with some results on the anti-holomorphic cohomology of the \hk \linebreak calculus of $\O_q(\ccpn)$. Taken together, these two results determine the anti-holomorphic Euler characteristic of the calculus.

\begin{thm}[\cite{KKCPN} Corollary 4.2] The Heckenberger--Kolb calculus $\Om^\bullet_q(\ccpn)$ is {\em connected}, which is to say $\H^0 = \bC 1$. 
\end{thm}

\begin{thm}[\cite{CPNLD} Proposition 7.2]   \label{thm:cpncohomology}
For the unique covariant complex structure $\Om^{(\bullet,\bullet)}$ of the Heckenberger--Kolb calculus $\Om^\bullet_q(\ccpn)$, it holds that 
\begin{align*}
H^{(0,k)} = 0, & &   \text{ for all } k = 1, \dots, n-1.
\end{align*}
\end{thm}

\begin{cor}
The anti-holomorphic Euler characteristic of $\Om^{(\bullet,\bullet)}$ is given by
\begin{align*}
\chi_{\adel}\big(\Om^\bullet\big) = \sum_{i=0}^{n-1} (-1)^i \dim_{\mathbb{C}}\big(H^{(0,i)}\big) = \dim_{\mathbb{C}}\big(H^{(0,0)}\big) = 1.
\end{align*}
\end{cor}

In  Corollary \ref{cor:LADDERSandVanishingcoho} below, it is observed that the vanishing of higher cohomologies  can alternatively  be concluded from our given  ladder presentation  of the complex structure of $\Om^\bullet_q(\ccpn)$. 

\begin{remark}
Most of the results of this subsection have been extended to the more general setting of the quantum Grassmannians (see \textsection 7 for a definition of quantum Grassmannnians). The existence of a noncommutative K\"ahler structure was established in \cite{KOS, MarcoConj}. Connectedness of the calculus was extended to the case of the quantum Grassmannians in \cite{KMOS}. The vanishing of higher cohomologies for the quantum Grassmannians was established in \cite{OSV} using a noncommutative generalisation of the Kodaira vanishing theorem, formulated within the framework of noncommutative K\"ahler and Fano structures. Thus the anti-holomorphic Euler characteristic of the quantum Grassmannian is again $1$. 
\end{remark}


\subsection{An Order I Presentation of the CQH-Complex Space} \label{subsection:OrderICPN}

In this subsection we introduce a distinguished family of highest weight forms $\nu_k \in \adel \Om^{(0,k)}$,  for $k=1,\dots, n-2$, and use them to  produce a ladder  presentation (recall Definition \ref{defn:ladder}) of the complex structure in Theorem \ref{thm:OrderIPresentation}. From this we then conclude in Corollary \ref{cor:cpnorderI} that $\O_q(\mathbb{CP}^{n-1})$ is an order I CQH-complex space.
We begin by identifying a highest weight element of $\O_q(\ccpn)$, and then calculate its associated Leibniz constants. 

\begin{lem} \label{lem:Leibconstant} 
The element $z_{1n}$ is a highest weight vector of $\O_q(\mathbb{CP}^{n-1})_{\hw}$, with weight 
\begin{align*}
\text{\em wt}(z_{1n}) = \varpi_1 + \varpi_{n-1}. 
\end{align*}
Moreover, it holds that 
\bal \label{eqn:Leibconstant}
\big(\del z_{1n}\big)z_{1n} = q^{2} z_{1n} \del z_{1n}, & &  \big(\adel z_{1n}\big)z_{1n} = q^{-2} z_{1n} \adel z_{1n}.
\eal
\end{lem}
\begin{proof}
The fact that $z_{1n} \in \O_q(\mathbb{CP}^{n-1})_{\hw}$, with the given weight, is a direct consequence of the definition of the action $\tr$ and the dual pairing, as presented in Appendix \ref{subsection:Oq(G)} and Appendix  \ref{subsection:app:SU}. To establish (\ref{eqn:Leibconstant}) we use the unit $\unit: \Om^{(1,0)} \simeq \O_q(SU_n)\, \square_{\,\O_q(U_{n-1})} \Phi\left(\Om^{(1,0)}\right)$ of Takeuchi's equivalence. Note first that
\begin{align*}
\unit\Big(\big(\del z_{1n}\big)z_{1n}\Big)= & \,\unit\big(\del z_{1n}\big)z_{1n} \\
= & ~ \proj_{\Om^{(1,0)}} \bigg(\Big(\sum_{a,b=1}^n u^1_aS(u^b_n)\, \square_{\,\O_q(U_{n-1})} [\del z_{ab}]\Big) z_{1n}\bigg)\\
= & ~~~\, \sum_{a = 1}^{n-1} u^1_aS(u^n_n)z_{1n} \, \square_{\,\O_q(U_{n-1})} [\del z_{an}].
\end{align*}
From the defining relations of $\O_q(SU_n)$ given in Appendix \ref{subsection:app:SU}, it is easy to conclude the identity   
\bal \label{eqn:uS(u)swap}
u^1_aS(u^n_n)z_{1n} = q^{2} \,z_{1n} u^1_aS(u^n_n), & & \text{ for all } a \neq n.
\eal
Hence, as claimed,
\begin{align*}
 \unit\Big(\big(\del z_{1n}\big)z_{1n}\Big)=   & \, q^{2}  z_{1n} \sum_{a = 1}^{n-1} u^1_aS(u^n_n) \, \square_{\,\O_q(U_{n-1})} [\del z_{an}] =  \,q^{2}\, \unit\Big(z_{1n} \del z_{1n}\Big).  
\end{align*}
The second identity  comes from Lemma \ref{prop:selfconjLeibdual} and the fact that $\O_q(\ccpn)$ is self-conjugate, as established in Corollary \ref{cor:selfconjCPN}. Alternatively, it can be calculated directly just as for the first identity.
\end{proof}

In the $q$-deformed setting it is no longer guaranteed that $\w \wed \w = 0$, for all forms $\w$. However, as the following lemma shows,  this identity does hold true for forms of  type  $\adel z_{1j}$.

\begin{cor}\label{cor:delzdelzzero}
It holds that
\bet
\item $\adel z_{1j} \wed \adel z_{1,j-1} = -q\inv \adel z_{1,j-1} \wed \adel z_{1j}$,  for all $j= 3, \dots, n$,
\item $\adel z_{1j} \wed \adel z_{1j} = 0, \text{ for all } j = 2, \dots, n.$
\eet
\end{cor}
\begin{proof}
By Corollary \ref{cor:QLEIBNIZ} 
\begin{align*}
0 = \adel^2\big(z_{1n}^2\big) = (2)_{q^{-2}} \adel\big(z_{1n} \adel z_{1n}\big) = (2)_{q^{-2}} \adel z_{1n} \wed \adel z_{1n}.
\end{align*}
Assuming now that $\adel z_{1j} \wed \adel z_{1j} = 0$, for some $j \geq 3$, we see that 
\begin{align*}
0 = F_{j-1} \tr \big(\adel z_{1j} \wed \adel z_{1j}\big) = &  \,  \adel z_{1j} \wed \big( F_{j-1} \tr \adel z_{1j})  + \big(F_{j-1} \tr \adel z_{1j}\big) \wed \big(K_{j-1}\inv \tr \adel z_{1j}\big)\\
= & \,   \adel z_{1j} \wed \adel (F_{j-1} \tr z_{1j})  +  \adel(F_{j-1} \tr z_{1j}) \wed  \adel(K_{j-1}\inv \tr z_{1j})\\
= & \, q^{-2} \adel z_{1j} \wed \adel z_{1,j-1} + q^{-3} \adel z_{1,j-1} \wed \adel z_{1j}.
\end{align*}
Thus whenever $\adel z_{1j} \wed \adel z_{1j} = 0$, we necessarily have that 
\begin{align*}
 \adel z_{1j} \wed \adel z_{1,j-1} =  - q^{-1} \adel z_{1,j-1} \wed \adel z_{1j}.
\end{align*}

Assume next that 
\begin{align} \label{eqn:adelzcommrels}
\adel z_{1j} \wed \adel z_{1,j-1} =  - q^{-1}\,\adel z_{1,j-1} \wed \adel z_{1j}.
\end{align}
Operating by $F_{j-1}$ gives us that   
\begin{align*}
0 = & \, F_{j-1} \tr \Big( \adel z_{1j} \wed \adel z_{1,j-1} + q^{-1}\,\adel z_{1,j-1} \wed \adel z_{1j}\Big)\\
= & \, \, q^{-1}\,  \adel z_{1,j-1} \wed \adel z_{1,j-1} + q^{-3} \, \adel z_{1,j-1} \wed \adel z_{1,j-1}\\
= & \, (q^{-1} + q^{-3}) \, \adel z_{1,j-1} \wed \adel z_{1,j-1}.
\end{align*}
Thus we see that whenever (\ref{eqn:adelzcommrels}) holds, we necessarily have that $\adel z_{1,j-1} \wed \adel z_{1,j-1} = 0$.  The corollary now follows by an inductive argument.
\end{proof}

\begin{lem} \label{lem:nukHW}
For  $k=0, \dots, n-2$, a highest weight vector  is given by
\begin{align}
 \nu_k  = \sum_{l =0}^{k} (-q)^l z_{1,n-l} \adel z_{1n} \wed \cdots \wed \wh{\adel z_{1,n-l}} \wed \cdots \wed \adel z_{1,n-k} \in \Om^{(0,k)},
\end{align}
where $\wh{\adel z_{1,n-l}}$ denotes that the factor $\adel z_{1,n-l}$ has been omitted. 
Moreover, it holds that 
\begin{align*}
\mathrm{wt}(\nu_k) = (k+1)\varpi_{1} + \varpi_{n-k-1}.
\end{align*}
\end{lem}
\begin{proof}
Since it is clear that $\nu_k$ is a weight vector, we need only show that $E_i \tr \nu_k = 0$, for all $i = 1, \dots, n-1$.  Note first that, for $1 \leq i \leq  n-k-1$, we must have $E_i \tr \nu_k = 0$.
Next, for any  $i = n-k, \dots, n-1$,  Lemma \ref{lem:EKFBasis} and (\ref{eqn:dualpairingproductflip})  imply that
\begin{align*} 
 E_i  \, \tr \, &\left( \sum_{l=0}^{n-i-2} (-q)^l z_{1,n-l} \adel z_{1n} \wed \cdots \wed \wh{\adel z_{1,n-l}} \wed \cdots \wed \adel z_{1,n-k}\right)
 \end{align*}
 is equal to the following sum
 \begin{align*}
 \sum_{l=0}^{n-i-2} (-q)^l  z_{1,n-l} \adel z_{1n} \wed \cdots \wed \wh{\adel z_{1,n-l}} \wed \cdots \wed  \adel\left(K_{i} \tr z_{1,i+1} \right) \wed \adel\big(E_i \tr  z_{1i}\big) \wed \cdots \wed \adel z_{1,n-k}.
\end{align*}
Lemma \ref{lem:EKFBasis} implies that this sum is equal to
 \begin{align*}
 q^3 \sum_{l=0}^{n-i-2} (-q)^l  z_{1,n-l} \adel z_{1n} \wed \cdots \wed \wh{\adel z_{1,n-l}} \wed \cdots \wed \adel z_{1,i+1}  \wed \adel z_{1,i+1}  \wed \cdots \wed \adel z_{1,n-k},
\end{align*}
which by Corollary \ref{cor:delzdelzzero} is equal to zero. Similarly, it holds that 
\begin{align*}
 E_i  \, \tr \, &\left( \sum_{l=n-i+1}^{k} (-q)^l z_{1,n-l} \adel z_{1,n} \wed \cdots \wed \wh{\adel z_{1,n-l}} \wed \cdots \wed \adel z_{1,n-k}\right) = 0.
\end{align*}
Hence we see that 
\begin{align*}
E_i \tr \nu_k = & \,  E_i \tr \Big((-q)^{n-i-1} z_{1,i+1}  \adel z_{1n} \wed \cdots  \wed \wh{\adel z_{1,i+1}} \wed \cdots \wed \adel z_{1,n-k} \\ 
                        & ~~~ + (-q)^{n-i} z_{1i}  \adel z_{1n} \wed \cdots  \wed \wh{\adel z_{1i}} \wed \cdots \wed \adel z_{1,n-k}\Big)\\
= & \, (-1)^{n-i-1} q^{n-i+2}   z_{1,i+1}  \adel z_{1,n} \wed \cdots  \wed \wh{\adel z_{1i}} \wed \cdots \wed \adel z_{1,n-k} \\
   & ~~~ + (-1)^{n-i} q^{n-i+2}  z_{1,i+1}  \adel z_{1,n} \wed \cdots  \wed \wh{\adel z_{1i}} \wed \cdots \wed \adel z_{1,n-k}\big)\\
= & ~~ 0.
\end{align*}
Thus $E_i \tr \nu_k = 0$,  for all  $i = 1, \dots, n-1$. Finally, as a direct examination confirms,  $\nu_k$ is a weight vector of weight $(k+1)\varpi_{1} + \varpi_{n-k-1}$, and so, $\nu_k$ is a highest weight vector as claimed. 
\end{proof}

\begin{cor} \label{cor:delnuomega} For the form $\adel \nu_k  \in \adel \Om^{(0,k)}_{\text{\em hw}}$, it holds that:
\begin{enumerate}
\item $\adel \nu_k = (k+1)_{q^2} \adel z_{1n} \wed \cdots \wed \adel z_{1,n-k}$,
\item $\adel \nu_k  \neq 0$,
\item $\nu_k \in \adel^\dagger \!\Om^{(0,k+1)}, ~~ \text{ for } k=0,\dots, n-1$.
\end{enumerate}
\end{cor}
\begin{proof} ~~~
\begin{enumerate}
\item By the commutation relations of Corollary \ref{cor:delzdelzzero}, we see that
\begin{align*}
\adel \nu_k = &  \sum_{l=0}^{k}\, (-q)^l \adel z_{1,n-l} \wed \adel z_{1n} \wed \cdots \wed \wh{\adel z_{1,n-l}} \wed \cdots \wed \adel z_{1,n-k} \\
= &  \sum_{l=0}^{k}\, q^{2l} \, \adel z_{1n} \wed \cdots \wed  \adel z_{1,n-k} \\
= &\, (k+1)_{q^2}\,  \adel z_{1n} \wed \cdots  \wed  \adel z_{1,n-k}.
\end{align*}

\item  If the coset of $\adel \nu_k$  in $\Phi\left(\Om^\bullet\right)$ were non-zero, then it is clear that $\adel \nu_k$ would have to be non-zero. Unfortunately, by Lemma \ref{basislem} we have that \begin{align*}
[\adel \nu_k] = (k+1)_{q^2} [\adel z_{1n}] \wed \cdots \wed [\adel z_{1,n-k}] = 0.
\end{align*}  
On the other hand, the coset of   $\adel z_{n1} \wed \cdots  \wed  \adel z_{n,n-k}$ in $\Phi\left(\Om^\bullet\right)$ is non-zero, as we see from
\begin{align*}
\left[\adel z_{n1} \wed \cdots  \wed  \adel z_{n,n-k}\right] = \left[\adel z_{n1}\right] \wed \cdots  \wed  \left[\adel z_{n,n-k}\right] = e^-_{1} \wed \cdots \wed e^-_{k+1}.
\end{align*} 
Hence $\adel z_{n1} \wed \cdots  \wed  \adel z_{n,n-k} \neq 0$. Now, 
as a direct calculation confirms, there exists a non-zero $\gamma \in \mathbb{R}$, such that 
\begin{align*}
\left(F^{k+2}_{n-1}  \cdots F^{k+2}_{n-k-1} F^{k+1}_{n-k-2} \cdots F^{k+1}_1\right) \tr \adel \nu_k = \gamma \,  \adel z_{n1} \wed \cdots  \wed  \adel z_{n,n-k}.
\end{align*} 
Thus $\adel z_{1n} \wed \cdots  \wed  \adel z_{1,n-k}$  must be non-zero as claimed.

\item Since $\nu_k$ is a highest weight vector, and Hodge decomposition is a decomposition of left $U_q(\frak{sl}_n)$-modules, the fact that the complex structure is of Gelfand type implies  that $\nu_k$ is either $\adel$-exact, $\adel$-coexact, or harmonic. Since we have just shown that  $\adel \nu_k \neq 0$, we must have that  $\nu_k$ is $\adel$-coexact as claimed. \qedhere
\end{enumerate}
\end{proof}

With these results in hand we are now ready to establish the main result of this subsection, an order I presentation of $\text{\bf C} = \left(\O_q(\mathbb{CP}^{n-1}),\Omega^{\bullet},\ast_\sigma, \Omega^{(\bullet,\bullet)}\right)$.

\begin{thm} \label{thm:OrderIPresentation}
Denoting $\Theta := \{\adel \nu_k \,|\, k= 0, \dots, n-2\}$, the pair $(z_{1n}, \Theta)$ is a real ladder presentation for $\text{\bf C} = \left(\O_q(\mathbb{CP}^{n-1}),\Omega^{\bullet},\ast_\sigma, \Omega^{(\bullet,\bullet)}\right)$. Moreover, it holds that 
\begin{align} \label{eqn:harmoniccpn}
\H^{(0,0)} = \bC 1, & & \H^{(0,m)} = 0, & & \text{ for } m=1, \dots, n-1.
\end{align}
\end{thm}
\begin{proof}  
Example \ref{eg:CalcMonoid}, together with Lemma \ref{lem:Leibconstant} and Lemma \ref{lem:nukHW}, imply that $z_{1n}^l \nu_k$ and $z_{1n}^l \adel \nu_k$ are highest weight vectors, for all $l \in \bN_0$, and $k = 0, \dots, n-1$. Moreover, it follows from Lemma \ref{lem:Leibconstant} and Lemma \ref{lem:nukHW}  that 
\begin{align*}
\mathrm{wt}\big(z^l_{1n}\nu_k\big) = \mathrm{wt}\big(z^l_{1n} \adel \nu_k\big) = (l+k+1)\varpi_1 + \varpi_{n-k-1} + l \varpi_{n-1}.
\end{align*}
Comparing this with the list of highest weights appearing in  the decomposition of $\Om^{(0,k)}$ in Lemma \ref{lem:DecompOfOm0kINTOIrreps}, and recalling that $\Om^{(0,k)}$ is multiplicity-free, we see that every non-trivial highest weight space of $\Om^{(0,k)}$ contains an element  of the form $z_{1n}^l \adel \nu_{k-1}$, or $z_{1n}^l \nu_k$, for some $l \in \bN_0$. Hence
\begin{align*}
\Om^{(0,0)} = \bC 1 \oplus \bigoplus_{l \in \bN_0}  U_q(\frak{sl}_{n}) (z_{1n}^{l+1}), & & \Om^{(0,k)} \simeq \bigoplus_{l \in \bN_0} U_q(\mathfrak{sl}_{n}) (z^l_{1n} \adel \nu_{k-1}) \oplus \bigoplus_{l \in \bN_0}  U_q(\frak{sl}_{n}) (z_{1n}^l \nu_k).
\end{align*}
Proposition \ref{prop:HodgeClosure} tells us that $\adel \Om^{(0,k-1)}_{\hw}$ and $\adel^\dagger \!\Om^{(0,k+1)}_{\hw}$ are closed under the action of the monoid $\O_q(\ccpn)_{\hw}$. Thus 
$
z_{1n}^l \adel \nu_{k-1} \in \adel \Om^{(0,k-1)} \text{ and } z_{1n}^l \nu_k \in \adel^{\dagger}\!\Om^{(0,k+1)}.
$
As a direct consequence
\begin{align*}
 \bigoplus_{l \in \bN_0} U_q(\frak{g}) (z_{1n}^l \adel \nu_{k-1})  =  \adel \Om^{(0,k-1)}, & & \bigoplus_{l \in \bN_0}  U_q(\frak{g}) (z_{1n}^l \nu_k) = \adel^\dagger \!\Om^{(0,k+1)}.
\end{align*}
Thus we see that the anti-holomorphic harmonic forms are exactly as claimed in (\ref{eqn:harmoniccpn}). Finally, we see that  $(z_{1n}, \Theta)$ is a ladder presentation of the calculus, which since the Leibniz constants of $z_{1n}$ are $q^2$ and $q^{-2}$,  is a real ladder presentation. 
\end{proof}

Combining this result with Lemma \ref{lem:DecompOfOm0kINTOIrreps} and Corollary \ref{cor:selfconjCPN} now gives us the following corollary.

\begin{cor} \label{cor:cpnorderI}
The CQH-complex space $\O_q(\mathbb{CP}^{n-1})$ is order I.
\end{cor}

We finish by observing that Proposition \ref{thm:cpncohomology} can be  concluded directly from Theorem \ref{thm:OrderIPresentation}. Explicitly, since the space of harmonic forms $\H^{(0,k)}$ is trivial, for all $k > 0$,  the bijection between harmonic forms and cohomology classes presented in Corollary \ref{cor:harmonictoclass} gives us the following result.

\begin{cor} \label{cor:LADDERSandVanishingcoho}
 For all $k >0$, the anti-holomorphic cohomology groups $H^{(0,k)}$ are trivial. 
\end{cor}

\subsection{A Dolbeault--Dirac  Pair of Spectral Triples for Quantum Projective Space}

In this subsection we verify solidity for the unique covariant K\"ahler structure of $\O_q(\ccpn)$. We then conclude that the K\"ahler structure gives a Dolbeault--Dirac pair of spectral triples.


\begin{lem} \label{lem:Aconstant}  It holds that
\begin{align*}
1. ~  A_{z,\adel \nu_k} \neq 0, & & 2. ~ A_{z,\adel \nu_k} \neq q^2 - 1, & & \text{ for all  } k = 0, \dots, n-2.
\end{align*}
\end{lem}
\begin{proof}  ~~~
\bet 
\item 
By Lemma \ref{lem:ANonPrimitive}, non-vanishing of $A_{z,\adel \nu_k}$ is equivalent to $\del z_{n1} \wed \adel \nu_k$ being a non-primitive form. As usual, we would like to demonstrate this by considering the coset of $\del z_{n1} \wed \adel \nu_k$ in $\Phi\left(\Omega^\bullet\right)$. Unfortunately, this coset is trivial. On the other hand, since the Lefschetz map $L$ is a left $U_q(\frak{sl}_n)$-module map, $\del z_{1n} \wed \adel \nu_k$ is primitive if and only if  $X \tr \left(\del z_{1n} \wed \adel \nu_k\right)$ is primitive, for any $X \in U_q(\frak{sl}_n)$. If we now fix
\begin{align*}
X := F^{k+2}_{n-1} \cdots F^{k+2}_1,
\end{align*}
a routine calculation confirms that 
\begin{align*}
\left[X  \tr\left(\del z_{1n} \wed \adel \nu_k\right)\right] = \gamma \, e^+_{n-k-1} \wed e^-_{n-1} \wed \cdots \wed e^-_{n-k-1},
\end{align*}
for a certain non-zero scalar $\gamma \in \mathbb{R}$.  Note next that 
\begin{align} \label{eqn:nonprimitive}
~~~ \left[L^{(n-1)-(k+2)+1} \left(X \tr \left(\del z_{1n}  \wed  \adel \nu_k\right)\right)\right]  = & \, \gamma [\k^{n-k-2}] \wed e^+_{n-k-1} \wed e^-_{n-1} \wed \cdots \wed e^-_{n-k-1}.
\end{align}
It was shown in \cite[Lemma 4.18]{MMF3} that $[\k]^{n-k-2}$ is a non-zero scalar multiple of
\begin{align*}
 \sum_{I \in O(n-k-2)} e^+_I \wed e^-_I,
\end{align*}
where summation is over  all ordered subsets of cardinality  $n-k-2$. This implies that  (\ref{eqn:nonprimitive}) is equal to a non-zero scalar multiple of 
\begin{align*}
e^+_1 \wed \cdots \wed e^+_{n-k-1} \wed  e^-_{1} \wed \cdots \wed e^-_{n-1}.
\end{align*}
Thus we see that $X \tr \left(\del z_{1n} \wed  \adel \nu_k\right)$ is non-primitive, implying that $\adel z_{1n} \wed \adel \nu_k$ is non-primitive. It now follows from Lemma \ref{lem:ANonPrimitive} that $A_{z,\adel \nu_k} \neq 0$ as claimed.

\item Using an analogous argument to the one above, it can be confirmed that  $\adel \nu_k \wed \del z_{1n}$ is non-primitive. Lemma \ref{lem:BETALminus1} then implies that  $A_{z,\adel \nu_k} \neq q^2 - 1$ as claimed. \qedhere
\eet
\end{proof}

\begin{lem} \label{lem:Bconstant}
It holds that
\begin{align*}
B_{\text{\small $z,\adel \nu_k$}} = q^2 (k+1)^{-1}_{q^2},  & & \text{ for all  $k = 0, \dots, n-2$}.
\end{align*}
%
%
\end{lem}
\begin{proof}
Following the same argument as Lemma \ref{lem:Leibconstant}, one can establish the identity
\begin{align*}
\adel z_{1n} z_{1j} = q^{-1} z_{1j} \adel z_{1n}, & & \text{ for all } j = 2, \dots, n-1.
\end{align*}
It now follows from Corollary \ref{cor:delzdelzzero}  that 
\begin{align*}
  \adel z_{1n} \wed  \nu_k 
= &  \, \sum_{l=0}^k (-q)^l \, \adel z_{1n} \wed z_{1,n-l} \, \adel z_{1n}  \wed \cdots  \wed \wh{\adel z_{1,n-l}} \wed \cdots \wed \adel z_{1,n-k}\\
= & ~ q^2 \, z_{1n} \adel z_{1n} \wed \adel z_{1,n-1}  \wed \cdots   \wed \adel z_{1,n-k} \\
 & +  q^{-1} \sum_{l=1}^k (-q)^l \, z_{1,n-l}\ \wed \adel z_{1n} \wed  \adel z_{1n}  \wed \cdots  \wed \wh{\adel z_{1,n-l}} \wed \! \cdots \! \wed \adel z_{1,n-k}\\
= &  \,~ q^2 (k+1)^{-1}_{q^2}\, z_{1n} \adel \nu_k.
 \end{align*}
Denoting by $\mu_{\nu_k}$ the $\DEL_{\adel}$-eigenvalue of $\nu_k$,  
\begin{align*}
\adel z_{1n} \wed \adel^\dagger  \adel \nu_k = &\,  \mu_{\nu_k} \adel z_{1n} \wed \nu_k \\
= &  \, q^2 (k+1)^{-1}_{q^2} \mu_{\nu_k}z_{1n}  \adel  \nu_k  \\
= &  \, q^2 (k+1)^{-1}_{q^2} z_{1n}  \adel^{\,} \adel^\dagger(\adel \nu_k),
\end{align*} 
giving us the claimed value of $B_{z,\adel \nu_k}$. \qedhere 
\end{proof}

\begin{cor}
The CQH-Hermitian space $\left(\O_q(\ccpn),\Omega^{\bullet},\ast_{\sigma}, \Omega^{(\bullet,\bullet)},\sigma\right)$ is solid, for all $q \in \bR\bs\{-1,0\}$.
\end{cor}

The following theorem and corollary, the main results of this section,  and indeed two of  the principal results of 
the paper, now follow directly from  Theorem \ref{thm:spectraltriplesfromSOLIDandCoH}, and the fact that $\Om^{(\bullet,\bullet)}$ is connected and has vanishing higher-order anti-holomorphic  cohomologies (as shown in Lemma \ref{cor:LADDERSandVanishingcoho}).

\begin{thm} \label{thm:theTHMCPN}
A Dolbeault--Dirac pair of spectral triples is given by
\begin{align*}
\Big(\O_q(\ccpn), L^2\big(\Om^{(\bullet,0)}\big), D_{\del}\Big), & & \Big(\O_q(\ccpn), L^2\big(\Om^{(0,\bullet)}\big), D_{\adel}\Big).
\end{align*}
\end{thm} 

\begin{cor}
The operators $D_{\del}$ and $D_{\adel}$ both have non-trivial associated $K$-homology classes. In particular, 
\begin{align*}
\text{\em Index}(\frak{b}(D_{\del})) = \text{\em Index}(\frak{b}(D_{\adel})) = 1.
\end{align*} 
\end{cor}

%
%
%

\begin{remark}
An important question to ask is whether or not the two spectral triples presented in Theorem \ref{thm:theTHMCPN} are unitarily equivalent, or at least have the same $K$-homology class. Note that since the $*$-map is neither an algebra map, nor a bounded operator (see \cite[\textsection 1.7]{NeshTusetLeabh} for details), it cannot be used to implement such an equivalence.
\end{remark}

\section{Irreducible Quantum Flag Manifolds} \label{section:QHSPs}

In the classical setting,  complex projective space is a very special example of a flag manifold. This picture extends directly to the quantum group setting, where  $\O_q(\ccpn)$ is a very special type of quantum flag manifold $\O_q(G/L_S)$ \cite{DijkStok}. Moreover, for the  irreducible quantum flag manifolds 
the \hk classification of  differential calculi presented above for $\O_q(\ccpn)$ extends directly, as does the existence of a unique covariant K\"ahler structure. In this section we carefully recall this material, and introduce a natural weakening of the definition of Gelfand type, which we term {\em weak Gelfand type}. We classify those non-exceptional  irreducible quantum flag manifolds satisfying this new condition, and discuss in detail the extension of the order I framework to this more general setting.


\subsection{The Quantum Homogeneous Spaces}

Let  $\frak{g}$ be a complex simple Lie \alg of rank $r$ and $U_q(\frak{g})$ the corresponding Drinfeld--Jimbo quantised enveloping algebra.   For $S$ a subset of simple roots,  consider the Hopf $*$-subalgebra
\begin{align*}
U_q(\frak{l}_S) := \big< K_i, E_j, F_j \,|\, i = 1, \ldots, r; j \in S \big>.
\end{align*} 
From the Hopf $*$-algebra embedding $\iota:U_q(\frak{l}_S) \hookrightarrow U_q(\frak{g})$, we get the dual Hopf $*$-algebra map $\iota^{\circ}: U_q(\frak{g})^{\circ} \to U_q(\frak{l}_S)^{\circ}$. By construction $\O_q(G) \sseq U_q(\frak{g})^{\circ}$, so we can consider the restriction map
\begin{align*}
\pi_S:= \iota^\circ|_{\O_q(G)}: \O_q(G) \to U_q(\frak{l}_S)^{\circ},
\end{align*}
and the Hopf $*$-subalgebra 
$
\O_q(L_S) := \pi_S\big(\O_q(G)\big) \sseq U_q(\frak{l}_S)^\circ.
$
We call the CMQGA-homogeneous space associated to   $\pi_S:\O_q(G) \to \O_q(L_S)$ the {\em quantum flag manifold} corresponding to $S$, and denote it by
\begin{align*}
\O_q\big(G/L_S\big) := \O_q \big(G)^{\text{co}\left(\O_q(L_S)\right)}.
\end{align*} 
We see that the definition of $\O_q(\mathbb{CP}^n)$, as given in the previous section, corresponds to the special case of $S = \{\a_1, \dots, \a_{n-2}\} \sseq \Pi(\frak{sl}_{n})$.

If $S = \{1, \ldots, r\}\bs \{\a_i\}$, where $\a_i$ has coefficient $1$ in the highest root of $\frak{g}$, then we say that the associated quantum flag manifold is of {\em irreducible}. In the classical limit of $q=1$, these homogeneous spaces reduce to the family of irreducible flag manifolds, or equivalently the compact Hermitian symmetric spaces \cite{BastonEastwood}. These algebras are also referred to as the  {\em cominiscule} quantum flag manifolds, again reflecting terminology in the classical setting. Presented below is a useful diagrammatic presentation of the set of simple roots defining the irreducible quantum flag manifolds, along with the names of the various series. See \cite{BastonEastwood} for a more detailed discussion. 

\subsection{The Heckenberger--Kolb Calculi and their K\"ahler Structures}

We now recall the extension of Theorem \ref{thm:HKClassCPN}, the Heckenberger--Kolb classification of first-order differential calculi over $\O_q(\ccpn)$, to the setting of irreducible quantum flag manifolds.

We call the maximal prolongation of the direct sum of these two calculi the {\em \hk calculus} of $\O_q(G/L_S)$, and denote it by $\Om_q^\bullet(G/L_S)$. It was shown in \cite{HKdR} that $\Om^\bullet(G/L_S)$ has classical dimension, and admits a left $\O_q(G)$-covariant $\bN^2_0$-grading $\Om^{(\bullet,\bullet)}$ satisfying conditions 1 and 3 of the definition of a complex structure. It was later shown in \cite{MarcoConj} that  the $*$-algebra structure of $\O_q(G/L_S)$ extends to the structure of a differential $*$-calculus on $\Om^\bullet_q(G/L_S)$,  for each irreducible quantum flag manifolds. Moreover, it was shown that the $\bN^2_0$-grading $\Om^{(\bullet,\bullet)}$  is a covariant complex structure  \wrt this $*$-structure. Finally, we note that since $\Om^{(1,0)}$ and $\Om^{(0,1)}$ are non-isomorphic as objects in  {\small$\, ^{\O_q(G)}_{\O_q(G/L_S)}  \mathrm{Mod}_0$}, this is the unique covariant complex structure for $\Om^\bullet_q(G/L_S)$.

\begin{thm} \cite[Theorem 7.2]{HK}
For  $S$ a subset of simple roots of irreducible type, there exist exactly two non-isomorphic, irreducible,  left-covariant, finite-dimensional, first-order differential calculi of finite dimension over $\O_q(G/L_S)$.
\end{thm}

\begin{center} 
{\tiny
\begin{tabular}{|c|c|c|c| }

\hline

 & &  &\\

\small $A_n$&
\begin{tikzpicture}[scale=.5]
\draw
(0,0) circle [radius=.25] 
(8,0) circle [radius=.25] 
(2,0)  circle [radius=.25]  
(6,0) circle [radius=.25] ; 

\draw[fill=black]
(4,0) circle  [radius=.25] ;

\draw[thick,dotted]
(2.25,0) -- (3.75,0)
(4.25,0) -- (5.75,0);

\draw[thick]
(.25,0) -- (1.75,0)
(6.25,0) -- (7.75,0);
\end{tikzpicture} & \small $\O_q(\text{Gr}_{r,s})$ & \small quantum Grassmanian \\

 &  & & \\
& & & \\

\small $B_n$ &
\begin{tikzpicture}[scale=.5]
\draw
(4,0) circle [radius=.25] 
(2,0) circle [radius=.25] 
(6,0)  circle [radius=.25]  
(8,0) circle [radius=.25] ; 
\draw[fill=black]
(0,0) circle [radius=.25];

\draw[thick]
(.25,0) -- (1.75,0);

\draw[thick,dotted]
(2.25,0) -- (3.75,0)
(4.25,0) -- (5.75,0);

\draw[thick] 
(6.25,-.06) --++ (1.5,0)
(6.25,+.06) --++ (1.5,0);                      

\draw[thick]
(7,0.15) --++ (-60:.2)
(7,-.15) --++ (60:.2);
\end{tikzpicture} & \small $\O_q(\mathbf{Q}_{2n+1})$ & \small odd  quantum  quadric\\

 & & &  \\
 & & & \\

\small $C_n$& 
\begin{tikzpicture}[scale=.5]
\draw
(0,0) circle [radius=.25] 
(2,0) circle [radius=.25] 
(4,0)  circle [radius=.25]  
(6,0) circle [radius=.25] ; 
\draw[fill=black]
(8,0) circle [radius=.25];

\draw[thick]
(.25,0) -- (1.75,0);

\draw[thick,dotted]
(2.25,0) -- (3.75,0)
(4.25,0) -- (5.75,0);

\draw[thick] 
(6.25,-.06) --++ (1.5,0)
(6.25,+.06) --++ (1.5,0);                      

\draw[thick]
(7,0) --++ (60:.2)
(7,0) --++ (-60:.2);
\end{tikzpicture} &\small   $\O_q(\mathbf{L}_{n})$ & \small 
quantum Lagrangian  \\

  &  & & \small Grassmannian \\ 
 & & & \\

\small $D_n$& 
\begin{tikzpicture}[scale=.5]

\draw[fill=black]
(0,0) circle [radius=.25] ;

\draw
(2,0) circle [radius=.25] 
(4,0)  circle [radius=.25]  
(6,.5) circle [radius=.25] 
(6,-.5) circle [radius=.25];

\draw[thick]
(.25,0) -- (1.75,0)
(4.25,0.1) -- (5.75,.5)
(4.25,-0.1) -- (5.75,-.5);

\draw[thick,dotted]
(2.25,0) -- (3.75,0);
\end{tikzpicture} &\small   $\O_q(\mathbf{Q}_{2n})$ & \small  even  quantum quadric \\ 
 &  & &  \\ 
 & & & \\

\small $D_n$ & 
\begin{tikzpicture}[scale=.5]
\draw
(0,0) circle [radius=.25] 
(2,0) circle [radius=.25] 
(4,0)  circle [radius=.25] ;

\draw[fill=black] 
(6,.5) circle [radius=.25] 
(6,-.5) circle [radius=.25];

\draw[thick]
(.25,0) -- (1.75,0)
(4.25,0.1) -- (5.75,.5)
(4.25,-0.1) -- (5.75,-.5);

\draw[thick,dotted]
(2.25,0) -- (3.75,0);
\end{tikzpicture} &\small   $\O_q(\textbf{S}_{n})$ & \small  quantum spinor  variety\\
 &  & & \\ 
 & & & \\

\small $E_6$& \begin{tikzpicture}[scale=.5]
\draw
(2,0) circle [radius=.25] 
(4,0) circle [radius=.25] 
(4,1) circle [radius=.25]
(6,0)  circle [radius=.25] ;

\draw[fill=black] 
(0,0) circle [radius=.25] 
(8,0) circle [radius=.25];

\draw[thick]
(.25,0) -- (1.75,0)
(2.25,0) -- (3.75,0)
(4.25,0) -- (5.75,0)
(6.25,0) -- (7.75,0)
(4,.25) -- (4, .75);
\end{tikzpicture}

 &\small  $\O_q(\mathbb{OP}^2)$ & \small  quantum Cayley plane \\
 &   & & \\ 
 & & & \\
\small $E_7$& 
\begin{tikzpicture}[scale=.5]
\draw
(0,0) circle [radius=.25] 
(2,0) circle [radius=.25] 
(4,0) circle [radius=.25] 
(4,1) circle [radius=.25]
(6,0)  circle [radius=.25] 
(8,0) circle [radius=.25];

\draw[fill=black] 
(10,0) circle [radius=.25];

\draw[thick]
(.25,0) -- (1.75,0)
(2.25,0) -- (3.75,0)
(4.25,0) -- (5.75,0)
(6.25,0) -- (7.75,0)
(8.25, 0) -- (9.75,0)
(4,.25) -- (4, .75);
\end{tikzpicture} &\small   $\O_q(\textrm{F})$  
& \small  quantum Freudenthal variety\\

& & & \\  &  &  & \small  \\

& & & \\  &  &  & \\

\hline
\end{tabular}
}
\end{center}

 It was observed in \cite[\textsection 4.5]{MMF3} that, for each $\Om^\bullet_q(G/L_S)$,  the space of $(1,1)$-forms contains a left-coinvariant  real form $\k$ which is unique up to real scalar multiple. Moreover, by extending the argument of   \cite[\textsection 4.5]{MMF3} for the case of $\O_q(\ccpn)$, the form $\kappa$ is easily seen to be a closed  central element of  $\Om^\bullet$. We now recall a conjecture originally proposed in  \cite[\textsection 4.5]{MMF3}.

\begin{conj}
For every irreducible quantum flag manifold $\O_q(G/L_S)$, a covariant K\"ahler structure for the \hk calculus of $\O_q(G/L_S)$ is given by the pair  $(\Om^{(\bullet,\bullet)},\k)$.
\end{conj}

The conjecture was motivated by the K\"ahler structure of $\O_q(\ccpn)$, as presented in \textsection \ref{thm:CPNKahler}, and as originally considered in \cite{MMF3}.  Subsequently, for all but a finite number of values of $q$, the conjecture was verified for every irreducible quantum flag manifold by Matassa.

\begin{thm}[\cite{MarcoConj} Theorem 5.10]
Let $\Om^\bullet_q(G/L_S)$ be the Heckenberger--Kolb calculus of the irreducible quantum flag manifold $\mathcal{O}_q(G/L_S)$. The pair $(\Om^{(\bullet,\bullet)},\kappa)$ is a covariant K\"ahler structure for all $q \in \mathbb{R}_{>0}\setminus F$, where $F$ is a finite, possibly empty, subset of $\mathbb{R}_{>0}$. Moreover, any element of $F$ is necessarily  non-transcendental.
\end{thm}

Building on this result, a positive definiteness result for the metric was established in \cite[Lemma 10.10]{DOS1}, allowing us to conclude that we have a CQH-Hermitian space, and hence a BC-triple.

\begin{cor}
For each irreducible quantum flag manifold $\O_q(G/L_S)$, there exists an open interval $I\sseq \mathbb{R}$ around $1$, such that the associated metric $g_{\kappa}$ is positive definite, for all $q \in I.$ Hence,  $\O_q(G/L_S)$, endowed with its Heckenberger--Kolb calculus, is a CQH-Hermitian space, for all $q\in I$, with an associated pair of Dolbeault--Dirac BC-triples. 
\end{cor}

\subsection{Irreducible Quantum Flag Manifolds of Weak Gelfand Type}

As is shown below, the only irreducible quantum flag manifold of Gelfand type is $\O_q(\mathbb{CP}^{n-1})$. However, the notion admits a natural generalisation, retaining many of the features of Gelfand type CQH-complex spaces, as we now present.

\begin{defn} We say that  a CQH-complex space ${\bf C} = \big(\text{\bf B}, \Om^{(\bullet,\bullet)}\big)$ is of {\em  weak Gelfand type} if $\adel \Om^{(0,\bullet)}$ is a graded multiplicity-free  left $A$-comodule. We say that a CQH-Hermitian space is of {\em weak Gelfand type} if its constituent CQH-complex space is of weak Gelfand type.
\end{defn}

Note first that all the results of \textsection 3.1 and \textsection 4.1 hold since neither make any assumptions about multiplicities.
Moreover, the proof of Lemma  \ref{lem:GelfandDiag} generalises directly to the weak Gelfand setting giving us the following lemma.

\begin{lem}
For a CQH-Hermitian space of weak Gelfand type,  the Laplacian $\DEL_{\adel}$ acts on every irreducible left $A$-sub-comodule of $\adel \Om^{(0,\bullet)}$ and $\adel^\dagger\! \Om^{(0,\bullet)}$ as a scalar multiple of the identity.
\end{lem}

Just as in the Gelfand setting, we can use the various symmetries of the Dolbeault double complex to find a number of equivalent formulations of weak Gelfand type. 
 We omit the proof, which is a straightforward generalisation of the proof in the Gelfand case.

\begin{lem} \label{lem:FOURGTS}
Let ${\bf C} = \big(\text{\bf B},  \Om^{(\bullet,\bullet)}\big)$ be a quantum homogeneous complex space, then the following conditions are equivalent:
\bet

\item ${\bf C}$ is  of weak Gelfand type,

\item $\text{\bf C}^{\mathrm{op}}$, the opposite CQH-complex space,  is of weak Gelfand type,

\item the graded left $A$-comodule  $\del \Om^{(\bullet,0)}$ is graded multiplicity-free.

\eet
\end{lem}

For the case of a quantum homogeneous Hermitian space, the definition of weak Gelfand type admits an additional number of equivalent formulations. We omit the proof of the following lemma, which is an easy  consequence of Lemma \ref{lem:HodgeOpDecomp}, Lemma \ref{prop:ADELISO}, and Lemma \ref{lem:FOURGTS} above.

\begin{lem} \label{lem:GTQHHS}
For ${\bf H} = \big(\text{\bf B}, \Om^{(\bullet,\bullet)}, \s\big)$ a CQH-Hermitian space,  the following are equivalent:
\bet
\item ${\bf H}$ is of Gelfand type,

\item the graded left $A$-comodule   $\adel^\dagger\! \Om^{(0,\bullet)}$ is graded multiplicity-free, 

\item the graded left $A$-comodule   $\del^\dagger \Om^{(\bullet,0)}$ is graded  multiplicity-free,

\item the graded left $A$-comodule  $\del \Om^{(\bullet,n)}$ is graded multiplicity-free,

\item the graded left $A$-comodule   $\del^\dagger\! \Om^{(\bullet,n)}$ is graded multiplicity-free, 

\item the graded left $A$-comodule $\adel \Om^{(n,\bullet)}$ is graded multiplicity-free,

\item the graded left $A$-comodule   $\adel^\dagger\!\Om^{(n,\bullet)}$ is graded multiplicity-free.

\eet
\end{lem}

Finally, we note that this lemma implies that the assumptions of Lemma \ref{lem:GelfandandConnectivity} hold in the weak Gelfand setting. Hence we again have an equivalence between connectivity and finite-dimensionality of $H^0$.

\bigskip

We would now like to identify those irreducible quantum flag manifolds which are of weak Gelfand type. As observed  by Stokman and Dijkhuizen \cite{DijkStok},  the preservation under $q$-deformation of the Weyl character formula  implies that, for any subset $S$ of simple roots, the branching rules for the inclusion $U_q(\frak{l}_S) \hookrightarrow U_q(\frak{g})$ are the same as in the classical case:

\begin{prop}[\cite{DijkStok} Proposition 4.2]  \label{prop:DS}
Let $\mu \in \mathcal{P}^+$. The  multiplicity  of  any irreducible
$U_q(\frak{l}_S)$-module in the decomposition of $V_{\mu}$ into irreducible  $U_q(\frak{l}_S)$-modules is the same as in the classical $q=1$ case.
\end{prop}

For any quantum homogeneous space $B = A^{\co(H)}$, and any left $H$-comodule $V$, Frobenius reciprocity tells us that the multiplicities of a right $A$-comodule in $\Psi(V)$ is completely determined by the branching rules for the the quantum homogeneous space. Thus we get the following corollary.  

\begin{cor}
An irreducible quantum flag manifold, endowed with its Heckenberger --Kolb calculus,  is of (weak) Gelfand type if and only if the classical  compact Hermitian space $G/L_S$ is of (weak) Gelfand type.
\end{cor}

To the best of our knowledge this question has not been explicitly addressed in the literature. However, there exist numerous suitable formulations of the necessary  classical  branching laws. (For example,  we point the reader to the  Young diagram approach of \cite{YoungBranching}, which is very similar in spirit to the presentation of Appendix  \ref{subsection:Branching}.) This allows us to answer the question with relative ease for the non-exceptional cases. To keep the paper to a reasonable length, the technical details of the proof will be given in later work. Pending this, we present the following as a conjecture.

\begin{conj} \label{conj:weakGelfClass}
The non-exceptional compact quantum Hermitian spaces, for which the covariant complex structure of the associated  Heckenberger--Kolb calculus is of weak Gelfand type, are precisely those
 presented in the following two diagrams:
The first identifies four countable families:

\vspace{2mm}
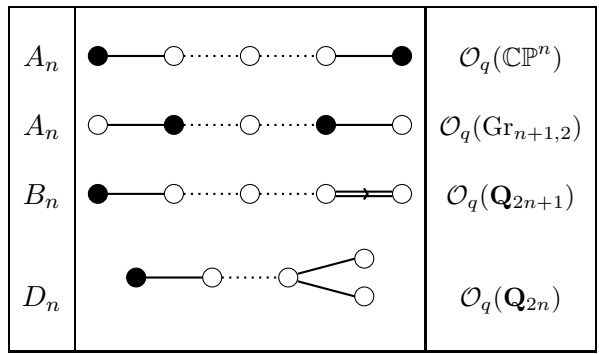
\begin{center}
\begin{tabular}{|c|c|c|}
\hline
{} & {} & {}\\
 $A_n$&
\begin{tikzpicture}[scale=.5]
\draw
(2,0)  circle [radius=.25] 
(4,0)  circle [radius=.25] 
(6,0) circle [radius=.25] ; 

\draw[fill=black]
(0,0) circle [radius=.25] 
(8,0) circle [radius=.25] ;
 
\draw[thick,dotted]
(2.25,0) -- (3.75,0)
(4.25,0) -- (5.75,0);

\draw[thick]
(.25,0) -- (1.75,0)
(6.25,0) -- (7.75,0);
\end{tikzpicture} & \small $\O_q(\mathbb{CP}^{n})$ \\

{} & {} & {} \\

 $A_n$&
\begin{tikzpicture}[scale=.5]
\draw
(0,0) circle [radius=.25] 
(4,0)  circle [radius=.25] 
(8,0) circle [radius=.25] ; 

\draw[fill=black]
(2,0)  circle [radius=.25] 
(6,0) circle [radius=.25] ;
 
\draw[thick,dotted]
(2.25,0) -- (3.75,0)
(4.25,0) -- (5.75,0);

\draw[thick]
(.25,0) -- (1.75,0)
(6.25,0) -- (7.75,0);
\end{tikzpicture} & \small $\O_q(\mathrm{Gr}_{n+1,2})$ \\

{} & {} & {} \\
$B_n$& 
\begin{tikzpicture}[scale=.5]
\draw
(4,0) circle [radius=.25] 
(2,0) circle [radius=.25] 
(6,0)  circle [radius=.25]  
(8,0) circle [radius=.25] ; 

\draw[fill=black]
(0,0) circle [radius=.25];

\draw[thick]
(.25,0) -- (1.75,0);

\draw[thick,dotted]
(2.25,0) -- (3.75,0)
(4.25,0) -- (5.75,0);

\draw[thick] 
(6.25,-.06) --++ (1.5,0)
(6.25,+.06) --++ (1.5,0);                      

\draw[thick]
(7,0.15) --++ (-60:.2)
(7,-.15) --++ (60:.2);
\end{tikzpicture} & \small $\O_q(\mathbf{Q}_{2n+1})$\\ 
 & {} & {}\\ 

$D_n$& 
\begin{tikzpicture}[scale=.5]

\draw[fill=black]
(0,0) circle [radius=.25] ;

\draw
(2,0) circle [radius=.25] 
(4,0)  circle [radius=.25]  
(6,.5) circle [radius=.25] 
(6,-.5) circle [radius=.25];

\draw[thick]
(.25,0) -- (1.75,0)
(4.25,0.1) -- (5.75,.5)
(4.25,-0.1) -- (5.75,-.5);

\draw[thick,dotted]
(2.25,0) -- (3.75,0);
\end{tikzpicture} &\small  $\O_q(\mathbf{Q}_{2n})$ \\ 
{} & {} & {}\\

\hline
\end{tabular}
\end{center}

The second diagram identifies three isolated examples, arising from low-dimensional redundancies in the table of compact quantum Hermitian spaces given above.

\begin{center}
\begin{tabular}{|c|rcl|l|}
\hline
{} & {} & {} & {} & {}\\ 
$B_3$& 
\begin{tikzpicture}[scale=.5]

\draw[fill=black]
(4,0)  circle [radius=.25]  ;

\draw
(6,0)  circle [radius=.25]  ;

\draw[thick] 
(4.25,-.06) --++ (1.5,0)
(4.25,+.06) --++ (1.5,0);    
\draw[thick]
(5,0.005) --++ (60:.2)
(5,-.005) --++ (-60:.2);
\end{tikzpicture} 
& $  \cong  $ &
\begin{tikzpicture}[scale=.5]

\draw
(4,0)  circle [radius=.25]  ;

\draw[fill=black]
(6,0)  circle [radius=.25]  ;

\draw[thick] 
(4.25,-.06) --++ (1.5,0)
(4.25,+.06) --++ (1.5,0);                      

\draw[thick]
(5,0.15) --++ (-60:.2)
(5,-.15) --++ (60:.2);

\end{tikzpicture} 

&\small  $\O_q(\mathbf{L}_2)  \simeq \O_q(\mathbf{Q}_5)$ \\ 

{} & {} & {} & {}  & {}\\
$D_3$& 
\begin{tikzpicture}[scale=.5]

\draw
(4,-.5)  circle [radius=.25]  ;

\draw[fill=black]
(6,0) circle [radius=.25] 
(6,-1) circle [radius=.25];

\draw[thick]
(4.25,-0.4) -- (5.75,0)
(4.25,-.6) -- (5.75,-1);
\end{tikzpicture} 
& $ \cong $ &

\begin{tikzpicture}[scale=.5]
\draw
(4,.5)  circle [radius=.25] ;

\draw[fill=black]
(2,.5)  circle [radius=.25] 
(6,.5) circle [radius=.25] ;
 
\draw[thick]
(2.25,.5) -- (3.75,.5)
(4.25,.5) -- (5.75,.5);

\end{tikzpicture} 
&\small  $\O_q(\mathbf{S}_3) \simeq \O_q(\mathbb{CP}^3)$ \\ 

{} & {} & {} & {}  & {}\\

{} & {} & {} & {}  & {}\\ 
$D_4$&

\begin{tikzpicture}[scale=.5]

\draw
(2,0) circle [radius=.25] 
(4,0)  circle [radius=.25]  ;

\draw[fill=black]
(6,.5) circle [radius=.25] 
(6,-.5) circle [radius=.25];

\draw[thick]
(2.25,0) -- (3.75,0)
(4.25,0.1) -- (5.75,.5)
(4.25,-0.1) -- (5.75,-.5);

\end{tikzpicture} & $  \cong  $ &\begin{tikzpicture}[scale=.5]

\draw[fill=black]
(2,0) circle [radius=.25]  ;

\draw
(4,0)  circle [radius=.25]  
(6,.5) circle [radius=.25] 
(6,-.5) circle [radius=.25];

\draw[thick]
(2.25,0) -- (3.75,0)
(4.25,0.1) -- (5.75,.5)
(4.25,-0.1) -- (5.75,-.5);

\end{tikzpicture} &\small  $\O_q(\mathbf{S}_4) \simeq \O_q(\mathbf{Q}_8)$ \\ 
{} & {} & {} & {}  & {} \\

\hline
\end{tabular}
\end{center}
Of these CQH-complex spaces, only  quantum projective space $\O_q(\ccpn)$ is of  Gelfand type.
\end{conj}

For the two exceptional cases, that is, the quantum Cayley plane $\O_q(\mathbb{OP}^2)$, and the quantum Freudenthal variety $\O_q(\mathbf{F})$, an explicit presentation of the necessary branching laws does not seem  to have appeared in the literature. Hence, it is most likely necessary to derive them from a general framework, such as the Littlemann path model  \cite{Littleman1,Littleman2}.  We postpone such intricacies to later work, and for now satisfy ourselves with Conjecture \ref{conj:ExceptionWeakGel}, as motivated in \textsection  \ref{subsection:SphericalGens}.

\subsection{Towards Order II CQH-Hermitian Spaces}

In the previous subsection we presented the definition of weak Gelfand type and discussed those results which carry over from the Gelfand setting. In this subsection we detail some of the {\em new} behaviours that arise in the weak Gelfand setting. We finish with some conjectures, motivated by the discussion of the subsection.

 \subsubsection{Hodge Decomposition and the $B_\hw$-Action}

In the weak Gelfand case Proposition \ref{prop:HodgeClosure} is no longer guaranteed to hold. Explicitly, for  any $\w \in \adel \Om^{(0,\bullet)}_\hw$, and  any $z \in \O_q(G/L_S)_\hw$, it is no longer guaranteed that the product $z \w$ is contained in $\adel \Om^{(0,\bullet)}$. This forces us to consider highest weight vectors of the form  $\Pi_{\adel}\big(z \w\big)$,   where we have denotes by $\Pi_{\adel}:\Om^{(0,\bullet)} \to \adel \Om^{(0,\bullet)}$  the projection operator  
associated to Hodge decomposition. Such forms are still eigenvectors of  $\DEL_{\adel}$, with an important observation being that 
$
\DEL_{\adel}\left(\Pi_{\adel}(z\w)\right) = \adel \adel^\dagger\!\left(z\w\right).
$

\subsubsection{Spherical Generators} \label{subsection:SphericalGens}

By the considerations of \textsection \ref{section:setofHandLWs}, the  weights of the highest weight vectors of $\O_q(G/L_S)$  form an additive submonoid $Z_S \sseq \frak{h}$ under addition. By Proposition \ref{prop:DS}, it is clear that the minimal number of generators of $Z_S$,  considered as a monoid, is the same in the classical $q=1$ case. A distinguished  minimal set of generators for $Z_S$, the {\em spherical weights}, was presented  in the classical case  by Kr\"amer in  \cite[Tabelle 1]{KramerSpherical}, and reproduced here in the table  below.

In this table,  the Dynkin nodes for the $A$-,$B$-,$C$-, and $D$-types are labelled in increasing order from left to right. 
For the exceptional cases, we again number from left to right in increasing order, finishing with the side node. Finally, note that for the spherical weights of $\O_q(\mathbf{S}_{2m})$, the weight $2\varpi_{2m-1}$ or $2\varpi_{2m}$ appears depending on the defining crossed node.  

\begin{center}
\begin{tabular}{|c|c|l}
\hline
      & \\
{\em \bf Irreducible Quantum Flag}  &  {\em \bf Spherical Weights of $\O_q(G/L_S)$ }\\
{\em  \bf Manifold $\O_q(G/L_S)$} & \\
      & \\
      \hline
            & \\
$\O_q(\text{Gr}_{r,s})$    &    $\varpi_{1} + \varpi_{r+s-1}, \, \varpi_{2} + \varpi_{r+s-2},  \dots, \varpi_{r} + \varpi_{s}$\\
      & \\
$\O_q(\mathbf{Q}_{2n+1})$ & $2\varpi_1, \, \varpi_2$ \\
      & \\
$\O_q(\mathbf{L}_n)$  &  $2\varpi_1, \, 2\varpi_2, \dots, 2\varpi_{n}$  \\
      & \\
$\O_q(\mathbf{Q}_{2n})$ & $2\varpi_1, \, \varpi_2$ \\
      & \\
$\O_q(\mathbf{S}_{2m})$  & $\varpi_2, \, \varpi_4, \dots, \varpi_{2m-2}, \, 2\varpi_{2m-1} \text{ or } 2\varpi_{2m}$\\
      & \\
      $\O_q(\mathbf{S}_{2m+1})$  & $\varpi_2, \, \varpi_4, \dots, \varpi_{2m-2}, \, \varpi_{2m} + \varpi_{2m+1}$\\
      & \\
$\O_q(\mathbb{OP}^2)$  & $\varpi_1 + \varpi_5$, \, $\varpi_{6}$\\
      & \\
$\O_q(\mathbf{F})$ & $2\varpi_{1}, \, \varpi_{2}, \, \varpi_{6}$\\
      & \\
\hline
\end{tabular}
\end{center}

Comparing this table with the proposed classification of Conjecture \ref{conj:weakGelfClass}, we make the following alternative version of the conjecture.

\begin{conj} \label{conj:weakSpherical}
The non-exceptional irreducible quantum flag manifolds $\O_q(G/L_S)$ of  Gelfand type  are exactly those for which $Z_S$ is generated by a single element, and those of weak Gelfand type are exactly those for which $Z_S$ which is generated by two elements. 
\end{conj}

Thus to accommodate the irreducible quantum flag manifolds of weak Gelfand type, it is necessary to generalise the order I requirements  and consider  irreducible $U_q(\frak{g})$-modules of the  form 
\begin{align*}
U_q(\frak{g})  \Pi_{\adel}\big(y^{l_1} z^{l_2} \w\big), & & \text{ for } \, y, z \in B,\, \w \in \Om^{(0,k)}_{\hw}, \text{ and }  l_1,l_2 \in \bN_0.
\end{align*}

Coming finally to the exceptional cases, we note that, from the table of spherical weights below, $\O_q(\mathbb{OP}^2)$ has two spherical generators, while $\O_q(\mathbf{F})$ has three. Speculating that Conjecture \ref{conj:weakSpherical} extends to the non-exceptional setting, we make the following conjecture.

\begin{conj} \label{conj:ExceptionWeakGel} It holds that:
\begin{enumerate}

\item The quantum Cayley plane $\O_q(\mathbb{OP}^2)$, endowed with its \hk calculus, is of weak Gelfand type.

\item  The quantum Freudenthal variety $\O_q(\mathbf{F})$, endowed with its Heckenberger--Kolb calculus, is not of weak Gelfand type.

\end{enumerate}
\end{conj}

\subsubsection{Leibniz Constants}

In the weak Gelfand setting, Leibniz constants are not guaranteed to exist, which is to say Corollary \ref{cor:LeibnizConstants} does not necessarily hold. We do however have the following lemma, whose proof is a direct generalisation of Corollary \ref{cor:LeibnizConstants}.

\begin{lem} 
Let  $\mathbf{C}$ be a CQH-complex space of weak Gelfand type, with constituent quantum homogeneous space $B_{\hw}$, and $z \in B_{\hw}$ a non-harmonic element. Then there exist non-zero constants $\lambda_{z}, \, \zeta_{z} \in \bC$, uniquely defined by 
\begin{align*}
\Pi_{\del}\left(\left(\del z\right)z\right) = \lambda_{z} \Pi_{\del}\left(z \del z\right), & & \Pi_{\adel}\left(\left(\adel z\right)z\right) = \z_z  \Pi_{\adel}\left(z \adel z\right),
\end{align*}
where $\Pi_\del$, and $\Pi_{\adel}$, are the projections from $\Om^\bullet$ to $\del \Om^\bullet$, and $\adel \Om^\bullet$, respectively, given by  Hodge decomposition.
\end{lem}

We now see that  Proposition \ref{prop:selfconjLeibdual} extends directly to the weak Gelfand setting.

\begin{cor}
Let  $\mathbf{C}$ be  a CQH-complex space with  constituent quantum homogeneous space $B$. If $B$ is self-conjugate, then, for any $z \in B$ such that $\lambda_z,\zeta_z \in \bR$,  we have  $ \z_z = \lambda_z \inv$.
\end{cor}
\begin{proof}
Note first that we have the identity 
$
* \circ \Pi_{\del} = \Pi_{\adel} \circ *.
$
Using this identity, the proof of Proposition \ref{prop:selfconjLeibdual} now carries over directly to the weak Gelfand setting.
\end{proof}

\subsubsection{Some Conjectures} \label{Subsubsection:Conjs}

Taking the above novel features into account, it is possible to generalise the order I framework,  to an {\em order II} framework, and include the irreducible quantum flag manifolds of weak Gelfand type. Solidity can be shown to extend to this setting, as well as  its equivalence with the compact resolvent of the Dolbeault--Dirac operator.  This gives us a robust set of tools with which to construct spectral triples for the irreducible quantum flag manifolds of weak Gelfand type. Preliminary investigations \cite{DOW} strongly suggest that solidity holds,  motivating us to make the following  conjecture. 


\begin{conj}
Let $\O_q(G/L_S)$ be an irreducible quantum flag manifold of weak Gelfand type, and let $\left(\Om^{(\bullet,\bullet)},\kappa\right)$ be its covariant K\"ahler structure, unique up to real scalar multiple.  A Dirac--Dolbeault pair of spectral triples is given by 
\begin{align*}
\Big(\O_q(G/L_S),L^2(\Om^{(\bullet,0)}), D_{\del}\Big), & & \Big(\O_q(G/L_S),L^2(\Om^{(0,\bullet)}), D_{\adel}\Big).
\end{align*}
Moreover, the associated $K$-homology class of each spectral triple is non-trivial.
\end{conj}

In the above conjecture we restrict to those irreducible quantum flag manifolds of weak Gelfand type, reflecting the fact that the approach of this paper naturally extends to this subfamily of spaces. However, there is no reason to suspect that the compact resolvent condition does not hold for all the irreducible quantum flag manifolds. Verifying the condition in this more general setting requires a more formal approach, based on the  classical proof of the compact resolvent condition for general Dirac operators \cite[\textsection 4]{FriedrichDirac}. This is at present being  undertaken in the setting  of noncommutative Sobolev spaces \cite{BBSobolev}, and so, we are motivated to reformulate  a conjecture originally presented in \cite{MMF3}. 

\begin{conj}
For every irreducible quantum flag manifold $\O_q(G/L_S)$, a covariant K\"ahler structure for the \hk calculus of $\O_q(G/L_S)$ is given by the pair  $(\Om^{(\bullet,\bullet)},\k)$.
Moreover,  a Dirac--Dolbeault pair of spectral triples is given by 
\begin{align*}
\Big(\O_q(G/L_S),L^2(\Om^{(\bullet,0)}), D_{\del}\Big), & & \Big(\O_q(G/L_S),L^2(\Om^{(0,\bullet)}), D_{\adel}\Big),
\end{align*}
and the associated $K$-homology class of each spectral triple is non-trivial.
\end{conj}


\appendix

\section{CQGAs, Quantum Homogeneous Spaces, and Frobenius Reciprocity} \label{app:CQGA}

In this appendix we recall the basics of cosemisimple Hopf algebras, compact matrix quantum groups algebras, quantum homogeneous spaces, as well as some natural compatibility requirements between them. We then recall the version of Takeuchi's categorical equivalence for quantum homogeneous spaces most suited to our purposes. Finally, we present an extension of Frobenius reciprocity to the setting of quantum homogeneous spaces.

\subsection{Hopf Algebras and  CQGAs}
All algebras are assumed to be unital and defined over $\mathbb{C}$. All unadorned tensor products are defined over $\mathbb{C}$. The symbols  $A$ and $H$ will denote  Hopf  \algn s with comultiplication $\DEL$, counit $\e$, antipode $S$, unit $1$, and multiplication $m$. 
We use Sweedler notation throughout, and denote $a^+ := a - \e(a)1$, for $a \in A$, and $A^+ = A \cap \ker(\e)$.

For any left $A$-comodule $(V,\DEL_L)$, its space of {\em matrix elements} is the sub-co\alg
\begin{align*}
C(V) : = \spn_\bC\{(\id \oby f)\DEL_L(v) \,|\, f \in \mathrm{Hom}_{\bC}(V,\bC), v \in V\} \sseq A.
\end{align*}

The notion of cosemisimplicity for a Hopf algebra will be essential in  the paper and all Hopf \algs  will be assumed to be cosemisimple. We present  three equivalent formulations of the definition (a proof of their equivalence can be found in \cite[\textsection 11.2.1]{KSLeabh}).

\begin{defn}
A Hopf algebra $A$ is called  {\em cosemisimple} if it satisfies the following three equivalent conditions:
\begin{enumerate}
\item There is an isomorphism $A \simeq \bigoplus_{V\in \wh{A}} C(V)$, where $\wh{A}$ denotes the  equivalence classes of left $A$-comodules.
\item The abelian category $^A\mathrm{Mod}$ of left $A$-comodules  is semisimple.
\item  There exists a unique linear map $\haar:A \to \bC$, which we call the {\em Haar measure}, such that $\haar(1) = 1$, and 
\begin{align*}
(\id \oby \haar) \circ \DEL(a) = \haar(a)1, & & (\haar \oby \id) \circ \DEL(a) = \haar(a)1.
\end{align*}
\end{enumerate}
\end{defn}

While the assumption of cosemisimplicity is enough for most of our requirements, when discussing positive definiteness we need the following stronger structure introduced in \cite{KoornDijk}.

\begin{defn}
A {\em compact matrix quantum group algebra}, or a {\em CMQGA}, is a finitely generated cosemisimple Hopf $*$-algebra $A$ such that  $\haar(a^*a) > 0$, for all non-zero $a \in A$.
\end{defn}

It is important to note that every CMQGA admits a (not necessarily unique) $C^*$-algebraic completion to a compact matrix quantum group in the sense of Woronowicz \cite{WoroCQPGs}. Moreover, every compact matrix quantum group contains a dense CMQGA \cite{KoornDijk}.

\subsection{CMQGA-Homogeneous Spaces} 


A left  {\em $A$-comodule algebra} $P$ is a comodule  which is also an algebra such that the comodule structure map $\DEL_L:P \to A \otimes P$ is an algebra map. Equivalently, it can be defined as a monoid object in $^A \mathrm{Mod}$, the category of left $A$-comodules. 
For a left $A$-comodule $V$ with structure map $\DEL_R$, we say that an element $v \in V$ is {\em coinvariant} if $\DEL_L(v) = 1 \otimes v$. We denote the subspace of all $A$-coinvariant elements by $\,^{\co(A)}V$, and call it the {\em coinvariant subspace} of the coaction. We use the analogous conventions for right comodules. 
\begin{defn}
A {\em homogeneous} right $H$-coaction on $A$ is a coaction of the form \linebreak $(\id \oby \pi)  \circ \DEL$, where $\pi: A \to H$ is a surjective Hopf \alg map. A  {\em quantum homogeneous space}  $B :=A^{\co(H)}$ is the coinvariant subspace of such a coaction. 
\end{defn}
As is easily verified, every quantum homogeneous space $B := A^{\co(H)}$ is a sub\alg of $A$. Moreover, the coaction of $A$ restricts to a left $A$-coaction  $\DEL_L: B \to A \otimes B$ giving it the structure of a left $A$-comodule \algn. We finish with a convenient, and natural, definition, which identifies the class of homogeneous spaces we will concern ourselves with in this paper.
\begin{defn}
A {\em CMQGA-homogeneous space}  is a quantum homogeneous space given by  a Hopf $*$-algebra projection $\pi:A \to H$, such that $A$ and $H$ are CMQGA's.
\end{defn}

\subsection{Takeuchi's Equivalence}

We  briefly recall Takeuchi's equivalence \cite{Tak}, or more precisely a special case of the bimodule version of Takeuchi's equivalence. For a more detailed presentation we refer the reader to \cite{Tak}, \cite{HK}, or \cite{MMF2, MMF3}. Throughout this subsection, $A$ and $H$ will be Hopf algebras, $\pi:A \to H$ a Hopf algebra map, and $B = A^{\text{co}(H)}$ the associated quantum homogeneous space.

\begin{defn}
The category $^A_B\mathrm{Mod}_0$ has as objects  $B$-bimodules $\F$, endowed with a left $A$-comodule structure $\DEL_L:\F \to A \otimes \F$, such that 
\begin{enumerate}  

\item $\F B^+ \sseq B^+ \F$

\item $\DEL_L(b f b') =  \DEL_L(b) \DEL_L(f) \DEL_L(b'),  ~~~~ \text{for all ~} b,b' \in B, f \in \F$.

\end{enumerate}
The morphisms in $^A_B \mathrm{Mod}_0$ are those  $B$-bimodule homomorphisms which are also homomorphisms of left $A$-comodules. The usual bimodule tensor product $\otimes_B$ can be used to give the category  a monoidal structure defined by
\begin{align*}
\DEL_L:\F \oby_B {\mathcal D} \to A \otimes  \F \otimes_B {\mathcal D}, & & f \otimes d \mapsto f_{(-1)}d_{(-1)} \otimes f_{(0)} \otimes_B d_{(0)},
\end{align*}
for $\F,\mathcal{D} \in \, ^A_B \mathrm{Mod}_0$.
\end{defn}

\begin{defn}
The category $^H\mathrm{Mod}$ has  left $H$-comodules as objects,  and  right $H$-comodule maps as morphisms. The usual tensor product of comodules $\otimes$ now endows the category with a monoidal structure
\begin{align*}
\DEL_L:V \oby W \to A \otimes  V \otimes W, & & v \otimes w \mapsto v_{(-1)}w_{(-1)} \otimes v_{(0)} \otimes w_{(0)}.
\end{align*}
\end{defn}


If $\F \in \, ^A_B \mathrm{Mod}_0$, then $\F/(B^+\F)$ becomes an object in $^H\mathrm{Mod}$ with the left $H$-coaction 
\bal \label{comodstruc0}
\DEL_L[f] = \pi(f_{(-1)}) \oby [f_{(0)}], & & f \in \F,
\eal
where $[f]$ denotes the coset of $f$ in  $\F/(B^+\F)$. A functor $\Phi: \,^A_B\mathrm{Mod}_0 \to \,^H\mathrm{Mod}$  is now defined as follows:  
$\Phi(\F) :=  \F/(B^+\F)$, and if $g : \F \to {\mathcal D}$ is a morphism in $^A_B\mathrm{Mod}_0$, then $\Phi(g):\Phi(\F) \to \Phi({\mathcal D})$ is the map to which $g$ descends on $\Phi(\F)$.

If $V \in \,^H\mathrm{Mod}$, then the {\em cotensor product} of $A$ and $V$, defined by
\begin{align*}
A \coby V := \ker(\DEL_R \oby \id - \id \oby \DEL_L: A\oby V \to A \oby H \oby V),
\end{align*}
becomes an object in $^A_B\mathrm{Mod}_0$ by defining a $B$-bimodule structure
\begin{align} \label{rightmaction}
b \Big(\sum_i a^i \oby v^i\Big)  := \sum_i b a^i \oby v^i, & & \Big(\sum_i a^i \oby v^i\Big) b := \sum_i a^i b \oby v^i,
\end{align}
and a left $A$-coaction  
$
\DEL_L\Big(\sum_i a^i \oby v^i\Big) := \sum_i a^i_{(1)} \oby a^i_{(2)} \oby v^i.
$ 
We now define a  functor ${\Psi:\,^H\mathrm{Mod} \to \, ^A_B\mathrm{Mod}_0}$ as follows: 
$
\Psi(V) := A \square_H V,
$
and if $\gamma$ is a morphism in $^H\mathrm{Mod}$, then $\Psi(\gamma) := \id \oby \gamma$.

\begin{thm}[Takeuchi's Equivalence] An adjoint equivalence of categories between   $^A_B\mathrm{Mod}_0$ and  $^H \mathrm{Mod}$  is given by the functors $\Phi$ and $\Psi$ and the natural isomorphisms
\begin{align*}
\counit:\Phi \circ  \Psi(V) \to V,   & & \Big[\sum_i a^i \oby v^i\Big] \mto \sum_{i} \e(a^i)v^i, \\ 
\unit: \F \to \Psi \circ \Phi(\F), & & f \mto f_{(-1)} \oby [f_{(0)}].
\end{align*} 
Moreover, the equivalence is a monoidal equivalence \wrt the the natural isomorphism
\begin{align*}
\mu: \Phi(\F) \otimes \Phi(\mathcal{D}) \to \Phi(\F \otimes_B \mathcal{D}), & & [f] \otimes [d] \mapsto [f \otimes_B d].
\end{align*}
\end{thm}

We define the {\em dimension} of an object $\F \in \,^A_B\mathrm{Mod}_0$ to be the vector space dimension of $\Phi(\F)$. Note that by cosemisimplicity of $A$, the category $^H\mathrm{Mod}$ is  semisimple, and so, $^A_B\mathrm{Mod}_0$ must also be semisimple.

\subsection{Frobenius Reciprocity for Quantum Homogeneous Spaces}

In this subsection we present a direct Hopf algebraic generalisation of  Frobenius reciprocity for equivariant vector bundles over a homogeneous space.  The proof carries over from the  classical situation without difficulty, so we will not state it. Moreover, there exists  a large number of related formulations of this result in the literature. For example,  it is established in the compact quantum group setting in \cite{PalFrobCQGs}, and in the general setting of coring comodules in \cite[\textsection 22.12]{BrzezWisbauer}.

The result is stated for two  Hopf algebras $A$ and $H$, and a surjective Hopf map $\pi:A \to H$. For $U,W$ two right $A$-comodules,  we denote by $\mathrm{Hom}^A(U,W)$ the space of right $A$-comodule maps from $U$ to $W$. Moreover, for any right $H$-comodule $V$, we denote by $\mathrm{Hom}^H(U,V)$ the space of right $H$-comodule maps from $U$ to $V$, with respect to the right $H$-comodule structure induced on $U$ by $\pi$.

\begin{lem}[Frobenius Reciprocity]  \label{prop:Frobenius} For   $U  \in  \,^A\mathrm{Mod}$ and $V  \in ^H \mathrm{Mod}$, it holds that 
\begin{align*}
\dim_{\mathbb{C}}\!\left(\text{\em Hom}^A(U, A\, \square_H V)\right) = \dim_{\mathbb{C}}\left(\text{\em Hom}^H(U,V)\right).
\end{align*}
\end{lem}

\section{Drinfeld--Jimbo  Quantum Groups} \label{APP:secnoumi}

In this appendix we recall some basic material about semisimple complex Lie algebras $\frak{g}$ and their associated Drinfeld--Jimbo quantised enveloping algebras $U_q(\frak{g})$. We also discuss their type $1$ representation theory, along with the associated quantum coordinate algebras $\O_q(G)$. Throughout, where basic proofs or details are omitted we refer the reader to \cite[\textsection 6,\textsection 7,\textsection 9]{KSLeabh}. We then specialise to the $A$-series  quantum coordinate algebra $\O_q(SU_n)$. We recall its explicit FRT-presentation and  identify it with the abstract quantum coordinate algebra presentation via a dual pairing of Hopf algebras. We finish with some standard material on Young diagrams necessary for an unambiguous presentation of the branching laws of \textsection \ref{section:CPN}.

\subsection{Quantum Integers} \label{app:quantumintegers}

Quantum integers are ubiquitous in the study of quantum groups. In this paper  they play a significant role in describing the spectrum of noncommutative Dolbeault--Dirac operators. There are two different but related formulations for quantum integers, so we take care to clarify our choice of conventions.

For $q \in \bC$, and $m \in \mathbb{N}$, the {\em quantum $q$-integer} $(m)_q$  is  the complex number
\begin{align*}
(m)_q := 1 + q + q^2 + \cdots + q^{m-1}.
\end{align*}
Note that when $q \neq 1$, we have the identity
\begin{align*}
(m)_q  = \frac{1-q^m}{1-q}.
\end{align*}

We contrast this with the alternative (perhaps more standard) version of quantum integers: 
\begin{align*}
[l]_q := q^{-m+1} + q^{-m+3} + \cdots q^{m-3} + q^{m-1}.
\end{align*}
It is instructive to note that the two conventions are related by the identity
\begin{align*}
[m]_q = q^{1-m} (m)_{q^2}.
\end{align*}
The second version of quantum integers was used in \cite{MMF3}, but will {\em not} be used in this paper. 
Instead, we adopt the first formulation, it being the one  which arises naturally  in our spectral calculations, as is most readily evidenced by Corollary \ref{cor:QLEIBNIZ}.

\subsection{Drinfeld--Jimbo Quantised Enveloping Algebras}

Let $\frak{g}$ be a finite-dimensional complex simple Lie algebra of rank $r$. We fix a Cartan subalgebra $\frak{h}$ with corresponding root system $\Delta \sseq \frak{h}^*$, where $\frak{h}^*$ denotes the linear dual of $\frak{h}$. Let $\Delta^+$ be a choice of positive roots, and 
let $\Pi(\frak{g}) = \{\alpha_1, \dots, \alpha_r\}$ be the corresponding set of simple roots. Denote by $(\cdot,\cdot)$ the symmetric bilinear form induced on $\frak{h}^*$ by the  Killing form of $\frak{g}$, normalised so that any shortest simple root $\a_i$ satisfies $(\a_i,\a_i) = 2$. The {\em coroot} $\alpha_i^{\vee}$ of a simple root $\a_i$ is defined by
\begin{align*}
\alpha_i^{\vee} := \frac{\a_i}{d_i} =  \frac{2\a_i}{(\a_i,\a_i)}, & & \text{ where } d_i := \frac{(\a_i,\a_i)}{2}.
\end{align*}
The Cartan matrix $(a_{ij})_{ij}$ of $\frak{g}$ defined by 
$
a_{ij} := \big(\alpha_i^{\vee},\alpha_j\big).
$

Let  $q \in \bR$ such that  $q \neq -1,0,1$, and denote $q_i := q^{d_i}$. The {\em quantised enveloping \algn}  $U_q(\frak{g})$ is the  \nc associative  \alg  generated by the elements   $E_i, F_i$, and $K_i, K^{-1}_i$, for $ i=1, \ldots, r$,  subject to the relations 
\begin{align*}
 K_iE_j =  q_i^{a_{ij}} E_j K_i, ~~~~  K_iF_j= q_i^{-a_{ij}} F_j K_i, ~~~~  K_i K_j = K_j K_i,  ~~~~ K_iK_i^{-1} = K_i^{-1}K_i = 1,\\
  E_iF_j - F_jE_i  = \d_{ij}\frac{K_i - K\inv_{i}}{q_i-q_i^{-1}}, ~~~~~~~~~~~~~~~~~~~~~~~~~~~~~~~~~~~~~~~~~
\end{align*}
along with the quantum Serre relations 
\begin{align*}
  \sum\nolimits_{s=0}^{1-a_{ij}} (-1)^s  \begin{bmatrix} 1 - a_{ij} \\ s \end{bmatrix}_{q_i}
   E_i^{1-a_{ij}-s} E_j E_i^s = 0,\quad \textrm{ for }  i\neq j,\\
  \sum\nolimits_{s=0}^{1-a_{ij}} (-1)^s \begin{bmatrix} 1 - a_{ij} \\ s \end{bmatrix}_{q_i}
   F_i^{1-a_{ij}-s} F_j F_i^s = 0,\quad \textrm{ for }  i\neq j,
\end{align*}
where we have used the $q$-binomial coefficient
\begin{align*}
\begin{bmatrix} n \\ r \end{bmatrix}_q := \frac{[n]_q!}{[r]_q! \, [n-r]_q!}.
\end{align*}
A Hopf \alg structure is defined  on $U_q(\frak{g})$ by setting
\begin{align*}
\DEL(K_i) = K_i \oby K_i, ~~  \DEL(E_i) = E_i \oby K_i + 1 \oby E_i, ~~~ \DEL(F_i) = F_i \oby 1 + K_i\inv  \oby F_i~~~~\\
 S(E_i) =  - E_iK_i\inv,    ~~ S(F_i) =  -K_iF_i, ~~~~  S(K_i) = K_i\inv,
 ~~ ~~\e(E_i) = \e(F_i) = 0, ~~ \e(K_i) = 1.     
\end{align*}
A Hopf $*$-algebra structure, called the {\em compact real form}, is defined by
\begin{align*}
K^*_i : = K_i, & & E^*_i := K_i F_i, & &  F^*_i  :=  E_i K_i \inv. 
\end{align*}
Finally, we denote by $U_q(\frak{n}_+)$, and $U_q(\frak{n}_-)$, the unital subalgebras of $U_q(\frak{g})$ generated by the elements $E_1, \dots, E_r$, and $F_1, \dots, F_r$, respectively.

\subsection{Type 1 Representations}

The set of {\em fundamental weights}  $\{\varpi_1, \dots, \varpi_r\}$  of $\frak{g}$ is the dual basis of simple coroots $\{\a_1^{\vee},\dots, \a_r^{\vee}\}$, which is to say  
\begin{align*}
\big(\alpha_i^{\vee}, \varpi_j\big) = \delta_{ij}, & & \text{ for all } i,j = 1, \dots, r.
\end{align*}
We denote by ${\mathcal P}$ the {\em integral weight lattice} of $\frak{g}$, which is to say the $\bZ$-span of the fundamental weights. Moreover, ${\mathcal P}^+$ denotes the cone of  {\em dominant integral weights}, which is to say the $\bN_0$-span of the fundamental weights. 

A $U_q(\frak{g})$-module  $V$ is irreducible if and only if it is of the form $U_q(\frak{g})z$, for $z$ a highest weight vector. Moreover, the space of highest weight elements of any irreducible module is necessarily one-dimensional. 

For each $\mu\in\mathcal{P}^+$ there exists an irreducible finite-dimensional $U_q(\frak{g})$-module  $V_\mu$, uniquely defined by the existence of a vector $v_{\mu}\in V_\mu$, which we call a {\em highest weight vector},  satisfying
\[
  E_i \triangleright v_\mu=0,\qquad K_i \triangleright v_\mu = q^{(\mu,\alpha_i)} v_\mu, \qquad
  \text{for all $i=1,\ldots,r$.}
\]
Moreover, $v_{\mu}$ is unique up to scalar multiple. We call any finite direct sum of such $U_q(\frak{g})$-representations a {\em type-$1$ representation}. In general, a vector $v\in V_\mu$ is called a \emph{weight vector} of \emph{weight}~$\mathrm{wt}(v) \in \mathcal{P}$ if
\begin{align}\label{eq:Kweight}
K_i \triangleright v = q^{(\mathrm{wt}(v), \alpha_i)} v, & & \textrm{ for all } i=1,\ldots,r.
\end{align}
Each type $1$ module $V_{\mu}$ decomposes into a direct sum of \emph{weight spaces}, which is to say, those subspaces of $V_{\mu}$ spanned by weight vectors of any given weight.
 For any such highest  weight vector $v$, we find it convenient to denote 
\begin{align*}
\mathrm{wt}(v) := \mu,   \text{~~~ and ~~~~~}  \mathrm{wt}_i(v) := \mu_i, & &\text{ where } ~~~ \mu = \sum_{i=1}^r \mu_i \varpi_i.
\end{align*}
We denote by $_{U_q(\frak{g})}\mathbf{type}_1$ the full subcategory of ${U_q(\frak{g})}$-modules whose objects are finite  sums of type-1 modules $V_\mu$, for $\mu \in \mathcal{P}^+$.  It is clear that $\,_{U_q(\frak{g})}\mathbf{type}_1$ is abelian, semisimple, and  equivalent to the category of  finite-dimensional representations of $\frak{g}$. Moreover, the Weyl character formula remains unchanged under $q$-deformation, which is to say, for any $\mu \in \mathcal{P}^+$, the dimensions of the weight spaces of the $U_q(\frak{g})$-module $V_\mu$ have the same as the dimensions as for the corresponding classical $\frak{g}$-module. Finally,  we denote by $_{U_q(\frak{g})}\mathbf{LF}_1$ the full subcategory of ${U_q(\frak{g})}$-modules whose objects are (not necessarily finite) sums of the type-1 modules. (Note that $\textbf{LF}$ stands for  {\em locally finite}.)

\subsection{Quantised Coordinate Algebras $\O_q(G)$} \label{subsection:Oq(G)}

Let $V$ be a finite-dimensional $U_q(\frak{g})$-module, $v \in V$, and $f \in V^*$, the linear dual of $V$. Consider the function 
\begin{align*}
c^V_{f,v}:U_q(\frak{g}) \to \bC, & & X \mapsto f\big(X(v)\big).
\end{align*}
The {\em coordinate ring} of $V$ is the subspace
\begin{align*}
C(V) := \text{span}_{\mathbb{C}}\{ c^V_{f,v} \,| \, v \in V, \, f \in V^*\} \sseq U_q(\frak{g})^*.
\end{align*}
In fact, we see that $C(V) \sseq U_q(\frak{g})^\circ$, and that a Hopf subalgebra of $U_q(\frak{g})^\circ$ is given by 
\begin{align*}
\O_q(G) := \bigoplus_{V \,  \in \, _{U_1(\frak{g})}\mathbf{type}_1} C(V).
\end{align*}
We call $\O_q(G)$ the {\em type-$1$ Hopf dual} of $U_q(\frak{g})$, or alternatively the {\em quantum coordinate algebra of $G$}, where $G$ is the unique connected, simply connected, complex algebraic group having $\frak{g}$ as its complex Lie algebra.

By construction $\O_q(G)$ is cosemisimple. Moreover, $\O_q(G)$  is a domain (for a proof see   \cite[Theorem I.8.9]{GoodBrown}). Dualising the compact real form of $U_q(\frak{g})$  gives a Hopf $*$-\alg structure on $\O_q(G)$, \wrt which it is a CQGA. Since every finite-dimensional irreducible representation of $U_q(\frak{g})$ is contained in some tensor
 product of fundamental representations, $\O_q(G)$  is finitely generated, and  hence  a CQMGA.

The evaluation pairing $\O_q(G) \times U_q(\frak{g}) \to \bC$ is by constructions a dual pairing of Hopf algebras. In particular it gives us a dual pairing of Hopf $*$-algebras, which is to say
\begin{align*}
\la X^*, f\ra = \ol{\la X, S(f)^*\ra}, & & \la X,f^*\ra = \ol{\la S(X)^*, f \ra}, & & \text{ for all } X \in U_q(\frak{g}), f \in \O_q(G).
\end{align*}
For any left $\O_q(G)$-comodule algebra $(V,\DEL_L)$, we can define a left $U_q(\frak{g})$-module structure on $V$ according to 
\bal \label{eqn:ComoduleToModule}
U_q(\frak{g}) \times V \to V, & & (X,v) \mapsto   \la S(X),v_{(-1)}\ra v_{(0)}.
\eal
This gives us the equivalence of categories 
\begin{align*}
^{\O_q(G)} \mathrm{Mod} \to   \, \,_{U_q(\frak{g})}\mathbf{LF}_1, && (V,\DEL_L) \mapsto (V,\tr).
\end{align*}
We will use this equivalence throughout the paper, tacitly identifying $\O_q(G)$-comodules
and $U_q(\frak{g})$-modules of  type-$1$. For any  left $\O_q(G)$-comodule algebra  $P$, a  useful result is that, for $X \in U_q(\frak{g})$, and $a,b \in P$, we have
\begin{align} \label{eqn:dualpairingproductflip}
X \tr (ab) = \big(X_{(2)} \tr a\big) \big(X_{(1)}\tr b\big).
\end{align}

\subsection{The Hopf Algebra $\O_q(SU_n)$} \label{subsection:app:SU} \label{subsection:SUN}

In this subsection, we recall the well-known  FRT-presentation of $\O_q(SU_{n+1})$, see \cite{FRT} or \cite[\textsection9]{KSLeabh} for further details. 

For $q \in \bR\bs\{-1,0\}$, let $\O_q(\text{M}_n)$ be the unital complex \alg generated by the elements  $u^i_j$, for $i,j = 1, \ldots, n$ satisfying the relations
\begin{align*}
u^i_ku^j_k  = qu^j_ku^i_k, &  & u^k_iu^k_j = qu^k_ju^k_i,                    & &   \; 1 \leq i<j \leq n; 1\leq k \leq n, \\
  u^i_lu^j_k = u^j_ku^i_l, ~~ &  & u^i_ku^j_l = u^j_lu^i_k + (q-q^{-1})u^i_lu^j_k, & &   \;  1 \leq i<j \leq n; 1 \leq k < l \leq n.
\end{align*}
A bi\alg structure on $\O_q(M_n)$, with coproduct $\DEL$ and counit $\e$, is uniquely determined by $\DEL(u^i_j) :=  \sum_{k=1}^n u^i_k \oby u^k_j$ and $\e(u^i_j) := \d_{ij}$. Let  $\dt_n$,  the {\em quantum determinant},  denote the element
\begin{align*}
\dt_{n} := \sum\nolimits_{\s \in S_n}(-q)^{\ell(\s)}u^1_{\s(1)}u^2_{\s(2)} \cdots u^n_{\s(n)} = \sum\nolimits_{\s \in S_n}(-q)^{\ell(\s)}u_1^{\s(1)}u_2^{\s(2)} \cdots u_n^{\s(n)},
\end{align*}
with summation taken over all permutations $\s$ of the set $\{1, \ldots, n\}$, and $\ell(\s)$ the number of inversions in $\s$. As is well known \cite[\textsection 9.2.2]{KSLeabh}, $\dt_n$ is a central and grouplike element of $\O_q(M_n)$. Consider next the quotient algebra $\O_q(M_n)/\la \dt_{n} - 1 \ra$, where $\la \dt_n - 1 \ra$ denotes the ideal generated by $\dt_n - 1$. Note that the maps $\DEL$ and $\e$ descend to a well-defined bialgebra structure on the quotient algebra, which in addition can be endowed with a Hopf \alg structure by defining
\begin{align} \label{eqn:antipodeformula}
S(u^i_j) := & (-q)^{i-j}\sum\nolimits_{\s \in S_{n-1}}(-q)^{\ell(\s)}u^{\s(k_1)}_{l_1}u^{\s(k_2)}_{l_2} \cdots u^{\s(k_{n-1})}_{l_{n-1}} \\
= & (-q)^{i-j}\sum\nolimits_{\s \in S_{n-1}}(-q)^{\ell(\s)}u^{k_1}_{\s(l_1)}u^{k_2}_{\s(l_2)} \cdots u^{k_{n-1}}_{\s(l_{n-1})},
\end{align}
where $\{k_1, \ldots ,k_{n-1}\} := \{1, \ldots, n\}\bs \{j\}$, and $\{l_1, \ldots ,l_{n-1}\} := \{1, \ldots, n\}\bs \{i\}$ as ordered sets. Finally, a Hopf $*$-\alg structure can be defined by setting  $(u^i_j)^* :=  S(u^j_i)$. We denote this Hopf $*$-\alg by $\O_q(SU_n)$, and call it the {\em quantum special unitary group of degree $n$}.

We now present a Hopf \alg isomorphism between $\O_q(SU_n)$ and the type-1 Hopf dual of $U_q(\frak{sl}_n)$.  A non-degenerate dual pairing of Hopf algebras between $\O_q(SU_n)$ and $U_q(\frak{sl}_n)$ is uniquely determined by 
\bal \label{eqn:dualpairing}
\la K_i,u^j_j\ra = q^{ \d_{i,j-1} - \d_{ij}}, & & 
\la E_i,u^{i+1}_i\ra = 1, & & \la F_i,u^i_{i+1}\ra = 1,
\eal
with all other pairings of generators being zero. This determines a Hopf \alg embedding of $\O_q(SU_n)$ into  $U_q(\frak{sl}_n)^\circ$, the  image of which is precisely the quantum coordinate algebra of $U_q(\frak{sl}_n)$. 

\section{Decomposing $\Om^{(0,\bullet)}$  into Irreducible $U_q(\frak{sl}_n)$-Modules} \label{subsection:Branching}

In this appendix, we use Frobenius reciprocity to decompose $\Om^{(0,\bullet)}$ into irreducible left $U_q(\frak{sl}_{n})$-submodules.  As a direct application,  we prove that $\O_q(\mathbb{CP}^{n-1})$, endowed with its Heckenberger--Kolb calculus, is of Gelfand type and  self-conjugate. 

Presentations of the irreducible modules occurring in the anti-holomorphic forms have previously  appeared in both the classical \cite[Proposition 5.2]{OsakaCPnSpec} and  quantum group literature \cite[Proposition 5.5]{SISSACPn}. We re-establish the presentation here so as to completely guarantee consistency of conventions, and to give a self-contained exposition  of how such presentations are obtained for the benefit of non-experts.

\subsection{Skew Young Diagrams and Young Tableaux} \label{subsection:app:Young}

A {\em Young diagram}  is a finite collection of boxes arranged in left-justified rows, with the row lengths in non-increasing order. We see that Young diagrams with $p$  rows are  equivalent to  {\em integer partitions} of {\em order} $p$, which is to say  $p$-tuples 
\begin{align*}
\mu = (\mu_1, \ldots, \mu_p) \in \bN^{p}_0, & & \text{ \st ~} \mu_1 \geq \cdots \geq \mu_p.
\end{align*}
We denote by $F^\mu$ the Young diagram corresponding to an integer partition $\mu$. For a partition $\mu = (\mu_1, \dots, \mu_p)$, its {\em conjugate} $\mu' = (\mu'_1, \dots, \mu'_{\mu_1})$ is the unique partition of order $\mu_1$ such that  $\mu_s'$ is equal to the number of boxes in the $s^{\text{th}}$ column of $F^{\mu}$. Note that the Young diagram  $F^{\mu'}$ can be obtained from $F^{\mu}$ by reflecting it along its northwest-southeast diagonal.

We can put a partial order on the set of all Young diagrams, or equivalently, a partial order on the set of all integer partitions, by defining $\mu \succeq \nu$ whenever $\mu_i \geq \nu_i$, for all $i = 1, \dots, p$, with the possible addition of trailing zeros. For any pair $\mu \succeq \nu$, the associated {\em skew Young diagram} $F^{\mu\bs \nu}$ is given by removing from  $F^\mu$  all boxes belonging to the obvious superimposition of $F^{\nu}$ on $F^{\mu}$. The {\em weight} of a skew  Young diagram is the number of boxes in the diagram, which is to say,
\begin{align*}
|\mu\bs \nu| := \sum_{i=1}^p \mu_i - \nu_i \in \bN_0.
\end{align*}

In what follows, we find the following alternative presentation of partitions useful. Denote by $e_i$, for $i = 1, \dots, p$, the standard generators of the monoid  $\bN^p_0$. The {\em fundamental} partitions are given by
\begin{align*}
\varpi_k := e_1 + e_2 + \cdots + e_k, & & \text{ for } k= 1, \dots,  p.
\end{align*}
Note that any partition $\mu = (\mu_1,\dots, \mu_p)$ can be expressed as the sum of fundamental partitions: 
\begin{align*}
\mu = (\mu_1 - \mu_2) \varpi_1 + (\mu_2 - \mu_3) \varpi_2 + \cdots + (\mu_{p-1} - \mu_p)\varpi_{p-1} + \mu_p \varpi_p,
\end{align*}

A {\em horizontal strip} is a skew Young diagram where every column contains at most one box. Alternatively, horizontal strips can be presented in terms of pairs of integer partitions  $\mu \succeq \nu$ \stn, for $p$ the order of $\mu$,
\bal \label{ineql:horizontalstrip}
\mu_1 \geq \nu_1 \geq \mu_2 \geq  \cdots \geq \mu_p \geq \nu_p.
\eal
We denote by $\text{HSC}(\mu)$ the set of all partitions $\mu \succeq \nu$ such that $F^{\mu\bs\nu}$ is a horizontal strip.

As suggested by the notation, we have an obvious  bijection between the irreducible representations of the quantised enveloping algebra $U_q(\frak{sl}_n)$, and Young diagrams with $r = \text{rank}(\frak{g})$ rows. Explicitly, let $V_\mu$ be an irreducible $U_q(\frak{g})$-module, with defining dominant integral weight $\mu \in {\mathcal P}^+$. Expressing $\mu$ in terms of the fundamental weights as $\mu = \sum_{i=1}^r a_i \varpi_i$, we have the corresponding partition $\mu = \sum_{i=1}^r a_i \varpi_i$, and hence the corresponding Young diagram $F^\mu$.

\subsection{The Decomposition}


Applying Frobenius reciprocity requires an understanding of the branching rules for the inclusion $U_q(\frak{l}_{n-1}) \hookrightarrow U_q(\frak{sl}_{n})$. However, as explained in greater detail in \textsection \ref{section:QHSPs}, the fact that Weyl's character formula remains unchanged under $q$-deformation implies that the branching laws are equivalent to the classical case.  The branching laws for the inclusion $\frak{l}_{n-1} \hookrightarrow \frak{sl}_{n}$ admit a well-known   formulation originally due to Weyl \cite{Weyl} (see also \cite[Proposition 5.1]{OsakaCPnSpec}) which allows us to immediately write down the corresponding $q$-deformed branching laws.

\begin{lem} \label{lem:thelnm1branching}
For any partition $\mu$, with corresponding irreducible $U_q(\mathfrak{sl}_{n})$-module $V_{\mu}$, an isomorphism of $U_q(\frak{l}_{n-1})$-modules is given by 
\bal \label{eqn:theBranching}
V_{\mu} \simeq \bigoplus_{\nu \in \text{\em HSC}(\mu)} V_{\ol{\nu}}\big(\nu_{n-1} - |\mu\bs\nu|\big), & & \text{ where ~ } \nu = \sum_{i=1}^{n-1} \nu_i \varpi_i, \text{  and \, } \ol{\nu}  := \nu - \nu_{n-1} \varpi_{n-1},
\eal
for $\varpi_i$ the fundamental partitions, and summation is over all partitions $\nu \in \text{\em HSC}(\mu)$, where $\text{\em HSC}(\mu)$ is the set of partitions complementary to the horizontal  strips of $\mu$ (as defined in Appendix \ref{subsection:app:Young}). In particular, the decomposition is multiplicity-free for each $V_{\mu}$.
\end{lem}

In preparation for the application of Frobenius reciprocity below, we now  consider two specific applications of the branching rules for the case of $U_q(\frak{sl}_3)$. The presentation here is given in terms of Young diagrams, with the intention of providing the reader with some visual intuition.

\begin{eg}

We begin with an explicit example of the branching procedure for the inclusion $U_q(\frak{l}_2) \hookrightarrow U_q(\frak{sl}_3)$, corresponding to the quantum homogeneous space $\O_q(\mathbb{CP}^2)$.  Consider the partition $\mu := (3,2)$ with corresponding Young diagram:  
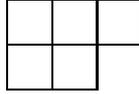
\begin{center}
\begin{tabular}{c}
 \begin{ytableau}
*(white) & *(white)  &*(white)  \\
*(white)  & *(white)  \\
\end{ytableau}
\end{tabular}
\end{center}
We present the $6$ possible partitions $\nu \in \text{HSC}(\mu)$ in Young diagram form, highlighting those boxes which form horizontal strips:

\bigskip

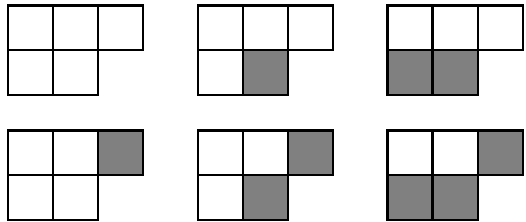
\begin{center}
\begin{tabular}{ccccc}

\begin{ytableau}
*(white) & *(white)  &*(white)  \\
*(white)  & *(white)  \\
\end{ytableau}

&  &

\begin{ytableau}
*(white) & *(white)  &*(white)  \\
*(white)  & *(gray)  \\
\end{ytableau}

& &

\begin{ytableau}
*(white) & *(white)  &*(white)  \\
*(gray)  & *(gray)  \\
\end{ytableau}

\\

& & \\

\begin{ytableau}
*(white) & *(white)  &*(gray)  \\
*(white)  & *(white)  \\
\end{ytableau}

& &

\begin{ytableau}
*(white) & *(white)  &*(gray)  \\
*(white)  & *(gray)  \\
\end{ytableau}

&  & 

\begin{ytableau}
*(white) & *(white)  &*(gray)  \\
*(gray)  & *(gray)  \\
\end{ytableau} 

\end{tabular}
\end{center}


\bigskip

Removing the highlighted boxes, we arrive at the set of Young diagrams corresponding to the partitions in $\text{HSC}(\mu)$:


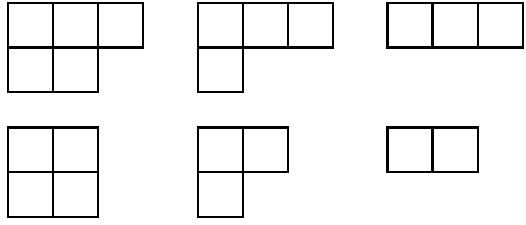
\begin{center}
\begin{tabular}{lclclc}

\begin{ytableau}
*(white) & *(white)  &*(white)  \\
*(white)  & *(white)  \\
\end{ytableau}  

&  & 

\begin{ytableau}
*(white) & *(white)  &*(white)  \\
*(white)   \\
\end{ytableau}  

&   & 

\begin{ytableau}
*(white) & *(white)  &*(white)  \\
\end{ytableau}

\\

 &  & \\

\begin{ytableau}
*(white) & *(white)   \\
*(white)  & *(white)  \\
\end{ytableau} 

&   & 

\begin{ytableau}
*(white) & *(white) \\
*(white)   \\
\end{ytableau} 

&   &

\begin{ytableau}
*(white) & *(white)  \\
\end{ytableau}

\end{tabular}
\end{center}

Next we present the Young diagrams of the partitions $\ol{\nu} = \nu - \nu_{2}\varpi_{2}$. We do this in two steps, first highlighting the boxes to be removed from $F^\nu$ in order to arrive at $F^{\ol{\nu}}$:

\bigskip

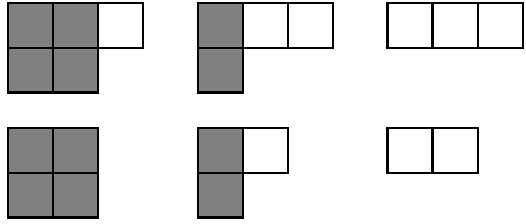
\begin{center}
\begin{tabular}{lclclc}

\begin{ytableau}
*(gray) & *(gray)  &*(white)  \\
*(gray)  & *(gray)  \\
\end{ytableau}  

&  & 

\begin{ytableau}
*(gray) & *(white)  &*(white)  \\
*(gray)   \\
\end{ytableau}  

&   & 

\begin{ytableau}
*(white) & *(white)  &*(white)  \\
\end{ytableau}

\\

 &  & \\

\begin{ytableau}
*(gray) & *(gray)   \\
*(gray)  & *(gray)  \\
\end{ytableau} 

&   & 

\begin{ytableau}
*(gray) & *(white) \\
*(gray)   \\
\end{ytableau} 

&   &

\begin{ytableau}
*(white) & *(white)  \\
\end{ytableau}

\end{tabular}
\end{center}


Removing the highlighted  columns we arrive at the following list of Young diagrams, where $\varnothing$ denotes the empty Young diagram:

\bigskip

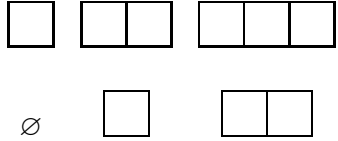
\begin{center}
\begin{tabular}{cccccc}

\begin{ytableau}
*(white)  \\
\end{ytableau} 

& 


\begin{ytableau}
 *(white)  &*(white)  \\
\end{ytableau}  

& 

\begin{ytableau}
*(white) & *(white)  &*(white)  \\
\end{ytableau} 

\\

& & \\ 

$\varnothing$

&  

\begin{ytableau}
 *(white) \\
\end{ytableau} 

&

\begin{ytableau}
*(white) & *(white)  \\
\end{ytableau}

\end{tabular}
\end{center}

Thus the decomposition of $V_\mu$ into irreducible $U_q(\frak{sl}_{2})$-modules is given by
\bal \label{eqn:egbranchsl}
V_{\mu} \simeq  & \, V_{\varpi_1} \oplus V_{2\varpi_1} \oplus V_{3\varpi_1} \oplus \bC \oplus V_{\varpi_1} \oplus V_{\varpi_2}. 
\eal
Finally, subtracting the number of boxes removed in the first step, from the number of columns removed in the second step, gives us the weight of $K_{2}$. Explicitly,  the decomposition of $V_{\mu}$ into irreducible $U_q(\frak{l}_2)$-modules is given by
\bal \label{eqn:egbranchingl}
V_{\mu} \simeq V_{\varpi_1}(2) \oplus V_{2\varpi_1}(0) \oplus V_{3\varpi_1}(-2) 
 \oplus \bC(1) \oplus  V_{\varpi_1}(-1) \oplus V_{2\varpi_1}(-3). 
\eal
Note that while the decomposition of $V_{\mu}$ into $U_q(\frak{sl}_2)$-submodules contains multiplicities, the decomposition into $U_q(\frak{l}_2)$-modules is multiplicity-free.

\end{eg}

\begin{eg}
As we saw in the previous example, the Young diagram presentation of the branching process can be understood as a combination of two steps: 
\bet
\item Remove from a Young diagram $F^\mu$ a horizontal strip $F^{\mu\bs\nu}$.

\item Remove all columns of height $n-1$. 

\eet
Thus finding all possible $U_q(\frak{sl}_{n})$-modules whose $U_q(\frak{sl}_{n-1})$-branching contains a given module $V$  amounts to finding all possible Young diagrams which can be operated  on by steps 1 and 2 to produce the Young diagram corresponding to $V$.

Let us apply this process to a concrete example corresponding to the case of $\O_q(\mathbb{CP}^3)$: For the $U_q(\frak{sl}_3)$-module $V_{\varpi_1}$ we will  find all possible $U_q(\frak{sl}_{4})$-modules whose branching contains $V_{\varpi_1}$ as a decomposition.  Recall first that $V_{\varpi_1}$ has the corresponding Young diagram:

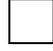
\begin{center}
\begin{tabular}{ccccc}

\begin{ytableau}
*(white)  \\
\end{ytableau}

\end{tabular}
\end{center}
We reverse step 2 by adding an arbitrary number $k \in \bN_0$ of columns of height $3$ to obtain the Young diagram:
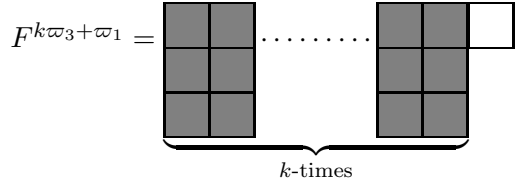
\begin{center}
\begin{tabular}{ccccc}
$F^{k \varpi_3 + \varpi_1} = \underbrace{
\begin{ytableau}
  *(gray) & *(gray)  \\
    *(gray)   & *(gray)  \\
        *(gray)   & *(gray)  
\end{ytableau}
{\,\cdots \cdots  \cdots \,}
\begin{ytableau}
*(gray)  & *(gray)  \\
*(gray)  & *(gray)  \\
*(gray)  & *(gray)
\end{ytableau}
}_{k\text{-times}}\!\!$
\begin{ytableau}
*(white)\\
\end{ytableau}
\end{tabular}
\end{center}
To reverse step 1, we must find all possible Young diagrams $F^{\mu}$ such that 
\begin{align*}
k \varpi_3 + \varpi_1 \in  \text{\em HSC}(\mu).
\end{align*}
In fact, we see that there exist exactly two such families  of Young diagrams.  The first, for a general $l \in \bN_0$, is given by

~~~~~\begin{tabular}{cccccc}
$F^{k\varpi_3 +  (l+1) \varpi_1} :=
\underbrace
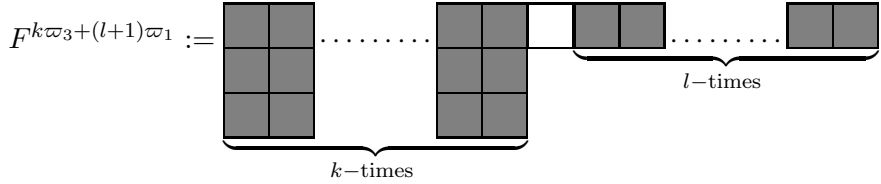
{\begin{ytableau}
  *(gray) & *(gray)  \\
    *(gray)   & *(gray)  \\
        *(gray)   & *(gray)  
\end{ytableau}
{\, \cdots \cdots  \cdots\,}
\begin{ytableau}
  *(gray) & *(gray) \\
    *(gray)   & *(gray)   \\
        *(gray)   & *(gray)  
\end{ytableau}
}_{k-\text{times }} \!\!$
\begin{ytableau}
*(white) \\
\end{ytableau}
$\!\!\underbrace{\begin{ytableau}
  *(gray) & *(gray)  \\
  \end{ytableau}
{\, \cdots \cdots  \cdots\,}
\begin{ytableau}
  *(gray) & *(gray) \\
\end{ytableau}
}_{l-\text{times }}$
\end{tabular}

The second, for a general $l \in \bN_0$, is given by

~~~~ \begin{tabular}{cccccc}
$F^{k\varpi_3 + \varpi_2 + l \varpi_1} :=
\underbrace
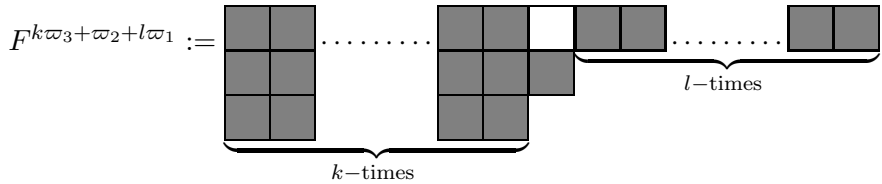
{\begin{ytableau}
  *(gray) & *(gray)  \\
    *(gray)   & *(gray)  \\
        *(gray)   & *(gray)  
\end{ytableau}
{\, \cdots \cdots  \cdots\,}
\begin{ytableau}
  *(gray) & *(gray) \\
    *(gray)   & *(gray)   \\
        *(gray)   & *(gray)  
\end{ytableau}
}_{k-\text{times }} \!\!$
\begin{ytableau}
*(white) \\
*(gray)  \\
\end{ytableau}
$\!\!\underbrace{\begin{ytableau}
  *(gray) & *(gray)  \\
  \end{ytableau}
{\, \cdots \cdots  \cdots\,}
\begin{ytableau}
  *(gray) & *(gray) \\
\end{ytableau}
}_{l-\text{times }}$
\end{tabular}

To perform this process for $U_q(\frak{l}_3)$-branching, we need to take care of the weight of $K_3$. Let us assume that we want to branch to the representation $V_{\varpi_1}(c)$, for some $c \in \bZ$. As we saw in the previous example, the weight of $K_3$ is precisely the number of boxes removed in step $1$ minus the  number of columns removed in step $2$. Thus we see that the branched representation $\varpi_3  + (l+1) \varpi_1$ contains $V_{\varpi_1}(c)$ if and only if $l-k = c$, while the  branched representation $\varpi_3  + \varpi_2 + l \varpi_1$ contains $V_{\varpi_1}(c)$ if and only if $l-k+1 = c$.
\end{eg}


Returning now to the general case of $\O_q(\ccpn)$. We would now like a complete description of the  $U_q(\frak{sl}_{n})$-modules appearing in the irreducible decomposition of $\Om^{(0,k)}$. First, however, we need to identify the inducing representations $V^{(0,k)}$ as a  $U_q(\frak{l}_{n-1})$-modules. (Note that we present the trivial case of $V^{(0,0)}$ to highlight the fact that its $K_{n-1}$-weight does not follow the general pattern for the higher forms.)

\begin{cor} \label{cor:antiholoaslmods}
We have the following isomorphisms of left $U_q(\frak{l}_{n-1})$-modules
\begin{align*}
V^{(0,0)} \simeq \mathbb{C}(0), & &  V^{(0,k)} \simeq V_{\varpi_{n-k-1}}(-k-1), ~~~ \text{ for } k=1, \dots, n-1.
 \end{align*}
\end{cor}
\begin{proof}
Since the isomorphism $V^{(0,0)} \simeq \mathbb{C}(0)$ is obvious, we can move directly onto the higher forms. Recalling that $e^-_i \wed e^-_i = 0$, for all $i$, we see from Lemma \ref{lem:EKFBasis} that
\begin{align*}
E_{j} \tr \big(e^-_{n-1} \wed \cdots \wed e^-_{n-k}\big) = 0, & & \text{ for } j = 1, \dots, n-2.
\end{align*}
Another application of Lemma \ref{lem:EKFBasis}  confirms that 
\begin{align*}
K_j \tr \left(e^-_{n-1} \wed \cdots \wed e^-_{n-k}\right) = q^{\d_{j,n-k-1}} e^-_{n-1} \wed \cdots \wed e^-_{n-k}, & & \text{ for } j = 1, \dots, n-2.
\end{align*}
Hence $e^-_{n-1} \wed \cdots \wed e^-_{n-k}$  is a $U_q(\frak{sl}_{n-1})$-highest weight vector of  weight $\varpi_{n-k-1}$.

Recalling that  the dimension of  universal enveloping algebra modules is unchanged under  $q$-defor-mation, we see that 
\begin{align*}
\dim_{\mathbb{C}}\Big(U_q(\frak{sl}_{n-1})\,e^-_{n-1} \wed \cdots \wed e^-_{n-k}\Big) =  \binom{n-1}{k}.
\end{align*}
Recalling from Theorem \ref{thm:HKBasisRels} that  $\Phi\big(\Om^{(0,k)}\big)$ is also an $\binom{n-1}{k}$-dimensional space, we see that
\begin{align*}
U_q(\frak{sl}_{n-1})\,e^-_{n-1} \wed \cdots \wed e^-_{n-k} = V^{(0,k)}.
\end{align*}
Thus $V^{(0,k)}$ is isomorphic as a $U_q(\frak{sl}_{n-1})$-module, to $V_{\varpi_{n-k-1}}$. Finally,  Lemma \ref{lem:EKFBasis} tells us that
\begin{align*}
K_{n-1} \tr \left(e^-_{n-1} \wed \cdots \wed e^-_{n-k}\right) =  q^{-k-1}  e^-_{n-1} \wed \cdots \wed e^-_{n-k}.
\end{align*}
Hence $V^{(0,k)}$ is isomorphic, as a $U_q(\frak{l}_{n-1})$-module, to $V_{\varpi_{n-k-1}}(-k-1)$, as claimed. 
\end{proof}

\begin{lem} \label{lem:DecompOfOm0kINTOIrreps} The complex structure $\Om^{(\bullet,\bullet)}$ is of Gelfand type.  Moreover, an irreducible $U_q(\frak{sl}_n)$-module appears  in the decomposition of $\Omega^{(0,k)}$ into irreducibles only if its highest weight is of the form
\begin{enumerate}
\item $0$   \text{ ~~~~~~~~~~~~~~~~~~~~~~~~~~~~~~~~~~~~~~~~~~~~~~~~~~~~~~~~~~~~~~~~~~~~\, when} $k =0$,
\item $(l+k+1) \varpi_1  + \varpi_{n-k-1} + l \varpi_{n-1}$,  \text{ ~~~~~~~~~~\, for } $l \in \bN_0$, ~~ when  $k = 0, \dots, n-2$,
\item $ (l + k) \varpi_1  + \varpi_{n-k} + l \varpi_{n-1}$,   \text{ ~~~~~~~~~~~~~~~~\,\,~ for } $l \in \bN_0$, ~~ when $k=2, \dots, n-1$.
\end{enumerate}
\end{lem}
\begin{proof} ~
Lemma \ref{lem:thelnm1branching}  tells us that the decomposition of any $U_q(\frak{sl}_n)$-module into $U_q(\frak{l}_{n-1})$-submodules is multiplicity-free. Frobenius reciprocity, as presented in Proposition \ref{prop:Frobenius}, now implies that $\Om^{(\bullet,\bullet)}$ is of Gelfand type.

Let us now assume that $k=1, \dots, n-2$. By Corollary \ref{cor:antiholoaslmods} above,  $V^{(0,k)}$  is isomorphic to $V_{\varpi_{n-k-1}}(-k-1)$ as a $U_q(\frak{l}_{n-1})$-module. Lemma \ref{lem:thelnm1branching}  tells us that $V_{\varpi_{n-k-1}}(-k-1)$ appears as a $U_q(\frak{l}_{n-1})$-submodule, of some $U_q(\frak{sl}_{n})$-module $V_{\mu}$, if and only if there exists a $\nu = \sum_{i=1}^{n-1} \nu_i \varpi_i$ in $\text{HSC}(\mu)$ such that $\ol{\nu} = \varpi_{n-k-1}$ and $\nu_{n-1} - |\mu\bs \nu| = -k-1$. Note first that a partition $\nu$ satisfies $\ol{\nu} = \varpi_{n-k-1}$ if and only if 
\begin{align} \label{qen:step1mujoining}
\nu = l \varpi_{n-1} + \varpi_{n-k-1}, & & \text{ for some } l \in \bN_0.
\end{align}
Next, recall that $\nu \in \text{HSC}(\mu)$, for some partition $\mu$, if and only if  (\ref{ineql:horizontalstrip}) is satisfied. Thus any $\mu$ must be of the form
\bal \label{eqn:slnHWdecomposition1}
l \varpi_{n-1} + \varpi_{n-k-1} + a \varpi_1, & & \text{ or } & & l \varpi_{n-1} + \varpi_{n-k} + a\varpi_1,  & \text{~~~ for some } a \in \bN_0. 
\eal
Next we see that the identity $\nu_{n-1} - |\mu\bs \nu| = -k-1$ is satisfied in the first case if and only if $l-a = -k-1$, and in the second case if $l - (a + 1) = -k-1$. Thus $V^{(0,k)}$ appears as a summand in the decomposition of $V_\mu$ into irreducible $U_q(\frak{l}_{n-1})$-modules if and only if 
\bal \label{eqn:HWdecomposition}
\mu = (l+k+1) \varpi_1 + \varpi_{n-k-1} + l \varpi_{n-1}, &  \text{ ~~~~ or ~~~~}  \mu = (l + k) \varpi_1  + \varpi_{n-k} + l \varpi_{n-1}. 
\eal
Frobenius reciprocity, for the inclusion $U_q(\frak{l}_{n-1}) \hookrightarrow U_q(\frak{sl}_n)$, now tells us that a left $U_q(\frak{sl}_{n})$-module appears as a submodule of $\Om^{(0,k)}$ if and only if its highest weight is of the form (\ref{eqn:HWdecomposition}), as claimed. The proofs for the cases $k=0$, and $k = n-1$, are  analogous, and so, we omit them. 
\end{proof}


\begin{cor} \label{cor:selfconjCPN}
Quantum projective space $\O_q(\ccpn)$  is self-conjugate.
\end{cor}
\begin{proof}
From Lemma \ref{lem:DecompOfOm0kINTOIrreps} above we see that the irreducible submodules of $\Om^{(0,0)}$ are distinct in dimension. Since the dimension of an irreducible $U_q(\frak{sl}_n)$-module $V$ is clearly equal to the dimension of its image under the $*$-map, $\O_q(\ccpn)$ must be self-conjugate.
\end{proof}

\bibliographystyle{siam}

\end{document}